\documentclass[preprint,11pt]{elsarticle}

\usepackage{amssymb}
\usepackage{amsthm}
\usepackage{color}
\usepackage{amsmath}

\usepackage{bbm} 
\newtheorem{theorem}{Theorem}
\newtheorem*{theorem*}{Theorem}
\newtheorem*{definition*}{Definition}
\newtheorem*{lemma*}{Lemma}
\newtheorem*{example*}{Example}
\newtheorem*{exercise*}{Exercise}
\newtheorem*{remark*}{Remark}
\newtheorem{corollary}[theorem]{Corollary}
\newtheorem{condition}{}

\newtheorem{lemma}[theorem]{Lemma}

\newtheorem{proposition}[theorem]{Proposition}

\newtheorem{remark}[theorem]{Remark}

\newtheorem{example}[theorem]{Example}
\usepackage[top=3cm, bottom=3cm, left=3cm, right=3cm]{geometry}
\numberwithin{theorem}{section}
\numberwithin{figure}{section}
\numberwithin{equation}{section}
\usepackage{mathtools}

\mathtoolsset{showonlyrefs}
\usepackage[colorlinks=true]{hyperref}

\journal{\quad}

\begin{document}

\begin{frontmatter}



\title{Support, absolute continuity and harmonic moments of \\ 
	fixed points of  
	the multivariate smoothing transform}


\author[label1]
{Jianzhang Mei}
 \address[label1]  
 {Department of Mathematical Sciences, Tsinghua University, Beijing, 100084, China}
\ead{mjz20@mails.tsinghua.edu.cn}	

\author[label2]{Quansheng Liu}
 \address[label2]  {Univ Bretagne Sud, CNRS UMR 6205, LMBA, F-56000, 
	Vannes, France}
	\ead{quansheng.liu@univ-ubs.fr}

%
%

\begin{abstract}
Consider the  multivariate smoothing transform fixed-point equation:  $\eta =$  law of 
$ \sum_{i=1}^N A_i Z_i$, where $N \geq 0$ is a random integer, $(A_i)_{i \geq 1}$ are $d \times d$ random nonnegative matrices,  $(Z_i)_{i \geq 1}$ is a sequence of  $\mathbb{R}_+^d$-valued random variables    independent of  $(N, A_1, A_2, \cdots)$, and all $Z_i$ have the same law $\eta$. For each fixed point $\eta$, under suitable conditions, we describe its support, establish  its  absolute continuity,  and prove the existence of its harmonic moments.
\end{abstract}
 



%


\end{frontmatter}


	\tableofcontents	

\section{Introduction}
\subsection{
	Background and objective}
Let $d \geq 1$ be an integer and  $(A_1,A_2,\cdots)$  a sequence of random $d \times d$  nonnegative matrices (whose   entries are all nonnegative). Assume that 
\begin{equation}
	N := \# \{ i \geq 1 : A_i \ne 0 \} < \infty  \quad  \mbox{ almost surely (a.s.).} 
\end{equation}  
Without loss of generality, we assume that 
$A_i \neq 0$ if $1 \leq i \leq N$ and $A_i=0$ if $i>N$. 
As usual we denote by  $\mathbb{R}^d_+ $ the set of column vectors $x \geq 0$,  where  $x \geq 0 $ means that all its components  are nonnegative. 
We consider the  fixed-point equation   of the multivariate smoothing transform $\mathcal{S}$ on the set of probability measures on   $\mathbb{R}^d_+ $: 
\begin{equation}\label{equ::smoothing_transform_original_form}
	\eta = \mathcal S(\eta) := \mbox{ law of }  \sum_{i=1}^N A_i Z_i,
\end{equation}
where $(Z_i)_{i\geq 1}$ is a sequence of independent and identically distributed (i.i.d.)  random variables with law $\eta$, independent of $(N, A_1, A_2, \cdots)$. 
In terms of random variables, the fixed-point equation 
also reads 
\begin{equation}\label{equ::smoothing_transform}
	Z \overset{\mathcal{L}}{=} \sum_{i = 1}^{N} A_i Z_i,
\end{equation}
where $Z$ has the same law $\eta$ as $Z_i$,  and $\overset{\mathcal{L}}{=} $ means the equality in law. 

For the one-dimensional case, the equation~\eqref{equ::smoothing_transform} 
has been studied by many authors, 
see for example:  \citet{mandelbrot1974multiplications, kahane1976certaines} and \citet{guivarc1990extension} for the study of the canonical self-similar cascades, also called Mandelbrot cascades; \citet{biggins1977martingale}, \citet{biggins1997seneta, biggins2005fixed},  Buraczewski, Iksanov and  Mallein \cite{buraczewski2021derivative}, 
	%
	%
	{\color{black} for branching random walks}; 
	\citet{durrett1983fixed} for infinite particle systems, \citet{liu1998fixed, liu2000generalized, liu2001asymptotic}  for a general study and various applications,
	\citet{alsmeyer2013fixed}  for $\mathbb R$-valued solutions instead of $\mathbb R_+$-valued solutions,
	\citet{bassetti2012self} for kinetic models,  \citet{jelenkovic2010information} for analysis of information ranking, and {Miller, Sheffield and Werner~\cite{miller2022simple}} for conformal loop ensembles (CLE). For other related models, see e.g. 
		\citet{MillardSchweinsbergYaglomtype}, \citet{MaillardPeningtonBrw}, \citet{ChenHelowerdeviation}, Chen, Garban and Shekhar \cite{ ChenGarbanShekharfixedbroanian},  \citet{hu2009minimal}.
	%
	The basic problems that we are interested in  include the existence and uniqueness  of solutions, and properties of the solutions. 
	
	%
	
	For the multidimensional case, a good progress has been made in the study of Equation~\eqref{equ::smoothing_transform} 
	in the last decade. In particular, 
	the existence and uniqueness of fixed points were shown in Buraczewski, Damek, Guivarc'h and Mentemeier~\cite{BurDaGuiMent14}  and~\citet{Ment16}; see also \citet{meiners2017solutions} for the case where $(A_i)_{i\geq 1}$ are complex numbers  or similarity matrices. In  \cite{BurDaGuiMent14}  and  \cite{Ment16}, it is also shown that 
	the fixed points have heavy tail under suitable assumptions; see also 
	Buraczewski, Damek, Mentemeier and Mirek~\cite{buraczewski2013heavy} for the case where $(A_i)_{i\geq 1}$ are invertible matrices. The multivariate version of Equation~\eqref{equ::smoothing_transform} has various applications in different contexts:  see e.g. \citet{bassetti2014multi}, Buraczewski et al. 
		\cite{buraczewski2013heavy} for kinetic models, ~\citet{huang2016moments} for branching random walks, \citet{leckey2019densities} and \citet{neininger2006survey} for the analysis of Quicksort algorithm and other related models. 
	
	
	
	\medskip 
	The objective of this article is to  study  the following properties of fixed points: 
	\begin{enumerate}
		\item[(a)]  The support.  For the one-dimensional case, 
		it was proven in~\cite{liu2001asymptotic} that the support of a fixed point is the positive real line. For the multidimensional case,  in 
		\cite{BurDaGuiMent14}  it was 
		shown that a fixed point takes value in a set that consists of many rays starting from the origin and their linear combinations with positive coefficients. In the present work we will show that  the support is exactly this set for the case where $(A_i)_{i \geq 1}$ are i.i.d., and we will also give some inclusion relations of the support for the non i.i.d. case. 
		\item[(b)] The absolute continuity. This property has been considered in \citet{liu2001asymptotic}, \citet{ damek2018absolute}, \citet{leckey2019densities}, and {  Damek, Gantert and Kolesko \cite{damek2019absolute}}, respectively for the one-dimensional $\mathbb R_+$-valued case, the complex-valued case, the invertible matrix case,  and the limit  of the fundamental martingale of a single-type branching process in random environment. 
		In this paper, for the general nonnegative  matrix case,  we will prove the absolute continuity  of fixed points, using the knowledge of their support,  and assuming the existence of some harmonic moments of the coefficients  $(A_i)_{i \geq 1}$. 
		\item[(c)]  The existence  of harmonic moments (moments of negative order). 
		This problem  has  been studied 
		in~\citet{liu2001asymptotic} and~\citet{huang2016moments}, 	 respectively for the one-dimensional case and multidimensional case. In this work  we will  improve the results of~\cite{huang2016moments}  by finding the critical value given by the spectral property of $A_1$. 
	\end{enumerate}
	
	
	\subsection{Statements of the main results}	
	
	\subsubsection{Notation and conditions}
	We first introduce some basic notation. Given a probability measure $\nu$ on a Polish space $\Omega$,  its support is defined by 
	\begin{equation} \label{def-support}
		\mathrm{supp}(\nu) = \{ x \in \Omega : \forall \text{ neigborhood } O \text{ of }x, \nu(O) > 0 \}.
	\end{equation}
	For a random variable $X$ with  law $\nu$, we define its support as $\mathrm{supp}(X) = \mathrm{supp}(\nu) $. For a vector $x = (x_1, \cdots, x_d)^T$, denote its $L^1$ norm by $|x| = \sum_{i = 1}^d|x_i|$,  and write $\langle x, y \rangle = \sum_{i = 1}^d x_iy_i$ for the scalar product of $x$ and $y = (y_1, \cdots, y_d)^T$.    Let $\mathbb{S}^{d-1} = \{ x \in \mathbb{R}^d : |x| = 1 \}$ be the unit sphere and $\mathbb{S}_{+}^{d-1} = \mathbb{S}^{d-1} \cap \mathbb{R}^d_+$ be its nonnegative orthant.
	Let $G$ be the set of nonnegative $d\times d$ matrices.  
	For a matrix $a \in G$, we denote by $\| a \|$ its operator norm, {\color{black} and}  let $r(a)$ be its spectral radius. 
	For a matrix or vector $a$, we write $a > 0$ to mean that $a$ is strictly positive, that is,  all its entries are strictly positive. 
	
	We will use the following conditions on the point process $(N, A_1, A_2,\cdots)$.
	\begin{condition}
		\label{cond::conditions_on_N} 
			$N \geq 1$  a.s., $ \mathbb{E}[N] \in (1, +\infty)$, $A_i \in G$ and  $A_i \neq 0$ for $1\leq  i\leq N$, and $A_i =0$ for $i >N$. 
	\end{condition}
	Let $\mu$ be the probability measure on $G$ such that for all continuous bounded function $f: G \rightarrow \mathbb R$, 
	\begin{align}
		\int f(a)\mu(da) = \frac{1}{\mathbb{E}[N]} \mathbb{E} \big[ \sum_{i = 1}^N f(A_i) \big].
	\end{align}
	Let
	\begin{equation}
		\Gamma := \{ a_1\cdots a_n : a_i \in \mathrm{supp}(\mu), n \geq 0 \}
	\end{equation}	
	be the multiplicative semigroup generated by $\mathrm{supp}(\mu)$, with the convention that $a_1 \cdots a_n$ is the identity matrix if $n = 0$.  
	
	\begin{condition}\label{cond::condition_on_mu}
		The measure $\mu$ satisfies the following conditions:
		\begin{itemize}
			\item (allowability) for each matrix $a \in \Gamma$, every row and column of $a$ has at least one  strictly  positive entry; 
			\item (positivity) the semigroup  $\Gamma$ contains at least one strictly positive matrix.  
			
		\end{itemize}
	\end{condition}
	
	Notice that this condition is weaker than that used in 	\cite{BurDaGuiMent14}, where  the allowability is assumed for each element of the closure of $\Gamma$.

	We also need some notation on the Ulam-Harris tree. Let $\mathbb{N}^* = \{ 1, 2, \cdots \}$ and $\mathbb{U} = \cup_{n = 0}^{\infty} (\mathbb{N}^*)^n$ with the convention that $(\mathbb{N}^*)^0 = \{ \emptyset \}$. An element of level $n$ will be denoted by $u = (u_1, \cdots, u_n)  = u_1 \cdots u_n \in \mathbb{U}$. For any integer $k \geq 0$ and any sequence $u = u_1\cdots u_n$, we write $u|k = u_1\cdots u_k$ if $k \leq n$; 
	by convention 
	$u|0=\emptyset$. {\color{black} We write $v \leq u$ if there exists $k \geq 0$ such that $u | k = v$.} The same notation is used 
	for an infinite sequence $u = u_1u_2\cdots$. 
	We say that $U \subset \mathbb{U}$ is a cover set of $ (\mathbb{N}^*)^\infty$, if for any infinite sequence $u = u_1u_2\cdots \in (\mathbb{N}^*)^\infty$, there exists a unique integer $k \geq 0$ such that $u|k \in U$. We say that a subset of $U \subset \mathbb{U}$ has  level $l$  if $l= \sup \{ |u| : u \in U \}$, where $|u|$ denotes the length of the sequence $u$; we say that it has finite level when $l< \infty$. 
	
	Let $(N_u, A_{u1}, A_{u2}, \cdots)_{u \in \mathbb{U}}$    be a family of  i.i.d.  random vectors   defined on some probability space $(\Omega, \mathcal{F}, \mathbb{P})$,  with the convention  that  $(N_\emptyset , A_{\emptyset 1}, A_{\emptyset 2}, \cdots) = (N, A_1, A_2, \cdots)$. By iterating Equation~\eqref{equ::smoothing_transform}, we notice that if $Z$ is a solution to Equation~\eqref{equ::smoothing_transform}, then for any cover set $U$ with finite level,
	\begin{equation}
		Z \overset{\mathcal{L}}{=} \sum_{u \in U} A_{u|1}\cdots A_{u| |u|} Z_u,
	\end{equation}  
	where $(Z_u)_{u \in \mathbb{U}}$ are i.i.d. copies of $Z$, independent of $\{(N_u, A_{u1}, A_{u2}, \cdots)\}_{u \in \mathbb{U}}$.  Notice that by Condition~\ref{cond::conditions_on_N}, there are a.s. only a  finite number of  non-zero terms  in the above summation when $U$ has finite level. 
	
	
	\begin{condition}\label{cond::conditions_on_both_N_and_mu}
		There exist two cover sets $U_i$ of finite levels 
		and two realizations $\{ (a_{v1}^{(i)}, a_{v2}^{(i)}, \cdots)\}_{v\in \mathbb{U}} \subset \mathrm{supp}(A_1,A_2,\cdots)$ such that
		\begin{equation}
			l_i := \sum_{u  \in U_i} a_{u|1}^{(i)} \cdots a_{u| |u|}^{(i)}, \quad i = 1,2,
		\end{equation}
		satisfy
		\begin{equation}\label{l1-ge-0-l2-ge-0-rl1-le-1-rl2-ge-1}
			l_1 >0, 	\; l_2 > 0, \quad r(l_1) < 1 ,  \; r(l_2) > 1
		\end{equation}
		(recall that $r(l_i)$ denotes the spectral radius of $l_i$). 
		
		
	\end{condition}
	A sufficient condition for \ref{cond::conditions_on_both_N_and_mu}  to hold  is that there exist two matrices $l_1, l_2$ within  the support of the sum matrix 
	\begin{equation}
		Y = \sum_{i = 1}^N A_i
	\end{equation}	
	such that~\eqref{l1-ge-0-l2-ge-0-rl1-le-1-rl2-ge-1} holds. In this case, $l_i$ can be taken  in the form 
	$l_i= a_1^{(i)}  + \cdots +  a_{n_i}^{ (i)}$,  where  $ (n_i, a_1^{(i)}, \cdots, a_n^{(i)}, 0,  \cdots) \in \mbox{supp} (N, A_1, \cdots, A_N, 0, \cdots)$, and Condition~\ref{cond::conditions_on_both_N_and_mu} holds with $U_1=U_2 = \mathbb N^*$, 
	and,  for all   $v\in \mathbb{U}$, 
	$a_{vj}^{(i)} = a_j^{(i)} $ if $ j \leq n_i$, and    $ a_{vj}^{(i)}  =0 $ if  $j >n_i$.  
	In  the one-dimensional case $d=1$, this sufficient condition was used in~\citet{liu2001asymptotic} to prove that the support of solution is the whole positive real line. Here, for the general case $d \geq 1$, Condition~\ref{cond::conditions_on_both_N_and_mu} is used 
	in the proof of Theorem~\ref{thm::second_theorem}, where the two matrices 
	$l_1$ and $l_2$ serve to explore the points of the support near the origin and infinity.  
	
	Another sufficient condition for Condition~\ref{cond::conditions_on_both_N_and_mu} to  hold is the 
	non-arithmeticity condition (see condition \ref{non-arithmetic-cond}),  when $\alpha = 1$, where $\alpha$ is defined  in Condition~\ref{cond::conditions_on_alpha} below. This will be seen  in  
	the next section, in Proposition \ref{prop::non-arithmeticity-implies-others}. 	    
	The non-arithmeticity is frequently used in the literature to study the products of random matrices, see e.g.,  \citet{buraczewski2016precise} and  Xiao, Grama and  Liu~\cite{XiaoGramaLiu_BerryEssen22}.

	Let
	\begin{equation}
		I_\mu = \left\{ s \geq 0 : \int \| a \|^s \mu(da) < \infty \right\}, 
	\end{equation}
	and $(M_i)_{i \geq 1}$ be a sequence of i.i.d. matrices with law $\mu$.
	For $s \in I_\mu$,  define the log-convex functions
	\begin{equation}
		\kappa(s) = \lim_{n \to \infty}\mathbb{E}[\|M_n \cdots M_1 \|^s]^{\frac{1}{n}}, \quad m(s) = \mathbb{E}[N] \kappa(s). 
	\end{equation}
	Notice that in the one-dimensional case $\kappa $ reduces to the Laplace transform of $\log M_1$. For $B\subset \mathbb{R}$, denote its interior by $\mathrm{int}(B)$.

	
	\begin{condition}\label{cond::conditions_on_alpha}
		There exists an $\alpha \in (0, 1] \cap \mathrm{int}(I_\mu)$ such that $m(\alpha) = 1$ and $m'(\alpha) < 0$. 
	\end{condition}
	{ We remark that from~\cite[Lemma 4.14]{BurDaGuiMent14}, when $\alpha = 1$, we have $\kappa(1) = r(\mathbb{E}M_1)$. }
	
	Condition~\ref{cond::conditions_on_alpha} has been proved to be essential  for the existence of solution to  the distributional  equation \eqref{equ::smoothing_transform}.
	For the case $\alpha  = 1$, the existence and uniqueness of nonzero solution to~\eqref{equ::smoothing_transform} were shown
	in~\cite[Theorem 2.3]{BurDaGuiMent14}, under Condition~\ref{cond::conditions_on_independence} stated below. 
	For the general case $\alpha \in (0, 1]$, the characterization of the solution was given in~\cite{Ment16} by using the Laplace transform. 
	

	\begin{condition}\label{cond::conditions_on_independence}
		Given $N$, the matrices $A_1,\cdots,  A_N$  are conditionally i.i.d. with law $\mu$.
	\end{condition}
	This simple case is 
	of particular interest.  It will be referred as the i.i.d. case. 
	

	\subsubsection{Support of solution}
	For  a matrix $a > 0$, we denote by $v(a)$ its Perron-Frobenius right eigenvector with unit norm, { that is, $v(a)$ is the unique positive vector satisfying $av(a) = r(a)v(a)$ and $|v(a)| = 1$.} Define 
	\begin{equation}
		\Lambda = \overline{\{v(a) : a \in \Gamma, a >0\}}, 
	\end{equation}
	where   $\overline{\{ \cdot \}}$ denotes the closure of the set $\{ \cdot \}$. 
	Our main theorem on the support of solution is as follows.
	
	\begin{theorem}[Support of solution in the i.i.d. case]\label{thm::second_theorem}
		Assume Conditions \ref{cond::conditions_on_N}, \ref{cond::condition_on_mu},  \ref{cond::conditions_on_both_N_and_mu}, \ref{cond::conditions_on_alpha},\ref{cond::conditions_on_independence} and $\alpha = 1$. Let $Z$ be a solution of Equation~\eqref{equ::smoothing_transform} such that $\mathbb{P}[Z = 0] = 0$ {   and $\mathbb{E}[|Z|] < \infty$}. Then, 
		\begin{equation}\label{overline-D-supp-Z-overline-H}
			D \subset \mathrm{supp}(Z) \subset H,
		\end{equation}
		where 
		\begin{align} \label{def-D}
			D & := \{ s_1v_1 + \cdots + s_n v_n : n \in \mathbb{N}^*, n \leq  \mathrm{esssup}(N), s_i \geq 0, v_i \in \Lambda, \forall i = 1, \cdots, n \}, \\
			H & := \{ s_1v_1 + \cdots + s_nv_n : n \in \mathbb {N}^*, s_i \geq 0, v_i \in \Lambda, \forall i = 1, \cdots, n \}. 
			\label{H-s1v1+codts-snvn-n-ge-1}
		\end{align}
		Moreover, when $\mathrm{esssup}(N) \geq d$, we have
		\begin{equation}\label{supp-Z-H-d}
			\mathrm{supp}(Z) = H =  H_d, \mbox{ where } 
			H_d = \{ s_1v_1 + \cdots + s_dv_d :  s_i \geq 0, v_i \in \Lambda, \forall i = 1, \cdots, d \}. 
		\end{equation}
	\end{theorem}

	
	
	For the case where the coefficients are  not necessarily i.i.d., 
	instead of Theorem \ref{thm::second_theorem},
	we have the following inclusion relations of the support. {\color{black} Let $\mathbb{R}_+ = \{ s \in \mathbb{R} : s \geq 0 \}$ and $\mathbb{R}_+ v = \{ sv : s \in \mathbb{R}_+ \}$} for any vector $v$. 
	
	\begin{theorem}[Support of solution]\label{thm::main_theorem}
		Assume Conditions \ref{cond::conditions_on_N}, \ref{cond::condition_on_mu},  \ref{cond::conditions_on_both_N_and_mu}. Suppose that $Z$ is a solution of Equation~\eqref{equ::smoothing_transform} such that $\mathbb{P}[Z = 0] = 0$. Then:
		\begin{enumerate}
			\item For any cover set $U \subset \mathbb{U}$ with finite level, and any realizations $\{(a_{u1},a_{u2},\cdots)\}_{u \in \mathbb{U}}\subset \mathrm{supp}(A_1,A_2,\cdots)$, we have, 
			\begin{align}\label{equ::support-containus-many-rays}
				\mathrm{supp}(Z) \supset \sum_{\substack{V \subset U,\, g_V > 0}} \mathbb{R}_+ v(g_{V}),  
			\end{align}	
			where 
			\begin{equation}
				g_V := \sum_{u \in V} a_{u|1}\cdots a_{u| |u|}.
			\end{equation}
			\item If additionally  
			Condition~\ref{cond::conditions_on_alpha} holds with $\alpha = 1$ {   and $\mathbb{E}[|Z|] < \infty$}, then
			with $H, H_d$ defined in \eqref{H-s1v1+codts-snvn-n-ge-1}, \eqref{supp-Z-H-d},
			\begin{equation}
				\mathrm{supp}(Z) \subset H  = H_d. 
			\end{equation}
		\end{enumerate}
	\end{theorem}
	
	We remark that the number of summands in~\eqref{equ::support-containus-many-rays} can be very large, because it corresponds to the number of subsets  $V$ of $U$ such that $g_V > 0$. Theorem~\ref{thm::main_theorem} shows that the projection of the support to $\mathbb{S}_{+}^{d-1}$ contains the convex hull generated by $\{ v(g_V) : V \subset U, g_V > 0 \}$ for any conver set $U$ with finite level and all realizations $\{ (a_{u1},a_{u2},\cdots)\}_{u \in \mathbb{U}}$; in particular, 
	\begin{equation*}
		\mathrm{supp}(Z) \supset \{ sv: s \geq 0, v \in \Lambda \}.
	\end{equation*}
	
	The characterization on solutions to~\eqref{equ::smoothing_transform} with $\alpha \in (0, 1)$ was derived under some assumptions in~\cite{Ment16}. For such solutions, we also have the same result as Part 1 of Theorem~\ref{thm::main_theorem}. {\color{black} For a nonnegative matrix $a$, we define $\iota(a) = \inf_{x \in \mathbb{S}_{+}^{d-1}} |ax|$. For a real number $y \geq 0$, we write $\log_- y = \max (- \log y, 0)$ and $\log_+ y = \max(\log y, 0)$ (by convention $\log 0 = -\infty$).} 
	\begin{condition}\label{cond::condition_Ment16}
		{\color{black} The measure $\mu$ satisfies
			$
			\int  \|a\|^{\alpha}( \log_+\|a\| + \log_-\iota(a^T)) d\mu(a) < +\infty.
			$
		}
	\end{condition}
	

	\begin{theorem}[Support of solution for $\alpha <1$]\label{thm::third_theorem}
		Assume Conditions~\ref{cond::conditions_on_N}, \ref{cond::condition_on_mu}, \ref{cond::conditions_on_alpha}, \ref{cond::condition_Ment16} and $\alpha \in (0, 1)$. Suppose that $Z$ is a solution of Equation~\eqref{equ::smoothing_transform} such that $\mathbb{P}[Z = 0] = 0$. Then \eqref{equ::support-containus-many-rays} holds for all cover set $U \subset \mathbb{U}$ with finite level, and any $(a_{u1},a_{u2},\cdots)\in \mathrm{supp}(A_1,A_2,\cdots)$ for $u \in \mathbb{U}$.
	\end{theorem}




	\subsubsection{Absolute continuity}
	An interesting application of Theorem~\ref{thm::second_theorem} (and also of Theorems~\ref{thm::main_theorem} and \ref{thm::third_theorem}) is to prove the absolute continuity of the solution $Z$. Set 
	\begin{equation}
		N(t) = \# \{ i \geq 1 : A_i^Tt \ne 0 \} = \# \{ i  \geq 1: t \not \in \ker A_i^T \}, \quad t \in \mathbb{R}^d.
	\end{equation} 
	\begin{condition}
		The point process $(N, A_1, A_2, \cdots )$ satisfies  the following conditions: 
		\begin{itemize}\label{cond::absolute-continuity}
			\item A.s. $N(t) \geq 1$ for all $t \in \mathbb{R}^d \setminus \{ 0 \}$, with 
			\begin{equation}
				\mathbb{P}[N(t) = 1] < 1, \quad \forall t \in \mathbb{R}^d \setminus\{ 0 \}.
			\end{equation}
			
			\item 
			There exists $\varepsilon_0 > 0$ such that, 
			for some  $i(t) \in \mathbb{N^*}$  with
			$
			A^T_{i(t)} t \ne 0, 
			$ we have
			\begin{equation}\label{sup-t-S-d-1-E-A-it-T-t-epsilon-0}
				\sup_{t \in \mathbb{S}^{d-1}} \mathbb{E}\big[ |A_{i(t)}^Tt|^{-\varepsilon_0} \big] < +\infty, 
			\end{equation}
			
		\end{itemize}
	\end{condition}
	
		%
	
	
	\begin{theorem}[Decay rate of the characteristic function and absolute continuity]\label{thm::absolute-continuity}
		Assume Conditions~\ref{cond::conditions_on_N}, \ref{cond::conditions_on_alpha},~\ref{cond::absolute-continuity}. Let $Z$ be a solution of Equation~\eqref{equ::smoothing_transform} such that $\mathbb{P}[Z = 0] = 0$. Then there are constants  $a>0$  and $C>0$ such that for all $t \in \mathbb R^d$, 
		$$ |\mathbb E e^{i \langle t, Z\rangle } | \leq C |t|^{-a}.  $$
		If additionally there exist linearly independent vectors $\{ v_i \}_{1 \leq i \leq d} \subset \mathbb{R}_+^d \setminus \{ 0 \}$ satisfying 
		\begin{equation}\label{equ::support-contains-many-vectors}
			\mathbb{R}_+ v_i \subset \mathrm{supp}(Z) , \quad i = 1, \cdots, d, 
		\end{equation}
		then, the law of $Z$ is absolutely continuous with respect to the Lebesgue measure on $\mathbb{R}^d$.
	\end{theorem}
	
		Condition~\eqref{equ::support-contains-many-vectors}
		can be checked by using Theorem~\ref{thm::second_theorem} for the i.i.d. case, and Theorems~\ref{thm::main_theorem} and \ref{thm::third_theorem} for  non i.i.d. case. Thus, the absolute continuity can be established 
		assuming only conditions on $(N,A_1,A_2,\cdots)$. Combining Theorems~\ref{thm::main_theorem} and~\ref{thm::absolute-continuity}, we obtain the following: 
		{\color{black}
		
		
		\begin{corollary}\label{cor::absolute-continuity-2}
			Assume Conditions \ref{cond::conditions_on_N}, \ref{cond::condition_on_mu},  \ref{cond::conditions_on_both_N_and_mu}, \ref{cond::conditions_on_alpha}, \ref{cond::absolute-continuity}, and $\Lambda$ contains at least $d$ linearly independent vectors. Let $Z$ be a solution of Equation~\eqref{equ::smoothing_transform} such that $\mathbb{P}[Z = 0] = 0$. Then, the law of $Z$ is absolutely continuous with respect to the Lebesgue measure on $\mathbb{R}^d$.
		\end{corollary}
		
		
	}
	
	\subsubsection{Harmonic moments}
	We finally derive the critical exponent for the existence of  harmonic moments of $Z$. We will use the Furstenberg-Kesten condition on $A_1$:
	\begin{condition}\label{cond::Furstenberg-Kesten}
		There exists a constant $c > 1$ such that
		\begin{equation}
			0 < \max_{1 \leq i,j\leq d} A_1(i, j) \leq c \min_{1\leq i,j \leq d} A_1(i, j), \quad \mbox{a.s.,}
		\end{equation}
		where {$A_1(i, j)$} is the $(i,j)$-th entry of $A_1$. 
	\end{condition}
	If $\mathbb{P}[N = 1] > 0$, then we denote by $\tilde{A}_1$ the random matrix with the law of $A_1$ conditioned on $\{  N = 1 \}$. Define
	\begin{equation}
		\tilde{I}_- := \{ s \leq 0 : \mathbb{E}[\| \tilde{A}_1 \|^s] < +\infty \},
	\end{equation}
	and let $(\tilde{A}_i)_{i \geq 2}$ be a sequence of i.i.d. copies of $\tilde{A}_1$. From~\cite[Proposition 2.2]{xiao2022edgeworth}, under Condition~\ref{cond::Furstenberg-Kesten}, the following limit exists for $s \in \tilde{I}_-$: 
	\begin{equation}
		\tilde{\kappa}(s) := \lim_{n \to \infty} \mathbb{E}[\| \tilde{A}_n \cdots \tilde{A}_1 \|^s ]^{\frac{1}{n}}  \in (0, \infty).
	\end{equation}
	
	\begin{theorem}[Harmonic moments]\label{thm::critical-exponent-negative-moment}
		Assume Conditions~\ref{cond::conditions_on_N}, \ref{cond::Furstenberg-Kesten}. Let $Z$ be a solution of~\eqref{equ::smoothing_transform} such that $\mathbb{P}[Z = 0] = 0$. Let $a > 0$ be such that $\mathbb{E}[\| A_1\|^{-a}] < +\infty$. Consider the three statements:
		\begin{equation}
			\mbox{(a) } \mathbb{E}[| Z |^{-a}] < \infty, \quad
			\mbox{(b) } \tilde{\kappa}(-a) \mathbb{P}[N = 1] < 1, \quad 
			\mbox{(c) } \mathbb{E}[| Z |^{-b}] < \infty, \quad \forall b \in (0, a).
		\end{equation}
		Then, we have the following conclusions:
		\begin{enumerate}
			\item If $\mathbb{P}[N = 1] = 0$, then (c) holds. 
			\item  If $\mathbb{P}[N = 1] > 0$, then $\mbox{(a)} \Rightarrow \mbox{(b)} \Rightarrow \mbox{(c)}$. If additionally  $\mathbb{E}[\| A_1\|^{-a-\varepsilon}] < +\infty$ for some $\varepsilon >0$, then  $\mbox{(a)} \Leftrightarrow \mbox{(b)}$. 
		\end{enumerate}
	\end{theorem}
	Therefore, in case 2, the critical value for the existence of harmonic moments is determined by the equation  $ \tilde{\kappa}(-a) \mathbb{P}[N = 1]  = 1$, under a mild moment condition. 
	
	{ \subsection{Examples}
		We provide three examples to illustrate the applications of the above theorems, in the case where $d=2$.  
		\begin{example}[Support]\label{exam::1}
			Define $v_1 = u = (1, 1)^T$, $v_2 = (1, 2)^T$, $a_1 = \frac{v_1\otimes u}{5}$, $a_2 = \frac{v_2\otimes u}{5}$, where $v\otimes u := v u^T$. Set $N = 2$, $\mu = \frac{1}{2} \delta_{a_1} + \frac{1}{2} \delta_{a_2}$, and let $A_1, A_2$ be two i.i.d. matrices with law $\mu$. 
			
			Since $\mathbb{E}[N]\kappa(1) = 2r(\int a\mu(da)) = r(a_1 + a_2) = 1$,
			we can check that Conditions~\ref{cond::conditions_on_N}, \ref{cond::condition_on_mu}, \ref{cond::conditions_on_both_N_and_mu}, \ref{cond::conditions_on_alpha}, \ref{cond::conditions_on_independence} hold, with $U_1 = U_2 = \mathbb{N}^* $, $l_1 = 2a_1$, $l_2 = 2a_2$, $\alpha = 1$. If $Z$ is a solution of~\eqref{equ::smoothing_transform} such that $\mathbb{P}[Z = 0] = 0$, then, by Theorem~\ref{thm::second_theorem}, we derive that $\mathrm{supp}(Z) = \{ s_1v_1 + s_2v_2 : s_1, s_2 \geq 0 \}$,  the convex cone generated by $v_1, v_2$. 
		\end{example}
		
		\begin{example}[Absolute continuity]\label{exam::2}
			Define $v_1, v_2, u, a_1, a_2$ as in Example~\ref{exam::1}, and let $X$ be a real random variable such that $\mathbb{P}[X = \frac{1}{4}] = \mathbb{P}[X = \frac{3}{4}] = \frac{1}{2}$. 
			Set $N = 3$, $A_1 = Xa_1$, $A_2 = Xa_2$, $A_3 = X(a_1 + a_2)$.
			
			In this case, it can be checked that Conditions~\ref{cond::conditions_on_N}, \ref{cond::condition_on_mu}, \ref{cond::conditions_on_both_N_and_mu}, \ref{cond::conditions_on_alpha} hold, with $U_1 = U_2 = \mathbb{N}^* $, $l_1 = \frac{a_1 +a_1 }{2}$, $l_2 = \frac{3(a_1+a_2)}{2}$, $\alpha = 1$. 
			
			Denote $v_3 = v_1 + v_2$. Since $N(t) = \#\{ i = 1, 2, 3 : \langle t, v_i\rangle \neq 0 \} \geq 2$ for each $t \in \mathbb{R}^d \setminus \{ 0\}$, Condition~\ref{cond::absolute-continuity} holds. Notice that $\Lambda = \{ v_i/|v_i| : i = 1, 2, 3 \}$ contains three linearly independent vectors. 
			
			Therefore, from Corollary~\ref{cor::absolute-continuity-2}, we derive that, 
			if $Z$ is a solution of~\eqref{equ::smoothing_transform} such that $\mathbb{P}[Z = 0] = 0$, then  the law of $Z$ is absolutely continuous. 
		\end{example}
		
		\begin{example}[Harmonic moments]
			Define $v_1, v_2, u, a_1, a_2$ as in Example~\ref{exam::1}. Let $\mathbb{P}[N = 1] = \mathbb{P}[N = 2] = \frac{1}{2}$. On $\{ N = 1 \}$, define $A_1 = Xa_1 + (1 - X)a_2$, where $X$ is a real random variable independent of $N$ with $\mathbb{P}[X = 0] = \mathbb{P}[X = 1] = \frac{1}{2}$; on $\{ N = 2 \}$, define $(A_1, A_2) = (a_1, a_2)$. 
			
			Notice that Conditions~\ref{cond::conditions_on_N} and \ref{cond::Furstenberg-Kesten} hold. 
			It can be  checked that $\tilde{\kappa}(s) = {(2^s + 3^s)} / {(2 \cdot 5^s)}$ for all $s \in \mathbb{R}$. Let $a_0 >0$ be the solution of equation $\tilde{\kappa}(-a)\mathbb{P}[N = 1] = 1$, that is, $a_0$ satisfies ${(2^{-a_0} + 3^{-a_0})} / {(2 \cdot 5^{-a_0})} = 2$. Therefore by Theorem~\ref{thm::critical-exponent-negative-moment},
			if $Z$ is a solution of~\eqref{equ::smoothing_transform} such that $\mathbb{P}[Z = 0] = 0$, then $\mathbb{E}[|Z|^{-a}] < +\infty$ for $a \in (0, a_0)$, and $\mathbb{E}[|Z|^{-a}] = +\infty$ for $a \in [a_0, \infty)$.
		\end{example}
	}

	\subsection{The structure of the paper}
	In Section~\ref{sec::section_2}, 
	we show that the usual condition of non-arithmeticity implies Condition~\ref{cond::conditions_on_both_N_and_mu}:  see Proposition~\ref{prop::non-arithmeticity-implies-others}. In Section~\ref{section::support}, we  first prove Theorem \ref{thm::main_theorem} and then use it to prove Theorems \ref{thm::second_theorem} and \ref{thm::third_theorem}.
	In Section~\ref{section::absolute_continuity}, we establish the absolute continuity by following the argument in~\citet{liu2001asymptotic} in the one-dimensional case, while overcoming specific challenges that arise in the multidimensional case.
	In Section~\ref{section::negative-moments}, we find the critical value of existence of harmonic moments, 
	using the spectral gap theory for negative parameter 
	developed in~\citet{xiao2022edgeworth}.  


	\section{Non-arithmeticity and Condition~\ref{cond::conditions_on_both_N_and_mu}} \label{sec::section_2}
	
	\subsection{Support under transformation}
	In this subsection, we do not assume any condition on the point process $(N, A_1, A_2, \cdots)$. In order to prove Theorem~\ref{thm::main_theorem}, we need to iterate~\eqref{equ::smoothing_transform} and obtain some  inclusion relations of the support of solution. To this end, we will make use of the following general result on the support (see \eqref{def-support} for the definition) of a measure on a Polish space. 
	\begin{lemma}\label{lem::support_breaking_general}
		For $i \geq 1$, let   $ \Omega_i$ be Polish space equipped with the Borel $\sigma$-field, and  $X_i$ be $\Omega_i$-valued  random variable.  For $n \geq 1$,  let $f : \prod_{i = 1}^{n} \Omega_i \to \Omega$ be a continuous mapping to some Polish space $\Omega$ equipped with the Borel $\sigma$-field. Then: 
		\begin{enumerate} 
			\item \label{item::support_breaking_without_independence} We have
			\begin{equation}\label{equ::support_breaking_general_1}
				\mathrm{supp}(f(X_1, \cdots, X_n)) \subset \overline{f(\mathrm{supp}(X_1),\cdots, \mathrm{supp}(X_n))}.
			\end{equation}
			In particular, if $O$ is an open set in $\Omega$ such that
			\begin{equation}\label{equ::ad_hoc_6}
				\mathbb{P}[f(X_1, \cdots, X_n) \in O] > 0,
			\end{equation}
			then, there exists an $n$-tuple   $(x_1, \cdots, x_n) \in \prod_{i = 1}^n \mathrm{supp}(X_i)$ such that $f(x_1, \cdots, x_n) \in O$.
			\item \label{item::support_breaking_with_independence}
			If $X_1, \cdots, X_n$ are independent, then, 
			\begin{equation}\label{equ::support_breaking_general_2}
				\mathrm{supp}(f(X_1, \cdots, X_n)) = \overline{f(\mathrm{supp}(X_1),\cdots, \mathrm{supp}(X_n))}.
			\end{equation}
			\item \label{item::support_breaking_independent_product}  If $(X_i)_{i \geq 1}$ is a sequence of independent random variables, 
			then 
			\begin{align} \label{supp-Z1Z2}
				\mathrm{supp}(X_1, X_2, \cdots ) = \prod_{i\geq 1}\mathrm{supp}(X_i).
			\end{align}
		\end{enumerate}
	\end{lemma}
	The proof of the lemma is postponed to the Appendix (Section \ref{sec::appendix}).


	We come to derive some inclusion relations  on $\mathrm{supp}(Z)$ by using Lemma~\ref{lem::support_breaking_general}. Recall that $\mathbb{U} = \cup_{n \geq 0}(\mathbb{N}^*)^n$. Using Lemma~\ref{lem::support_breaking_general} and the independence between $(N,(A_i)_{i \geq 1})$ and $(Z_i)_{i \geq 1}$, we deduce from~\eqref{equ::smoothing_transform} that
	\begin{align}\label{suppZ-suppi=1NAiZi-supset}
		\mathrm{supp}(Z) & = \mathrm{supp}\big( \sum_{i = 1}^N A_iZ_i \big) \notag \\
		& \supset \bigg\{ \sum_{i = 1}^n a_iz_i : (n,(a_i)_{i \geq 1}) \in \mathrm{supp}(N,(A_i)_{i \geq 1}), (z_i)_{i \geq 1} \in \mathrm{supp}((Z_i)_{i \geq 1}) \bigg\} \notag \\
		& = \bigg\{ \sum_{i = 1}^n a_iz_i : (n,(a_i)_{i \geq 1}) \in \mathrm{supp}(N,(A_i)_{i \geq 1}), z_i \in \mathrm{supp}(Z)  \bigg\},
	\end{align}
	where we use Part 2 of Lemma~\ref{lem::support_breaking_general}  for the inclusion, and Part 3 of Lemma~\ref{lem::support_breaking_general} for the last equality. 
	
	For any cover set $U \subset \mathbb{U}$ with finite level, we have
	\begin{equation}\label{equ::support_breaking_cover_set}
		\mathrm{supp}(Z) \supset \bigg\{ \sum_{u \in U} {  g_u}z_u : (a_{ui})_{i \geq 1} \in \mathrm{supp}((A_i)_{i \geq 1}), z_u \in \mathrm{supp}(Z), \forall u \in \mathbb{U} \bigg\},
	\end{equation}
	where ${  g_u} :=   a_{u|1} \cdots   a_{u| |u|}$ if $u = u_1\cdots u_n$. Indeed,  this can be proved  by induction. Note that for $U = \{ \emptyset \}$, the right hand side of~\eqref{equ::support_breaking_cover_set} is $\mathrm{supp}(Z)$, hence~\eqref{equ::support_breaking_cover_set} holds for any cover set  of level $0$. Suppose that~\eqref{equ::support_breaking_cover_set} holds for any cover set $U$ with level $n \geq 0$, that is, the maximal length of elements in $U$ is $n$. 
	For any cover set $U$ with level $n+1$, 
	let $V$ be the set of parents of the elements of $U$ with  length  $n+1$.   Since $U$ is a cover set of level  $n+1$, we have  $V \cap U = \emptyset$ and 
	
	\begin{equation}
		\{vi : v \in V, i \in \mathbb{N}^* \} = \{ u \in U : |u| = n+1 \}.
	\end{equation}
	Given $(a_{ui})_{i \geq 1} \in \mathrm{supp}((A_i)_{i\geq1})$ and $z_u\in \mathrm{supp}(Z)$ for each $u \in \mathbb{U}$, we have
	\begin{equation}\label{sum-u-in-U-guzu}
		\sum_{u \in U} {  g_u}z_u= \sum_{\substack{u \in U \\ |u| \leq n}} {  g_u}z_u + \sum_{\substack{u \in U \\ |u|=n+1}} {  g_u}z_u = \sum_{\substack{u \in U \\ |u| \leq n}} {  g_u}z_u + \sum_{v \in V} {  g_v} \sum_{i \geq 1} a_{vi}z_{vi}.
	\end{equation}
	Note that using~\eqref{suppZ-suppi=1NAiZi-supset}, we have $\sum_{i \geq 1} a_{vi}z_{vi} \in \mathrm{supp}(Z)$. Let $\tilde{z}_u = \sum_{i \geq 1} a_{ui}z_{ui}$ if $u \in V$, and let $\tilde{z}_u = z_u$ if $u \not \in V$. Then, we have $\tilde{z}_u \in \mathrm{supp}(Z)$ for each $u \in \mathbb{U}$. Define \mbox{$U_1 = \{ u \in U : |u| \leq n \}\cup V$}. Then, we obtain
	\begin{equation} \label{eq-sumU-sumU1}
		\sum_{u \in U} {  g_u}z_u = \sum_{u \in U_1}{  g_u} \tilde{z}_u. 
	\end{equation}
	Since $U_1$ is a cover set of $\mathbb{U}$ with level $n$, by the induction hypothesis   \eqref{equ::support_breaking_cover_set} holds for $U_1$. So from  \eqref{eq-sumU-sumU1}
	we derive that $\sum_{u \in U} {  g_u}z_u \in \mathrm{supp}(Z)$. 
	So we have proved that \eqref{equ::support_breaking_cover_set} remains true for any cover set $U$ with level $n+1$. Therefore, by induction,  \eqref{equ::support_breaking_cover_set} holds for any cover set $U$ with any finite level $n \geq 0$. 
	
	We will make use of the following particular case of~\eqref{equ::support_breaking_cover_set}. For any cover set $U \subset \mathbb{U}$ with finite level and any $(a_{u1},a_{u2},\cdots) \in \mathrm{supp}(A_1,A_2\cdots)$ for each $u \in \mathbb{U}$, we denote
	\begin{equation}
		l := \sum_{u  \in U} a_{u|1}\cdots a_{u | |u|}.
	\end{equation}
	Then, we derive from~\eqref{equ::support_breaking_cover_set} that
	\begin{equation}\label{equ::support_breaking_easy_form}
		\mathrm{supp}(Z) \supset \{ l z :  z \in \mathrm{supp}(Z) \}.
	\end{equation}
	
	\subsection{Non-arithmeticity implies Condition~\ref{cond::conditions_on_both_N_and_mu}}
	
	\begin{condition}[Non-arithmeticity]\label{non-arithmetic-cond} 
		We say that a measure $\nu$ on the set of nonegative matrices is non-arithmetic, if for all $t > 0$, $\theta \in [0, 2\pi)$ and function $\vartheta : \mathbb{S}_+^{d-1} \to \mathbb{R}$, there exist $a \in \Gamma(\nu)$ and $x \in \Lambda(\Gamma(\nu))$ such that
		\begin{equation}
			\exp \big( it\log|ax| -i\theta+i(\vartheta(a\cdot x) - \vartheta(x)) \big) \ne 1,
		\end{equation}
		where $\Gamma(\nu) = [\mathrm{supp}(\nu)]$ is the multiplicative semigroup generated by $\mathrm{supp}(\nu)$, and
		$\Lambda(\Gamma(\nu)) = \overline{\{v(a) : a \in \Gamma(\nu) , a > 0\} }$.
	\end{condition} 
	
	For any nonnegative matrix $a$, define
	\begin{equation}
		N(a) = \mathrm{max} \{ \|a\|, \iota(a)^{-1} \},
	\end{equation}
	The following proposition shows that when $\alpha = 1$,  Condition~\ref{cond::conditions_on_both_N_and_mu} can be deduced from the non-arithmeticity. 
	
	\begin{proposition}\label{prop::non-arithmeticity-implies-others}
		Condition~\ref{cond::conditions_on_both_N_and_mu} holds in each of the following cases: 
		
		\begin{enumerate}
			\item Case 1:  Conditions~\ref{cond::conditions_on_N} and \ref{cond::condition_on_mu} hold, \ref{cond::conditions_on_alpha} holds with $\alpha = 1$,   and  \ref{cond::conditions_on_independence} holds  with 
			$\mu$   non-arithmetic and  
			\begin{equation} \label{1st-moment-cond}
				\int \log_+N(a)d\mu(a) < +\infty.
			\end{equation}
			\item Case 2: Conditions~\ref{cond::conditions_on_N} and \ref{cond::condition_on_mu} hold, \ref{cond::conditions_on_alpha} holds with $\alpha = 1$, and 
			{  the law $\mu_Y$ of the random variable $Y := \sum_{i = 1}^N A_i$
				is  non-arithmetic }and   
			\begin{equation}
				\int \log_+N(a)d\mu_Y(a) < +\infty.
			\end{equation}
			
		\end{enumerate} 
	\end{proposition}
	{  Remark that for any cover set $U$ with finite level,  Condition~\ref{cond::conditions_on_both_N_and_mu} still holds if we are in  Case 2 of Proposition~\ref{prop::non-arithmeticity-implies-others}, with $Y$ replaced by $\sum_{u \in U} A_{u | 1} \cdots A_{u  | |u |}$. 
	}

	Let $(M_i)_{i \geq 1}$ be a sequence of i.i.d. matrices with law $\mu$
	satisfying  \eqref{1st-moment-cond}, 
	{ $G_n = M_1 \cdots M_n$ for each $n \geq 1$}, and
	\begin{equation}
		\gamma = \lim_{n \to \infty}\frac{1}{n} \mathbb{E}[\log\|G_n\|]
	\end{equation}
	be the Lyapunov exponent associated with the sequence $(M_n)_{n \geq 1}$.
	Then $\gamma$ is finite. Recall that 
	\begin{equation}
		\kappa(s) := \lim_{n \to \infty}\mathbb{E}[\|G_n \|^s]^{\frac{1}{n}}, \quad m(s) := \mathbb{E}[N] \kappa(s), \quad s \in I_\mu. 
	\end{equation}
	{ Recall also that from~\cite[Lemma 4.14]{BurDaGuiMent14}, when $\alpha = 1$, we have $\kappa(1) = r(\mathbb{E}M_1)$. } Define the convex function
	\begin{equation}
		\Lambda(s) = \log \kappa(s), \quad s \in I_\mu. 
	\end{equation}
	
	\begin{lemma}\label{lem::gamma_is_strictly_smaller}
		Assume Conditions~\ref{cond::conditions_on_N},~\ref{cond::condition_on_mu}. Suppose that \mbox{$\int \log_+ N(a)d\mu(a) < +\infty$}, $I_\mu$ has non-empty interior, and $\mu$ is non-arithmetic.  Then, we have
		\begin{equation}\label{gamma-less-1-s0}
			\gamma < \frac{1}{s} \log \kappa(s), \quad \forall s \in I_\mu.
		\end{equation}
	\end{lemma}
	
	
	\begin{proof}[Proof of Lemma~\ref{lem::gamma_is_strictly_smaller}]
		From~\cite[Theorem 6.1]{BurDaGuiMent14}, we know that $\gamma = \Lambda'(0+)$. 
		{ Since $\mu$ is non-arithmetic, we can argue as in{  ~\cite[Theorem 2.6]{GuiLe16}} and ~\cite[Lemma 3.16]{XiaoGramaLiu_BerryEssen22} and derive that $\Lambda := \log \kappa$ is a strictly convex function on $I_\mu$. } Indeed, although the statements there require different conditions on the measure $\mu$, we can prove in the same way by using the spectral gap theory developed in~\cite[Proposition 3.1]{BurDaGuiMent14} and the non-arithmeticity of $\mu$. 
		Since $\Lambda$ is strictly convex, we obtain that for $s \in I_\mu \setminus \{0\},$
		\begin{equation}
			\log \kappa(s) = \Lambda(s) > s \Lambda'(0+) + \Lambda(0) = s \gamma.
		\end{equation}
	\end{proof}

	\begin{corollary}\label{cor::gamma_is_strictly_smaller}
		{ Under the hypothesis of Lemma~\ref{lem::gamma_is_strictly_smaller} and Condition~\ref{cond::conditions_on_alpha}, we have}
		\begin{equation}\label{gamma-le-log-frac-1-EN}
			\gamma < \log \frac{1}{\mathbb{E}[N]}. 
		\end{equation}
	\end{corollary}
	\begin{proof}
		Using Lemma~\ref{lem::gamma_is_strictly_smaller} and the condition that $\mathbb{E}[N]\kappa(\alpha) = m(\alpha) = 1$, we have (since $\alpha \in (0,1]$), 
		\begin{equation}
			\gamma   < \frac{1}{\alpha} \log \kappa(\alpha) = \frac{1}{\alpha} \log \frac{1}{\mathbb{E}[N]} \leq \log \frac{1}{\mathbb{E}[N]}.
		\end{equation}
		So \eqref{gamma-le-log-frac-1-EN} holds.  
	\end{proof}
	
	
	\begin{lemma}\label{lem::product_with_small_spectral_radius}
		Assume Conditions~\ref{cond::conditions_on_N},\ref{cond::condition_on_mu},  $\int \log_+ N(a)d\mu(a) < +\infty$. 
		Then for all 
		$\eta > e^\gamma$
		and $\varepsilon > 0$, there exist $K \geq 1$ and $a_1, \cdots, a_{K} \in \mathrm{supp}(\mu)$ such that $g := a_{1} \cdots a_K > 0$ and
		\begin{equation}\label{equ::ad_hoc_8}
			\|g\| \leq \varepsilon \eta^K.
		\end{equation}
		In particular, we have
		\begin{equation}\label{equ::ad_hoc_9}
			r(g) \leq \varepsilon \eta^K.
		\end{equation}
	\end{lemma}
	\begin{proof}
		Using the definition of $\gamma$ and the fact that $\gamma < \log \eta$, we know there exists $\delta > 0$ and an integer $n_0 > 0$ such that
		\begin{equation}
			\frac{1}{n} \mathbb{E}[\log \|G_n\|] < \log \eta - \delta, \quad \forall n \geq n_0.
		\end{equation}
		This implies that 
		\begin{equation}
			\mathbb{P}\big[\frac{1}{n}\log\|G_n\| < \log \eta - \delta \big] > 0, \quad \forall n \geq n_0.
		\end{equation}
		Using Part 1 of Lemma~\ref{lem::support_breaking_general}, we know that there exist $n \geq n_0$ and $b_1, \cdots, b_n \in \mathrm{supp}(\mu)$ such that $b = b_1\cdots b_n$ satisfies
		\begin{equation}
			\frac{1}{n} \log \|b_1\cdots b_n\| < \log \eta - \delta,
		\end{equation}
		which yields
		\begin{equation}
			\|b_1\cdots b_n\| < \exp(-n\delta)\eta^n.
		\end{equation}
		By Condition~\ref{cond::condition_on_mu}, there exist an integer $m \geq 1$ and $h_1,\cdots,h_m\in\mathrm{supp}(\mu)$ such that $h_1\cdots h_m > 0$. Then, using the submultiplicity of matrix norm, we have
		\begin{align}\label{bmb1ana1-le}
			\|h_1\cdots h_m b_1\cdots b_n\|
			& \leq \| h_1\cdots h_m\|\cdot\|b_1 \cdots b_n\|     \nonumber  \\ 
			&\leq \| h_1\cdots h_m\| e^{-n\delta}\eta^n = (\| h_1\cdots h_m\|\eta^{-m}e^{-n\delta})\eta^{n+m}. 
		\end{align}
		We choose $n$ large enough such that $\| h_1\cdots h_m\|\eta^{-m}e^{-n\delta} \leq \varepsilon$. Let $g = h_1\cdots h_m b_1\cdots b_n$ and $K = n+m$. Then, we have $g > 0$ and $\|g\| < \varepsilon{\eta^K}$. The inequality $r(g) \leq \varepsilon{\eta^K}$ follows from the fact that $r(g) = \lim_{k \to \infty} \|g^k\|^{1/k} \leq \|g\|$.
	\end{proof}
	
	%

	\begin{lemma}\label{lem::products_with_large_spectral_radius}
		Let $A$ be a random matrix with law $\mu$ which satisfies Conditions~\ref{cond::conditions_on_N},~\ref{cond::condition_on_mu}, 
		$\int \log_+ N(a)d\mu(a) < +\infty$,  
		$1 \in I_\mu$, and
		$r(\mathbb EA) > e^{\gamma}.$
		Then, for all 
		$R > 0$, there exist an integer $K > 0$ and a sequence $a_1, \cdots, a_K$ in $\mathrm{supp}(\mu)$ such that $g := a_{1} \cdots a_K > 0$ and that
		\begin{equation}
			r(g) \geq   R \; \big(r(\mathbb E A)\big)^K. 
		\end{equation} 
	\end{lemma}
	\begin{proof}
		By Condition~\ref{cond::condition_on_mu},
		there are $m \geq 1, h_1, \cdots, h_m \in \mbox{supp} (\mu)$ such that 
		$h_1 \cdots h_m >0$. This implies that  $(\mathbb{E}[A])^{m} >0$, that is, 
		$\mathbb{E}[A]$ is primitive. 
		Using the Perron-Frobenius decomposition, we know that there exists a positive vector $u$ such that $\mathbb{E}[A]u = r( \mathbb E A) u$. 
		By   the law of large numbers  of \cite{furstenberg1960products}
		(which is a consequence of  Kingman's subadditive ergodic theorem), we have 
		\begin{equation}\label{equ::ad_hoc_23}
			\gamma = \lim_{n \to \infty} \frac{1}{n} \log \|G_n\| \quad a.s. 
		\end{equation}
		The idea to this proof is as follows. On the one hand,  the expectation  $\mathbb E[ \|G_n\|]$ is of order $ \big( r(\mathbb E A)\big)^{n(1+o(1))}$. On the other hand, we have $\|G_n\| \leq e^{n(\gamma + o(1))}$ with high probability. Note that 
		$ \big( r(\mathbb E A)\big)^{n(1+o(1))}$
		is much larger than $e^{n(\gamma + o(1))}$. So $\|G_n\|$ has to be much larger than 
		$ \big( r(\mathbb E A)\big)^{n(1+o(1))}$
		with positive probability. 
		
		Let $\delta = \left( \log r( \mathbb EA) - \gamma \right) / 2 > 0$, and $v$ be a positive vector with $\langle u, v \rangle = 1$. Define $$E_n = \{ \log \|G_n\| \leq n(\gamma + \delta) \}.$$ Using the fact that $\mathbb{E}[G_nu] = (\mathbb{E}A)^nu = \big( r(\mathbb E A)\big)^{n}u$,
		we have
		\begin{align}\label{equ::ad_hoc_24}
			\big(r(\mathbb E A)\big)^n &= \mathbb{E}[\langle v, G_nu\rangle] \notag \\
			& = \mathbb{E}[\langle v, G_nu\rangle\mathbbm{1}_{E_n}] + \mathbb{E}[\langle v, G_nu\rangle\mathbbm{1}_{E_n^c}] \notag \\
			& \leq |v|_\infty |u| \exp(n(\gamma+\delta)) + \mathbb{E} [\langle v, G_nu\rangle \mathbbm{1}_{E_n^c}], 
		\end{align}
		where $|v|_\infty := \max_{i = 1, \cdots, d} |v(i)|$ if $v = (v(1), \cdots, v(d))$. 
		Since $\gamma + \delta < \log r(\mathbb E A)$, there exists $n_0 \geq 1$ large enough such that for all $n \geq n_0$,
		\begin{equation}\label{equ::ad_hoc_26}
			\big(r(\mathbb E A)\big)^n - |v|_\infty |u| \exp(n(\gamma+\delta)) > \frac{\big( r(\mathbb E A)\big)^{n}}{2} .
		\end{equation} 
		By~\eqref{equ::ad_hoc_24}, this implies that
		\begin{equation}\label{equ::ad_hoc_25}
			\mathbb{E}[\langle v, G_nu\rangle \mathbbm{1}_{E_n^c}] > \frac{\big(r(\mathbb E A)\big)^{n}}{2}, \quad n \geq n_0.
		\end{equation}
		In particular $\mathbb{P}[E_n^c] > 0$ for $n \geq n_0$. 
		
		Let $\varepsilon > 0$ be any small number.  Using~\eqref{equ::ad_hoc_23}, we know that there exists $n_1 = n_1(\varepsilon)  \geq n_0$ such that for all $n \geq n_1$, we have $\mathbb{P}[E_n^c] < \varepsilon$. Using the fact that $\mathbb{P}[E_n^c] > 0$ for $n \geq n_1$, we see that for each $n \geq n_1$, we can define a probability measure $\mathbb{\tilde{P}}$ (which depends on $n$) such that  the corresponding expectation 
		$ \mathbb{\tilde{E}}$ satisfies, for any continuous function $f: G \to \mathbb R$, 
		\begin{equation}
			\mathbb{\tilde{E}}[f(G_n)] = \frac{\mathbb{E}[f(G_n)\mathbbm{1}_{E_n^c}]}{\mathbb{P}[E_n^c]}.
		\end{equation}
		Using~\eqref{equ::ad_hoc_25} and the fact that $\mathbb{P}[E_n^c] < \varepsilon$ for $n \geq n_1$, we have
		\begin{equation}
			\mathbb{\tilde{E}}[\langle v, G_nu\rangle] = \frac{\mathbb{E}[\langle v, G_nu\rangle \mathbbm{1}_{E_n^c}]}{\mathbb{P}[E_n^c]} > \frac{\big(r(\mathbb E A)\big)^{n+1}}{2\varepsilon}, \quad n \geq n_1. 
		\end{equation}
		This implies that
		\begin{equation}
			\frac{\mathbb{P}[E_n^c \cap \{ \langle v, G_nu\rangle> {\big(r(\mathbb E A)\big)^{n}} / {(2\varepsilon)} \}]}{\mathbb{P}[E_n^c]}  = \mathbb{\tilde{P}}\big[\langle v, G_nu\rangle > {\big(r(\mathbb E A)\big)^{n}} / {(2\varepsilon)} \big] > 0, \quad n \geq n_1.
		\end{equation}
		In particular, we have $\mathbb{P}\big[\langle v, G_nu\rangle > {\big(r(\mathbb E A)\big)^{n}} / {(2\varepsilon)} \big] > 0$ for all $n \geq n_1$. Thus, by Part 1 of Lemma~\ref{lem::support_breaking_general}, for each given $n \geq n_1$, there exists a sequence of elements $b_1, \cdots, b_n$ in $\mathrm{supp}(\mu)$ such that $b := b_1 \cdots b_n$ satisfies
		\begin{equation}\label{lower-bd-vbw}
			\langle v, bu\rangle> \frac{\big(r(\mathbb E A)\big)^{n}}{2\varepsilon}.
		\end{equation}
		
		{ Recall that $h_1, \cdots, h_m \in \mathrm{supp}(\mu)$ satisfy $h = h_1 \cdots h_m > 0$.} Let $c = \min_{1\leq i, j \leq d} h(i,j) > 0$, where $h(i,j)$ is the $(i,j)$-th entry of $h$. Then, we have $h \geq cww^T$ where $w = (1, \cdots, 1)^T$. { For two vectors $v_1,v_2$, denote $v_1\otimes v_2 = v_1 v_2^T$.} Note that for $n \geq n_1$, with 
		$\|b \|_{1,1} =  \sum_{1\leq i,j \leq d}  b(i,j) $, 
		\begin{equation}
			hbh \geq c^2 ww^T bww^T
			= c^2 w \|b \|_{1,1}  w^T
			\geq \frac{c^2\langle v, bu\rangle}{|v|_\infty |u|_\infty} w\otimes w.
		\end{equation}
		Since $0 \leq g_1 \leq g_2$ implies that $r(g_1) \leq r(g_2)$, we get that for $n \geq n_1$, 
		\begin{equation}\label{lower-bd-raba-first}
			r(hbh) \geq \frac{c^2\langle v, bu\rangle}{|v|_\infty |u|_\infty} r(w \otimes w) .
		\end{equation}
		Combining~\eqref{lower-bd-vbw}, \eqref{lower-bd-raba-first}, and the fact that $r(w \otimes w) = d$, we obtain that for $n \geq n_1$, 
		\begin{equation}\label{lower-bd-raba-second}
			r(hbh) \geq \frac{\big(r(\mathbb E A)\big)^{n}c^2d}{2\varepsilon |v|_\infty |u|_\infty  } = \frac{\big(r(\mathbb E A)\big)^{-2m}c^2d}{2 \varepsilon |v|_\infty |u|_\infty} \big(r(\mathbb E A)\big)^{n+ 2m}.
		\end{equation} 
		Therefore, choosing $\varepsilon <  {\big(r(\mathbb E A)\big)^{-2m}c^2d} / ({2 R |v|_\infty |u|_\infty })$,  
		$n \geq n_1$ 
		(which depends on  $\varepsilon$) and letting $g = hbh$, $K = n + 2m$, we have $g > 0$ and
		\begin{equation}
			r(g) \geq R\big(r(\mathbb E A)\big)^{K}.
		\end{equation}		
		This finishes the proof.  
	\end{proof}
	We are now prepared to prove Proposition~\ref{prop::non-arithmeticity-implies-others}.
	\begin{proof}[Proof of Proposition~\ref{prop::non-arithmeticity-implies-others}]
		(1)  Proof of  \ref{cond::conditions_on_both_N_and_mu}  in Case 1. 
		Let $n_1 \in \mathrm{supp}(N)$ such that $n_1 \leq \mathbb{E}[N] $. Using the non-arithmeticity and the condition  that $\int \log_+ N(a)d\mu(a) < +\infty$, we deduce from Corollary~\ref{cor::gamma_is_strictly_smaller} that $\gamma < \log \frac{1}{\mathbb{E}[N]}$. Applying Lemma~\ref{lem::product_with_small_spectral_radius} with $\eta = 1/\mathbb{E}[N]$ and $\varepsilon = 1$, we know that there exist $K \geq 1$ and $a_1, \cdots, a_K \in \mathrm{supp}(\mu)$ such that $g := a_1 \cdots a_K$ satisfies $g > 0$, 
		\begin{equation}
			r(g) < \mathbb{E}[N]^{-K}.
		\end{equation}
		Let $l_1=n_1^Kg = n_1^K a_1 \cdots a_K $. Then, we have $l_1 > 0$, $r(l_1) < (n_1 / \mathbb{E}[N])^K \leq 1$, and $l_1$ can be written as
		\begin{equation}\label{l_1-equal-u-Nk}
			l_1 = \sum_{u\in (\mathbb{N}^*)^K} b_{u|1} \cdots b_{u|K}, 
		\end{equation}
		where 
		{ 	$(b_{v1}, b_{v2}, \cdots, b_{vn_1}, b_{v(n_1 + 1)}, \cdots) = (a_{|v| + 1}, a_{|v|+1}, \cdots , a_{|v|+1}, 0, \cdots)$ for $v\in   \bigcup_{ 0 \leq j \leq K - 1} (\mathbb{N}^*)^j$.}
		{ Using Condition~\ref{cond::conditions_on_independence} and Part 2 of  Lemma~\ref{lem::support_breaking_general}, we know $(n_1, b_{v1}, b_{v2}, \cdots) \in \mathrm{supp}(N, A_1, A_2, \cdots)$ for each $v \in \bigcup_{ 0 \leq j \leq K-1} (\mathbb{N}^*)^j$. }

		{ Since $r(\mathbb{E}A) \mathbb{E}[N] = \kappa(1) \mathbb{E}[N] = 1$ (cf. Condition~\ref{cond::conditions_on_alpha} with $\alpha = 1$), we have $r(\mathbb{E}A) = 1 / \mathbb{E}[N] > e^\gamma$. Using Lemma~\ref{lem::products_with_large_spectral_radius} instead of Lemma~\ref{lem::product_with_small_spectral_radius}} ,
		we can also construct a matrix $l_2 >0$ such that $r(l_2) > 1$ and $l_2$ can be written in a  form similar to~\eqref{l_1-equal-u-Nk}. This implies Condition~\ref{cond::conditions_on_both_N_and_mu} 
		in Case 1. 
		
		(2) Proof of  \ref{cond::conditions_on_both_N_and_mu}  in Case 2. The argument is similar. Indeed, 
		note that $\mu_Y$ also satisfies Condition~\ref{cond::condition_on_mu}. We can substitute $\mu$ with $\mu_Y$ in Lemma~\ref{lem::gamma_is_strictly_smaller} and Corollary~\ref{cor::gamma_is_strictly_smaller} to get that the Lyapunov exponent of $\mu_Y$ is strictly negative. Then, we can argue as in the proof of Lemmas~\ref{lem::product_with_small_spectral_radius} and~\ref{lem::products_with_large_spectral_radius} (with $\eta = 1$) 
		to conclude 
		that there exist $l_1', l_2' \in \Gamma(\mu_Y)$ which satisfy $l_1' > 0$, $l_2' > 0$ and $r(l_1') < 1 < r(l_2')$. { Using Part 1 of Lemma~\ref{lem::support_breaking_general} with $n = 1$, $X_1 = (N, A_1, A_2, \cdots)$, $f : (m, a_1, a_2\cdots) \mapsto \sum_{i = 1}^m a_i$,
			we see that $\mathrm{supp}(Y) = \mathrm{supp}(f(X_1)) \subset \overline{f(\mathrm{supp}(X_1))}$, so each $y' \in \mathrm{supp}(Y)$ can be approximated by $y$ of the form $y = a_1 + \cdots + a_n$, where $(n, a_1, \cdots, a_n, 0, \cdots) \in \mathrm{supp}(N, A_1, A_2, \cdots)$. Thus, $l_1'$ and $l_2'$ can be approximated by $l_1$ and $l_2$ of the forms similar to~\eqref{l_1-equal-u-Nk} which satisfy $l_1 > 0$, $l_2 > 0$ and $r(l_1) < 1 < r(l_2)$.}
		This shows that Condition~\ref{cond::conditions_on_both_N_and_mu} holds.
	\end{proof}

	\section{Support of solution: proofs of theorems}\label{section::support}
	
	
	We use the inclusion relation derived in Lemma~\ref{lem::support_breaking_general} to prove Theorems \ref{thm::second_theorem} and~\ref{thm::main_theorem}. We  also  prove Theorem~\ref{thm::third_theorem} by a different approach.

	
	\subsection{Proof of Theorems~\ref{thm::second_theorem} and \ref{thm::main_theorem}}
	In this section, we assume Conditions~\ref{cond::conditions_on_N}, \ref{cond::condition_on_mu}, \ref{cond::conditions_on_both_N_and_mu}. Let $l_1, l_2$ be two matrices as defined in Condition~\ref{cond::conditions_on_both_N_and_mu}, that is, there exist 
	two cover sets $U_1, U_2$ of finite levels such that
	\begin{equation}
		l_i = \sum_{u\in U_i} a_{u|1}^{(i)} \cdots a_{ u| |u|}^{(i)}, \quad i = 1,2,
	\end{equation}
	where $(a^{(i)}_{u1},a^{(i)}_{u2},\cdots) \in  \mathrm{supp}(A_1,A_2,\cdots)$, $u \in \mathbb{U}$,
	satisfying
	\begin{equation}
		l_1 > 0, \quad l_2 > 0, \quad r(l_1) < 1 , \quad r(l_2) > 1.
	\end{equation}
	By~\eqref{equ::support_breaking_easy_form}, we know that $\{ l_iz:z \in \mathrm{supp}(Z) \} \subset \mathrm{supp}(Z)$, $i \in \{1, 2 \}$. By the Perron-Frobenius theorem, for each $i \in \{ 1, 2 \}$, there exist two vectors $u_i > 0, v_i > 0$ and a matrix $q_i$ such that
	\begin{align} \label{def-vi} 
		& l_i = r(l_i)v_i \otimes u_i + q_i, \quad l_iv_i = r(l_i)v_i,\quad l_i^Tu_i = r(l_i)u_i, \\ \label{li-rli-vi-otimes-ui-qi}
		& \langle u_i, v_i\rangle = 1, \quad \langle v_i, v_i\rangle = 1, \quad r(q_i) < r(l_i).
	\end{align}
	where $v_i \otimes u_i := v_i u_i^T$ is a matrix with rank one. 
	
		Recall that for a matrix $g$ and a vector $x$, we write $g \cdot x = \frac{gx}{|gx|}$ if $gx \neq 0$. 
	From \cite{HennionLimit97}, there exists a distance $d$ on 
	$\mathbb{S}_+^{d-1}= \{ x \in \mathbb R_+^d: \| x \| = 1 \}$ such that 
	\begin{itemize}
		\item $\sup\{d(x, y) : x , y \in \mathbb{S}_+^{d-1}\} = 1$; 
		\item $|x-y| \leq 2d(x,y)$ for $x , y \in \mathbb{S}_+^{d-1}$.
		\item on  the subset 
		$  \{ x \in \mathbb S_+^{d-1}:  x>0 \}$
		of  $\mathbb{S}_+^{d-1}$, 
		the metric $d$ induces the same topology as that of the Euclidean metric. 
	\end{itemize}
	Moreover, for any nonnegative matrix $g$, there exists $c(g) \leq 1$ such that
	\begin{itemize}
		\item for all $x, y \in \mathbb{S}_+^{d-1}$, we have $d(g \cdot x, g \cdot y) \leq c(g) d(x, y)$;
		\item $c(g) < 1$ if and only if $g$ is strictly positive;
		\item if $g'$ is another nonnegative matrix, then $c(gg') \leq c(g)c(g')$.
	\end{itemize}
	
	\begin{lemma}\label{lem::support-contains-zero}
		Under Conditions~\ref{cond::conditions_on_N}, \ref{cond::condition_on_mu}, \ref{cond::conditions_on_both_N_and_mu}, we have $0 \in \mathrm{supp}(Z)$.
	\end{lemma}
	\begin{proof}
		Choose $z \in \mathrm{supp}(Z)$. Using~\eqref{equ::support_breaking_easy_form}, we have  $l_1^n z \in \mathrm{supp}(Z)$ for  all $n \geq 1$. It follows from $r(l_1) < 1$ that $0 = \lim_{n \to \infty} l_1^nz \in \mathrm{supp}(Z)$.
	\end{proof}

	\begin{lemma}\label{lem::g1g2}
		Assume Conditions~\ref{cond::conditions_on_N}, \ref{cond::condition_on_mu},  \ref{cond::conditions_on_both_N_and_mu}. Then, there exists $\delta > 0$ such that for all $\varepsilon > 0$, there are two matrices $g_1, g_2$ which possess  the following properties:
		\begin{enumerate}
			\item[1.] $\delta v_2\otimes u_1 \leq g_i \leq \delta^{-1} v_2 \otimes u_1 $ for $i = 1, 2$; 
			\item[2.]  $r(g_1) \in [\frac{1}{2},1)$; 
			\item[3.]   for all $z \in \mathbb{S}_+^{d-1}$, we have $d(g_1\cdot z, v_2) \leq \varepsilon$;
			\item[4.]  $  \big\{ g_1 z_1 + g_2z_2 : z_1, z_2 \in \mathrm{supp}(Z) \big\} \subset \mathrm{supp}(Z)$.
		\end{enumerate}
	\end{lemma}
	\begin{proof}
		Since $\mathbb{E}[N] > 1$, we can take $(k, (a_i)_{i \geq 1}) \in \mathrm{supp}(N,(A_i)_{i \geq 1})$ with $k \geq 2$. Let $J \geq 1$ be an integer. Define
		\begin{equation}
			g_\eta = a_{\eta(1)}a_{\eta(2)}\cdots a_{\eta(J)}, \quad \mbox {if } \eta = (\eta(1), \cdots, \eta(J)) \in \{ 1, \cdots, k \}^J. 
		\end{equation}
		
		(1)
		We first prove that: we can find a constant $C > 0$, such that for any $t > 0$, there exist an integer $n_0(t) \geq 1$ and a set of positive integers $\{ m(n, \eta,t) \}$, which satisfy 
		\begin{equation}\label{v2-otimes-u1-le}
			t \,v_2 \otimes u_1  \leq l_2^n g_\eta l_1^{m(n, \eta, t)}\leq Ct \,v_2 \otimes u_1, \quad \forall n \geq n_0(t), \quad \forall \eta \in \{ 1, \cdots, k \}^J.
		\end{equation}
		
		Notice that there exists $c > 1$ such that
		\begin{align}\label{C-1rlinvi}
			c^{-1} r(l_i)^n v_i\otimes u_i \leq l_i^n \leq cr(l_i)^n v_i\otimes u_i, \quad n \geq 1, \quad i = 1, 2.
		\end{align}
		Indeed, 
		let $\delta = (r(l_i) - r(q_i)) / 2 > 0$. Then for large enough $n$, we have $\|q_i^n\|^{1/n} \leq r(q_i) + \delta$. Therefore 
		\begin{equation}\label{equ::ad_hoc_17}
			\lim_{n \to \infty}\frac{\|q_i^n\|}{r(l_i)^n} \leq \lim_{n \to \infty}\left( \frac{ r(q_i) + \delta}{r(l_i)} \right)^n = 0.
		\end{equation}
		Using the fact that $l_i^n = r(l_i)^n v_i \otimes u_i + q_i^n$ for $n \geq 1$, this implies~\eqref{C-1rlinvi}.
		
		It follows from ~\eqref{C-1rlinvi} that, for all $\eta \in \{ 1, \cdots, k \}^J$, $m \geq 1$ and $n \geq 1$,
		\begin{align}\label{l2ng-eta-l1m-less}
			l_2^n g_\eta l_1^m & \leq c^2 r(l_1)^mr(l_2)^n (v_2\otimes u_2) g_\eta (v_1\otimes u_1) \notag \\
			& = c^2 \langle u_2, g_\eta v_1\rangle r(l_1)^mr(l_2)^n v_2\otimes u_1.
		\end{align}
		Similarly, for all $\eta \in \{ 1, \cdots, k \}^J$, $m \geq 1$ and $n \geq 1$,
		\begin{equation}\label{l2ng-eta-l1m-greater}
			l_2^n g_\eta l_1^m \geq c^{-2}\langle u_2, g_\eta v_1\rangle r(l_1)^mr(l_2)^n v_2\otimes u_1.
		\end{equation}
		
		Fix $c > \max(r(l_1)^{-1}, r(l_2))$. Since $r(l_1) < 1 < r(l_2)$, we notice that for any  number $s > 0$ and any integer $n$  with $ r(l_2)^n > cs$, there exists an integer $m \geq 1$ such that
		$$r(l_1)^m r(l_2)^n \in [s, cs]$$ 
		(we can choose the first $m \geq 1$ satisfying  $r(l_1)^m r(l_2)^n \leq cs$). Thus, for all $t > 0$ and $\eta \in \{ 1, \cdots, k \}^J$, there exist an integer $n_0(t, \eta) \geq 1$ and a set of positive integers $\{ m(n, \eta, t) \}_{n \geq n_0(t, \eta)}$, such that
		\begin{equation}\label{u2g-eta-v1}
			r(l_1)^{m(n, \eta, t)} r(l_2)^n \in \Big[\frac{c^2t}{\langle u_2, g_\eta v_1\rangle}, \frac{c^3t}{\langle u_2, g_\eta v_1\rangle}\Big], \quad \forall n \geq n_0(t, \eta), 
		\end{equation}
		by applying the previous inequality with $s = \frac{c^2 t}{\langle u_2, g_\eta v_1\rangle}$.
		Combining \eqref{l2ng-eta-l1m-less}, \eqref{l2ng-eta-l1m-greater} and \eqref{u2g-eta-v1},  and writing   
		$n_0(t) = \max \{ n_0(t,\eta): \eta  \in  \{ 1, \cdots, k \}^J\}$ and  $C = c^5$,   we get that for all $t > 0$, 
		\begin{equation}
			t \, v_2 \otimes u_1 \leq l_2^n g_\eta l_1^{m(n, \eta, t)}\leq C t \, v_2 \otimes u_1, \quad \forall n \geq n_0(t), \quad \forall \eta \in \{ 1, \cdots, k \}^J.
		\end{equation}			
		This ends  the proof of~\eqref{v2-otimes-u1-le}.
		
		(2) We next construct $g_1, g_2$. Fix $t > 0$ and $J \in \mathbb{N}^*$ such that
		\begin{equation}
			0 < t < \frac{1}{2C(1+C)\langle v_2, u_1 \rangle}, \quad k^Jt > \max\big\{ \frac{2}{\langle v_2, u_1 \rangle}, 2 \big\},
		\end{equation}
		and write,  with $\varepsilon >0$ given in the lemma,  
		\begin{equation}
			n = \max\big\{ { \big\lfloor \frac{\log \varepsilon}{\log c(l_2)} \big\rfloor} + 1, n_0(t) \big\}.
		\end{equation}
		Denote by $|I|$ the cardinality of a set $I$. We take a sequence $\{ I(i) \}_{0 \leq i \leq k^J}$ of subsets of $\{ 1, \cdots, k \}^J$ such that $I(0)= \emptyset$, $I(i-1) \subset I(i)$ and $|I(i) \setminus I(i - 1)| = 1$ for all $i = 1, \cdots, k^J$. Therefore $|I(i)|= i$ for all $i=0, 1, \cdots, k^J$. Define
		\begin{align}
			h_\eta & = l_2^n g_\eta l_1^{m(n, \eta, t)}, \quad \forall  \eta \in \{ 1,\cdots k \}^J, \\
			W & = \Big\{ i \geq 1 : r\big(\sum_{\eta \in I(i)} h_\eta \big) \geq \frac{1}{2} \Big\}, \\
			i_0 & = \min\{ \inf W, k^J \},
		\end{align}
		with the convention that $\inf \emptyset = +\infty$. Let
		\begin{align}
			g_1 = \sum_{\eta \in I(i_0)}h_\eta, \quad g_2 = \sum_{\eta \in \{ 1, \cdots, k \}^J \setminus I(i_0)}h_\eta,\quad 
			\delta = \min(t, \frac{1}{Ctk^J}).
		\end{align}
		
		We claim that $2 \leq i_0 \leq k^J - 1$, so that $\emptyset \neq  I(i_0) \neq \{ 1, \cdots, k \}^J$, and thus  
		$g_1 > 0$, $g_2 > 0$. Using~\eqref{v2-otimes-u1-le}, for any subset $I \subset \{ 1, \cdots, k \}^J$, we have   
		\begin{equation}\label{I-t-v2-u1-bound}
			|I| t \, v_2 \otimes u_1 \leq \sum_{\eta \in I} h_\eta \leq C|I|t \, v_2 \otimes u_1.
		\end{equation}
		In particular, taking $I = I(i_0)$ and $I=\{ 1, \cdots, k \}^J \setminus I(i_0)$, we get 
		\begin{equation}\label{I-t-v2-u1-bound-g1-g2}
			i_0 t \, v_2 \otimes u_1 \leq g_1 \leq Ct i_0 \, v_2 \otimes u_1, \quad 
			(k^J- i_0)  t \, v_2 \otimes u_1 \leq g_2 \leq C(k^J- i_0)t  \,  v_2 \otimes u_1. 
		\end{equation}
		Since $k^J t > 2/\langle v_2,u_1\rangle$, we deduce from~\eqref{v2-otimes-u1-le} that 
		\begin{equation}
			r\big(\sum_{\eta \in I(k^J - 1)} h_\eta\big) \geq (|I(k^J)| - 1)t \cdot  r(v_2 \otimes u_1) = (k^J - 1) t \langle v_2,u_1\rangle > \frac{k^J}{2}t \langle v_2,u_1\rangle > 1.   
		\end{equation}
		By the definition of $W$, this implies that $k^J - 1 \in W$.
		Since $t < 1/(2C(C+1)\langle v_2,u_1\rangle)$, we can also deduce from~\eqref{v2-otimes-u1-le} that
		\begin{equation}
			r\big( \sum_{\eta \in I(1)} h_\eta \big) < r(Ct \, v_2 \otimes u_1) < \frac{1}{2}.
		\end{equation}
		By the definition of $W$, this implies that $1 \not \in W$. Thus, by the definition of $i_0$, we see that $2 \leq i_0 \leq k^J-1$, so that $g_1 > 0$, $g_2 > 0$.

		(3) We then prove all the listed properties of $g_1, g_2$ as follows. 
		
		1. By~\eqref{I-t-v2-u1-bound} and the definition of $\delta$,
		\begin{equation}\label{bound-g1}
			\delta \, v_2 \otimes u_1 \leq t \, v_2 \otimes u_1 \leq g_1 \leq k^J C t  \, v_2 \otimes u_1 \leq  \delta^{-1} \, v_2\otimes u_1.
		\end{equation}
		Similarly, the same inequality also holds for $g_2$. 
		
		2. By the definition of $W$ and $i_0$,
		\begin{equation}\label{r-sum-I-i0-greater}
			r(g_1) = r\big( \sum_{\eta \in I(i_0)} h_\eta \big) \geq \frac{1}{2}.
		\end{equation}
		So it remains to show that
		\begin{equation}
			r(g_1) = r\big( \sum_{\eta \in I(i_0)} h_\eta \big) < 1.
		\end{equation} Using 
		\eqref{I-t-v2-u1-bound-g1-g2} and \eqref{r-sum-I-i0-greater}
		we obtain
		\begin{equation}
			\frac{1}{2} \leq r(i_0 C t \, v_2 \otimes u_1) = i_0Ct \langle v_2,u_1\rangle,
		\end{equation}
		which yields that $i_0 \geq 1/(2Ct\langle v_2,u_1\rangle)>1+C$, hence $\frac{C}{i_0-1} < 1$. Thus,
		\begin{align}
			\sum_{\eta \in I(i_0)\setminus I(i_0 - 1)}h_\eta & \leq Ct \, v_2\otimes u_1  = \frac{C}{i_0-1} \cdot (i_0-1) t \, v_2\otimes u_1 <  \sum_{\eta \in I(i_0-1)}h_\eta. 
		\end{align}
		This implies
		\begin{equation}
			\sum_{\eta \in I(i_0)}h_\eta < 2 \sum_{\eta \in I(i_0-1)} h_\eta.
		\end{equation}
		Hence using the fact that $i_0-1 \not \in W$, we derive that
		\begin{equation}\label{r-sum-I-i0-le}
			r(g_1) = r\big( \sum_{\eta \in I(i_0)}h_\eta \big) < 2  \, r\big( \sum_{\eta \in I(i_0-1)} h_\eta \big) < 2\cdot \frac{1}{2} = 1.
		\end{equation}
		
		3. Note that for any $z \in \mathbb{S}_+^{d-1}$, 
		\begin{equation}
			d(g_1 \cdot z, v_2) = d\big( (l_2^n \sum_{\eta \in I(i_0)} g_\eta l_1^{m(n,\eta)})\cdot z, l_2^n \cdot v_2 \big) \leq c(l_2)^n d((\sum_{\eta \in I(i_0)} g_\eta l_1^{m(n,\eta)})\cdot z, v_2) \leq c(l_2)^n, \quad 
		\end{equation}
		Using the definition of $n$ and the fact that $c(l_2) < 1$, we have $d(g_1 \cdot z, v_2) \leq c(l_2)^n \leq \varepsilon$.
		
		4. Using~\eqref{equ::support_breaking_cover_set} and~\eqref{equ::support_breaking_easy_form}, we derive that
		\begin{align}\label{suppZ-supset-l2n-sum}
			\mathrm{supp}(Z) & \supset \{ l_2^n z : z \in \mathrm{supp}(Z) \} \notag \\
			& \supset \bigg\{ l_2^n \sum_{\eta \in \{1, \cdots, k\}^J} g_\eta z_\eta :  z_\eta \in \mathrm{supp}(Z) \bigg\} \notag \\
			& \supset \bigg\{ l_2^n \sum_{\eta \in \{1, \cdots, k\}^J} g_\eta l_1^{m_\eta} z_\eta : m_{\eta } \in \mathbb{N}, z_\eta \in \mathrm{supp}(Z) \bigg\}.
		\end{align}
		In the last display, using the decomposition  $ \{1, \cdots, k\}^J = I(i_0) \cup 
		( \{1, \cdots, k\}^J  \setminus I(i_0))$,  taking $z_\eta = z_1 $ for $\eta \in I(i_0)$ and  
		$z_\eta = z_2 $ for  $\eta \in  \{1, \cdots, k\}^J  \setminus I(i_0)$, and using 
		the definition of $g_1, g_2$,  we see that  \eqref{suppZ-supset-l2n-sum}  implies that
		\begin{equation}
			\mathrm{supp}(Z) \supset \{ g_1 z_1 + g_2 z_2 : z_1, z_2 \in \mathrm{supp}(Z) \}.
		\end{equation}
		

		%
		%


	\end{proof}

	The following technical lemma is about the dyadic expansion with base $\theta \in \left[ \frac{1}{2}, 1\right)$. When $\theta = \frac{1}{2}$, the result is well-known. 
	\begin{lemma}\label{lem::denseness_dyadic_representation}
		For any $\theta \in \left[ \frac{1}{2}, 1 \right)$, we have
		\begin{equation} \label{def-W} 
			\Big\{ \sum_{n \geq 1}\eta_n \theta^n : \eta_n = 0 \text{ or } 1 \Big\} 
			= \Big[0, \frac{\theta}{1-\theta}\Big]. 
		\end{equation}
	\end{lemma}
	\begin{proof} Let $E$ be the set defined by the left hand side of  \eqref{def-W}. 
		Note that $\frac{\theta}{1-\theta} = \sum_{n \geq 1} \theta^n$, so that $E \subset 
		\big[0, \frac{\theta}{1-\theta}\big]$ and  $\frac{\theta}{1-\theta}  \in E$. We now prove the reverse inclusion $
		\big[0, \frac{\theta}{1-\theta}\big] \subset E$.  
		Let  $x \in [0, \frac{\theta}{1-\theta}]$. We will prove that $x \in E$. 
		Since we have seen that  $\frac{\theta}{1-\theta}  
		\in E$,
		we only consider $x \in [0, \frac{\theta}{1-\theta})$.  
		Define $(x_n)_{n \geq 0}$ and $(\eta_n)_{n \geq 1}$ as follows:
		\begin{enumerate}
			\item Let $x_0 = 0$. 
			\item For each $n \geq 1$, suppose that $x_{n - 1}$ is defined.  We will define $x_n$ and  $\eta_n$. If $x_{n - 1} + \theta^n \leq x$, we let $x_n = x_{n - 1} + \theta^n$ and $\eta_n = 1$; otherwise we let $x_n = x_{n - 1}$ and $\eta_n = 0$.
		\end{enumerate}
		
		Then, by reduction, we have 
		$$x_n \leq x \quad  \forall  n \geq 0,  \quad \mbox{ and }  \quad x_n = \sum_{i = 1}^n \eta_i \theta^i 
		\quad   \forall n \geq 1.$$ 
		
		Note that the set $I := \{ n \geq 1 : \eta_n = 0 \}$ is not empty. We claim that $|I| = \infty$, where  $|I| $ denotes the cardinality of $I$. Otherwise, we take the maximal $n \in I$. Then, we have $x \in [x_{n - 1}, x_{n-1}+\theta^n)$. But, 
		\begin{equation}
			x \geq x_{n - 1} + \sum_{ j > n} \theta^j = x_{n - 1} + \frac{\theta}{1-\theta} \cdot \theta^{n} \geq x_{n - 1} + \theta^n.
		\end{equation}
		This contradicts with the fact that $x \in [x_{n - 1}, x_{n-1}+\theta^n)$. So the claim follows. 
		
		As $(x_n)_{n \geq 1}$ is increasing and $x \in [x_{n - 1}, x_{n-1}+\theta^n)$ for $n \in I$, we have
		\begin{equation}
			x = \lim_{n \in I, n \to \infty} x_{n - 1} = \lim_{n \to \infty} x_{n- 1} = \sum_{n \geq 1} \eta_n \theta^n.
		\end{equation}
		So $x \in E$, and the proof is finished. 
	\end{proof}
	
	\begin{remark} By a proof  similar to that of Lemma~\ref{lem::denseness_dyadic_representation}, we can 
		prove  the following more general result  about  the $p$-adic expansion: if $p > 1$ is a real number, then
		\begin{equation}
			\Big\{ \sum_{n \geq 1} \eta_n p^{-n} : \eta_n = 0, 1, \cdots, \lceil p\rceil - 1 \Big\} = \Big[0,  \frac{\lceil p \rceil - 1}{p-1}\Big].
		\end{equation}
		When $p>1$ is an integer, this result is well-known.  Lemma~\ref{lem::denseness_dyadic_representation} concerns the case where $p=1/\theta \in (1, 2]$. 
	\end{remark}

	For the following statement,  
	recall that $v_2$ is the Perron-Frobenius right eigenvector of $l_2$ with unit norm, defined in \eqref{li-rli-vi-otimes-ui-qi}.
	\begin{lemma}\label{lem::contains_one_ray}
		Under Conditions~\ref{cond::conditions_on_N}, \ref{cond::condition_on_mu}, \ref{cond::conditions_on_both_N_and_mu},  we have
		\begin{equation}
			\mathrm{supp}(Z) \supset \mathbb{R}_+v_2 := \{ sv_2 : s \geq 0 \}.
		\end{equation}
	\end{lemma}
	\begin{proof}
		Recall that  there exists $\delta > 0$ such that for any $\varepsilon >0$, there are two matrices $g_1, g_2 > 0$ which satisfy  the four properties  listed in 
		Lemma~\ref{lem::g1g2}. 
		

		(1) We first prove that
		\begin{equation}\label{supp-Z-subset-sum-m=0-n}
			\mathrm{supp}(Z) \supset \bigg\{\sum_{m = 0}^{\infty} \eta_m g_1^m(g_2z) : \eta_m = 0 \mbox{ or } 1,  z \in \mathrm{supp}(Z) \bigg\}.
		\end{equation}
		Since $r(g_1) < 1$, the series in~\eqref{supp-Z-subset-sum-m=0-n} converges for any choice of $\eta_m$ and $z \in \mathrm{supp}(Z)$. 
		In fact we will prove the following  result slightly stronger than 
		\eqref{supp-Z-subset-sum-m=0-n}: 
		\begin{equation}\label{I-n-sum-m=0-n}
			I_n  : = \bigg\{\sum_{m = 0}^{n} \eta_m g_1^m(g_2z_m) : z_m \in \mathrm{supp}(Z), \eta_m = 0\mbox{ or } 1 \bigg\} \subset \mathrm{supp}(Z), \quad \forall n \geq 0.
		\end{equation}
		(Here $z_m$ may vary with $m$, while in  \eqref{supp-Z-subset-sum-m=0-n} $z$ does not depend on $m$.)
		We will prove this by induction. 
		By Lemma~\ref{lem::support-contains-zero}, we have $0 \in \mathrm{supp}(Z)$. 
		{ Combining this with the condition that $  \big\{ g_1 z_1 + g_2z_2 : z_1, z_2 \in \mathrm{supp}(Z) \big\} \subset \mathrm{supp}(Z)$ (cf. Part 4 of Lemma~\ref{lem::g1g2}), }
		we derive $I_0 \subset \mathrm{supp}(Z)$.

		Suppose that $I_n\subset \mathrm{supp}(Z)$ for some $n \geq 0$. Then, for all $\eta_m \in \{0, 1\}$ and $z_m \in  \mathrm{supp}(Z)$ for $m = 0, 1, \cdots, n+1$, we have
		\begin{align}
			\sum_{m = 0}^{n+1} \eta_m g_1^m(g_2z_m) 
			& =   g_1\big( \sum_{m=0}^{n} \eta_{m+1} g_1^m g_2z_{m+1} \big) + \eta_0g_2z_0   \notag \\
			& \in \{ g_1z_1 + g_2 z_2 : z_1 \in I_n,z_2 \in \mathrm{supp}(Z) \} \notag \\& 
			\subset \{ g_1z_1 + g_2 z_2 : z_1 \in \mathrm{supp}(Z),z_2 \in \mathrm{supp}(Z) \} \notag \\
			& \subset \mathrm{supp}(Z),
		\end{align}
		where in the last inclusion we use Part 4 of Lemma~\ref{lem::g1g2}. This implies $I_{n+1} \subset  \mathrm{supp}(Z)$. 
		Therefore, by  induction,  $I_{n} \subset  \mathrm{supp}(Z)$ for all $n \geq 0$. This gives  \eqref{I-n-sum-m=0-n}. 
		Letting $z_m = z \in \mathrm{supp}(Z)$ in~\eqref{I-n-sum-m=0-n} and taking  limit as  $n \to \infty$, we deduce that~\eqref{supp-Z-subset-sum-m=0-n} holds. 
		
		(2) We then prove that $\mathrm{supp}(Z) $ contains a segment. 
		To this end, we will 
		prove that  $\mathrm{supp}(Z) $ is dense in a segment of the form $[0, c] v_2$,
		using Lemma~\ref{lem::denseness_dyadic_representation} with $\theta = r(g_1)$. 
		
		We fix a nonzero vector $z \in \mathrm{supp}(Z)$. 
		Then, by the result proved in (1) above, 
		$\sum_{m = 0}^\infty \eta_m g_1^m(g_1g_2z)  \in \mathrm{supp}(Z)$.   Notice that,  
		with  $t : = |g_1g_2z|$, 
		\begin{align}\label{sum-m-0-infty-eta-m-g-1-m-g1-g2-z-sum-m-0}
			& \left| \sum_{m = 0}^\infty \eta_m g_1^m(g_1g_2z) - \big( \sum_{m=0}^\infty \eta_mr(g_1)^m \big) tv_2 \right| \notag \\
			& \leq \left| \sum_{m = 0}^\infty \eta_m g_1^m(g_1g_2z) - \sum_{m = 0}^\infty \eta_m g_1^m(tv(g_1))\right|   
			+ \left| \sum_{m = 0}^\infty \eta_m g_1^m(tv(g_1)) - \big( \sum_{m=0}^\infty \eta_mr(g_1)^m \big) tv_2 \right| \notag \\
			& \leq \left| \sum_{m = 0}^\infty \eta_m g_1^m(g_1g_2z - tv(g_1)) 
			\right| + \frac{|v(g_1)-v_2|t}{1-r(g_1)}.
		\end{align} 
		
		Using the inequality $d(g_1\cdot w, v_2) < \varepsilon$ for $w \in \mathbb{S}_+^{d-1}$ (cf. Part 3 of Lemma~\ref{lem::g1g2}), we get $d(v(g_1), v_2) = d(g_1 \cdot v(g_1), v_2) < \varepsilon$;  hence by the triangle inequality of the distance $d$, we obtain 
		\begin{equation}\label{d-g-1-cdot-z}
			d(g_1\cdot w, v(g_1)) < 2\varepsilon, \quad \forall w \in \mathbb{S}_+^{d-1}.
		\end{equation}
		Using the fact that $\delta v_2\otimes u_1 \leq g_i \leq \delta^{-1}v_2 \otimes u_1 $ for $i = 1, 2$ (cf. Part 1 of Lemma~\ref{lem::g1g2}), we know  that 
		\begin{equation}\label{t-le-delta--2-v-2}
			t \leq C := \delta^{-2}|v_2|\langle v_2,u_1\rangle \langle u_1,z\rangle.
		\end{equation}
		Then, by the inequality  $|x-y| \leq 2d(x,y)$ for any $x, y\in\mathbb{S}_+^{d-1}$ and~\eqref{d-g-1-cdot-z},
		\begin{equation}\label{g1g2z-tvg1-t}
			|g_1g_2z - tv(g_1)| = t \left| \frac{g_1g_2z}{t} - v(g_1) \right| \leq t \cdot 2d\big( \frac{g_1g_2z}{t}, v(g_1) \big) \leq 4t\varepsilon \leq  4C\varepsilon,
		\end{equation}
		Note that $|v(g_1) - v_2| \le 2d(v(g_1), v_2) < 2\varepsilon$. Plugging~\eqref{g1g2z-tvg1-t} into~\eqref{sum-m-0-infty-eta-m-g-1-m-g1-g2-z-sum-m-0}, we get
		\begin{align}\label{sum-m-0-infty-eta}
			& \left| \sum_{m = 0}^\infty \eta_m g_1^m(g_1g_2z) - \big( \sum_{m=0}^\infty \eta_mr(g_1)^m \big) tv_2 \right| \notag \\
			&\leq \big( \sum_{m=0}^{\infty}\|g_1^m\|\big) 4C\varepsilon + \frac{|v(g_1)-v_2|t}{1-r(g_1)} \notag \\
			& \leq \bigg( \big( \sum_{m=0}^{\infty}\|g_1^m\|\big) 4C + \frac{2C}{1-r(g_1)} \bigg)\varepsilon  =: C_1 \varepsilon,
		\end{align}
		where $\sum_{m = 0}^\infty \|g_1^m\| < +\infty$  because $r(g_1) = \lim_{n}\|g_1^m\|^{1/m} < 1$. 
		
		Note that 
		\begin{equation}
			t = |g_1g_2z| \geq \delta^2 |v_2|\langle v_2,u_1\rangle\langle u_1,z\rangle =: c_1.
		\end{equation}
		By Lemma~\ref{lem::denseness_dyadic_representation} and the condition that  $r(g_1) \in [\frac{1}{2},1)$ (cf. Lemma~\ref{lem::g1g2}(2)), 
		\begin{equation}
			\big\{ \big( \sum_{m=0}^\infty \eta_mr(g_1)^m \big) tv_2 : \eta_m = 0\mbox{ or } 1 \big\}  = \big[0,\frac{t}{1-r(g_1)}\big] v_2 \supset [0,c_2] v_2,
		\end{equation}
		where $c_2 := c_1/(1-r(g_1)) > 0$. Using~\eqref{sum-m-0-infty-eta}, { we deduce from $\sum_{m = 0}^\infty \eta_m g_1^m(g_1g_2z)  \in \mathrm{supp}(Z)$} that for any $x \in [0,c_2]v_2$, we have $B(x, C_1\varepsilon) \cap \mathrm{supp}(Z) \ne \emptyset$. 
		Since $ \mathrm{supp}(Z) $ is closed and $\varepsilon > 0$ is arbitrary, this implies that $x \in \mathrm{supp}(Z)$.
		Therefore 
		\begin{equation} \label{segment0c2v2} 
			[0,c_2]v_2 \subset \mathrm{supp}(Z).
		\end{equation}
		
		(3) Using~\eqref{equ::support_breaking_easy_form}  (with $l=l_2$) repeatedly, we derive from~\eqref{segment0c2v2} that $[0, c_2 r(l_2)^n] \subset \mathrm{supp}(Z)$ for all $n \geq 1$. If follows that $\mathbb{R}_+ v_2 \subset \mathrm{supp}(Z)$. 
	\end{proof}
	
	We will need the following lemma about the strict monotonicity of the spectral radius  to prove Theorem~\ref{thm::main_theorem}.
	
	\begin{lemma}[Monotonicity of the spectral radius]\label{lem::comparison_spectral_radius}
		Let $a, b$ be two nonnegative $d \times d$ matrices.
		If $a + b > 0$ and $b \neq 0$,
		then $r(a) < r(a + b)$.
	\end{lemma}
	\begin{proof}
		Note that $(a + b)^3 \geq a^3 + (a+b)b(a+b) > a^3$, which implies that $r((a+b)^3) > r(a^3)$. Since $r(a^n) = r(a)^n$ for $n \geq 1$, we get $r(a+b) > r(a)$. 
	\end{proof}
	
	We are ready to prove Theorem~\ref{thm::main_theorem}.
	\begin{proof}[Proof of Theorem~\ref{thm::main_theorem}]
		
		We fix a cover set $U \subset \mathbb{U}$ with finite level, and a family  $(a_{ui})_{i \geq 1} \in \mathrm{supp}((A_i)_{i \geq 1})$, $u \in \mathbb{U}$. Denote
		\begin{align}
			g_u & = a_{u|1}\cdots a_{ u| |u|}, \quad \forall u \in \mathbb{U}, \\
			g_{V} & = \sum_{u \in V} g_{u}, \quad \forall V \subset \mathbb{U}. 
		\end{align} 
		
		(1) We first prove that the support contains many rays. From Lemma~\ref{lem::support-contains-zero}, we know that $0 \in \mathrm{supp}(Z)$. With \eqref{equ::support_breaking_cover_set}, this implies that for $V \subset U$, 
		\begin{equation}\label{equ::g-V-z-in-supp-Z}
			g_{V} z \in \mathrm{supp}(Z), \quad \forall z \in \mathrm{supp}(Z).
		\end{equation}
		
		From Lemma~\ref{lem::contains_one_ray}, we know that $\mathbb{R}_+ v_2 \subset \mathrm{supp}(Z)$. Using this and ~\eqref{equ::g-V-z-in-supp-Z}, we see that  for any $V \subset U$, we have $\mathbb{R}_+ (g_V^nv_2)  \subset \mathrm{supp}(Z)$. When $g_V > 0$, we have $d(g_V^n \cdot v_2, v(g_V)) 
		= d(g_V^n \cdot v_2,  g_V^n \cdot v(g_V)) \leq c(g_{V})^n
		\to 0$ as $n \to \infty$. Hence we get
		\begin{equation}\label{mathbb-R-v-gV-in-supp}
			\mathbb{R}_+ v(g_V) \subset \mathrm{supp}(Z), \quad \forall V \subset U \mbox{ such that } g_V > 0.
		\end{equation}
		
		(2) We then establish an inclusion relation of the form $\sum_{ V \subset U} I_{V, n}\subset \mathrm{supp}(Z)$, $n \geq 1$, where $I_{V, n}$ will be defined below, which satisfies $\mathbb{R}_+ v(g_V) \subset I_{V, n}$ for each $V \subset U$ such that $g_V > 0$. From this, we establish the first statement of Theorem~\ref{thm::main_theorem}. 
		
		Let $\tilde{U} = \{ u \in U : g_u \ne 0 \}$. Note that $\tilde{U}$ is a finite set, since $N < \infty$ a.s. 
		For any $V \subset \tilde{U}$ such that $g_V > 0$,  and any $n \geq 1$, we define
		\begin{equation}
			I_{V,n} := \bigg\{ \sum_{\substack{u^i \in \tilde{U}, 1\leq i\leq n \\ \cup_i \{ u^i \} = V}} g_{u^1}g_{u^2}\cdots g_{u^n} z_{u^1\cdots u^n} : z_{u^1\cdots u^n} \in \mathrm{supp}(Z) \bigg\}.
		\end{equation}
		Using~\eqref{equ::support_breaking_cover_set} and the fact that $0 \in \mathrm{supp}(Z)$, we know that  $I_{V,n} \subset \mathrm{supp}(Z)$ for any $n \geq 1$ and $V \subset \tilde{U}$. 
		
		Note that  for $n \geq 1$, 
		\begin{equation}
			g_V^n = {\color{black} \Big( \sum_{u \in V} g_u \Big)^n} =  \sum_{\substack{u^i \in V, 1\leq i \leq n}} g_{u^1}\cdots g_{u^n}, \quad \forall V \subset \tilde{U}.
		\end{equation}
		Using this together with the definition of $I_{V,n}$ and the fact that $\mathbb{R}_+v(g_V) \subset  \mathrm{supp}(Z)$ (see~\eqref{mathbb-R-v-gV-in-supp}), we have for any $n \geq 1$, $s \geq 0$, $V \subset \tilde{U}$,
		\begin{align}\label{gVn-svgv}
			&g_V^n (sv(g_V)) - \sum_{\substack{u^i \in V, 1\leq i \leq n \\ \cup \{ u^i \} \subsetneqq V}} g_{u^1}\cdots g_{u^n} (sv(g_V)) \nonumber  \\ 
			&=   \sum_{\substack{u^i \in V, 1\leq i\leq n \\ \cup_i \{ u^i \} = V}} g_{u^1}\cdots g_{u^n}  (sv(g_V)) 
			=: v_n (V)   	\in I_{V,n}.
		\end{align}
		Here and after,  $A \subsetneqq B$ means that $A$ is a proper subset of $B$ (i.e. $A \subset B$, $A \ne B$).
		Then, 
		\begin{align}
			0 & \leq \sum_{\substack{u^i \in V, 1\leq i \leq n \\ \cup \{ u^i \} \subsetneqq V}} g_{u^1}\cdots g_{u^n} (sv(g_V)) \notag \\& \leq \sum_{W \subsetneqq V} \sum_{\substack{u^i \in W, 1\leq i \leq n}} g_{u^1}\cdots g_{u^n} (sv(g_V)) \notag \\
			& = \sum_{W \subsetneqq V} g_W^n sv(g_V).
		\end{align}
		By the definition of $\tilde{U}$ and the fact that $V \subset \tilde{U}$, we know that 
		for $W \subsetneqq V$, we have 
		$g_V - g_W \geq 0$  and  $g_V - g_W \neq 0$. 
		It follows from Lemma~\ref{lem::comparison_spectral_radius} that $r(g_W) < r(g_V)$. For each $t \geq 0$, let $s = s(n, t, V) = \frac{t}{r(g_V)^n}$. Using the fact that $r(g_W) < r(g_V)$ for $W \subsetneqq V$, we see that $\frac{\|g_W^n\|}{r(g_V)^n} \to 0$;  hence
		\begin{equation}
			\sum_{W \subsetneqq V} g_W^n sv(g_V) = \sum_{W \subsetneqq V} \frac{g_W^n t}{r(g_V)^n}v(g_V) \to 0, \quad n \to \infty.
		\end{equation}
		Using this 
		and \eqref{gVn-svgv} with $s = \frac{t}{r(g_V)^n}$, we derive that for each $t \geq 0$, 
		as $n\rightarrow \infty$, 
		\begin{equation}
			tv(g_V) - v_n (V) \rightarrow 0, \mbox{ where }  v_n (V)  \in   I_{V, n}.  
		\end{equation}
		
		By~\eqref{equ::support_breaking_cover_set},
		\begin{align}\label{sum-V-subset-tilde-U}
			\sum_{V \subset \tilde{U}} I_{V,n} & = \sum_{V \subset \tilde{U}} \bigg\{ \sum_{\substack{u^i \in \tilde{U}, 1\leq i\leq n \\ \cup_i \{ u^i \} = V}} g_{u^1}g_{u^2}\cdots g_{u^n} z_{u^1\cdots u^n} : z_{u^1\cdots u^n} \in \mathrm{supp}(Z) \bigg\} \notag \\
			& = \bigg\{ \sum_{\substack{u^i \in \tilde{U}, 1\leq i\leq n}} g_{u^1}g_{u^2}\cdots g_{u^n} z_{u^1\cdots u^n} : z_{u^1\cdots u^n} \in \mathrm{supp}(Z) \bigg\} \notag \\
			& \subset \mathrm{supp}(Z).
		\end{align}
		Notice that, since $\tilde{U}$ is finite, we have, for each 
		sequence  $ \{ t_V : V \subset \tilde{U} \}$ in $\mathbb R_+$,   as  $n \rightarrow \infty$,  
		\begin{align}
			\sum_{\substack{V \subset \tilde{U}, g_V > 0 }} t_V v(g_V) 
			- \sum_{V \subset \tilde{U}, g_V > 0} v_n(V)  \rightarrow 0, 
		\end{align}
		where (from   \eqref{sum-V-subset-tilde-U})  the second term in the left hand side is in 
		$\mathrm{supp}(Z)$.  Since $\mathrm{supp}(Z)$ is closed, this implies that 
		\begin{equation}\label{sum-stack-v}
			\sum_{\substack{V \subset \tilde{U}, g_V > 0}} \mathbb{R}_+ v(g_V) \subset \mathrm{supp}(Z).
		\end{equation}

		(3) Assume Condition~\ref{cond::conditions_on_alpha} holds with $\alpha = 1$ {   and $\mathbb{E}[|Z|] < \infty$}. We then prove that $\mathrm{supp}(Z) \subset \overline{H}$. Recall that
		\begin{equation}
			H := \{ s_1v_1 + \cdots + s_nv_n : n \geq 1, s_i \geq 0, v_i \in \Lambda\}.
		\end{equation}	
		From~\cite[Proposition 3.1]{BurDaGuiMent14} and the condition  that $\alpha = 1$, there exists a probability measure $\nu$ on $\mathbb{S}_+^{d-1}$ with support in $\Lambda$, such that for any bounded measurable  function $f: \mathbb{S}_+^{d-1} \rightarrow \mathbb R_+$,
		\begin{equation}\label{int-E-A1-x-f-A1-x}
			\int \mathbb{E}\Big[ \sum_{i = 1}^N|A_i x| f(A_i \cdot x)\Big] \nu(dx) = \int f(x) \nu(dx), 
		\end{equation}
		where $ g \cdot x = \frac{ gx }{  | gx | }$ denotes the direction of $gx$, for $g \in G$ and $x \in  \mathbb{S}_+^{d-1}$. 
		Substituting $f$ in~\eqref{int-E-A1-x-f-A1-x} with the identity function on $\mathbb{S}_+^{d-1}$, we have
		\begin{align}
			\mathbb{E}[\sum_i A_i] \big(\int x\nu(dx)\Big) = \int x\nu(dx).
		\end{align}
		This implies that
		\begin{equation}
			v := \int x\nu(dx)
		\end{equation} is a Perron-Frobenius eigenvector of $\mathbb{E}[\sum_i A_i]$. Since $\nu$ is supported on $\Lambda$, there exists a sequence $\{ \nu_n \}_{n \geq 1}$ of measures on $\Lambda$, such that $\mathrm{supp}(\nu_n)$ is a finite set for each $n \geq 1$, and $\nu_n$ converges weakly to $\nu$. So $\int x \nu_n(dx) \in H$ by the definition of $H$, so that 
		$$v = \lim_{n \to \infty} \int x \nu_n(dx) \in \overline{H}. $$
		
		Let $(A_{ui})_{i \geq 1}, u \in \mathbb{U},$ be  i.i.d. copies of  $(A_i)_{i \geq 1}$, and let 
		\begin{equation}
			G_u = A_{u|1}\cdots A_{u| |u|}.
		\end{equation}
		Since $v$ is a Perron-Frobenius right eigenvector of $\mathbb{E}[\sum_i A_i]$, from ~\cite[Theorem 2.3]{BurDaGuiMent14} and its proof, we know that there exists $c > 0$ such that 
		\begin{equation}
			Z \overset{\mathcal{L}}{=} c\lim_{n \to \infty} \sum_{ |u| = n} G_u v.
		\end{equation}
		From~\cite[Lemma 4.2]{BurDaGuiMent14}, for any $a \in \Gamma$, $w \in \Lambda$, we have $a\cdot w \in \Lambda$. Thus, we deduce from $v \in \overline{H}$ that a.s.  $G_u v \in \overline{H}$ for all $u \in \mathbb{U}$, hence $\mathrm{supp}(Z) \subset \overline{H}$.
		
		{ (4) In order to show $\mathrm{supp}(Z) \subset H$, it suffices to prove that $H \subset \mathbb{R}^d$ is closed. We will prove that $H = H_d$ and $H_d$ is closed. Recall that $$H_d = \{ s_1v_1 + \cdots + s_dv_d :  s_i \geq 0, v_i \in \Lambda, \forall i = 1, \cdots, d \}.$$
			
			To prove $H = H_d$, we will use Carathéodory's theorem in convex analysis (see~\cite{rockafellar2015convex}): if a point $x$ lies in the convex hull of a set $P \subset \mathbb{R}^{d - 1}$, then $x$ can be written as the convex combination of at most $d$ points in $P$. 
			By definitions of $H$ and $H_d$, we only need to  prove that $H \subset H_d$. 
			Let $v = s_1 v_1 + \cdots +s_n v_n $ be an arbitrary element in $H$, where $s_i \geq 0$, $v_i \in \Lambda$.  If $n \leq d$, then 
			$v \in H_d$ by the definition of $H_d$  (since the coefficient may be $0$).  
			Suppose that $n>d$. Then $v/ | v | $  is a convex combination of  the $n$ elements $v_1, \cdots, v_n \in \Lambda$. 
			Since $\Lambda \subset \mathbb{S}_+^{d-1}$ lies in the hyperplane with dimension $d - 1$ defined by the equation $x_1 + \cdots + x_d =1$, by Carathéodory's theorem it follows  that 
			$v/ |v|$ is a combination of at most $d$ elements   $v_1', \cdots, v_d' \in \Lambda$.   		 
			Therefore $v/ |v| \in H_d$, and hence $v \in H_d$.  This shows that    $H \subset H_d$, hence $H = H_d$.
			
			We now show that $H_d$ is closed. Let $(v^{(i)})_{i \ge 1}$ be a convergent sequence in $H_d$, and by the definition of $H_d$, write $v^{(i)} = \sum_{j=1}^ds^{(i)}_jv^{(i)}_j$ with $s^{(i)}_j \geq 0$, $v^{(i)}_j \in \Lambda$ for $i \geq 1$, $j = 1, \cdots, d$. Since $\Lambda \subset \mathbb{S}^{d-1}$ is closed and $\mathbb{S}^{d-1}$ is compact, we know $\Lambda$ is also compact. Note that $\max_{i\ge 1, j=1, \cdots d}({s^{(i)}_j}) \leq \max_{i \geq 1} |v^{(i)}| < +\infty$. By compactness, there exists a subsequence $(i_k)_{k\geq 1}$ of $\mathbb N^*$, such that for all $j = 1, \cdots, d$, the limits $s_j := \lim_{k\to \infty}s^{(i_k)}_j$ and $v_j := \lim_{k \to \infty}v^{(i_k)}_j$ exist, with   $s_j \in \mathbb{R}_+$ and $v_j \in \Lambda$. This implies that $\lim_{i \to \infty}v^{(i)} = \sum_{j = 1}^d s_jv_j \in H_d$, so $H_d$ is closed. 
		}
	\end{proof}

	\begin{proof}[Proof of Theorem~\ref{thm::second_theorem}]
		(1) { We first prove~\eqref{overline-D-supp-Z-overline-H}. Part 2 of Theorem~\ref{thm::main_theorem} already shows that $\mathrm{supp}(Z) \subset H$, so it suffices to show $D \subset \mathrm{supp}(Z)$.   Let $n \in \mathrm{supp}(N)$. Let $g_1, \cdots, g_n \in \Gamma$ with $g_i > 0$ for $i = 1, \cdots n$. By the definition of $\Gamma$, for $i = 1, \cdots, n$, there exist $n_i \geq 1$ and $a_1^{(i)}, \cdots, a_{n_i}^{(i)} \in \mathrm{supp}(\mu)$, such that $g_i = a_1^{(i)}\cdots a_{n_i}^{(i)}$. Let $k(i) = n_i$ for $i = 1, \cdots, n$, and $k(i) = 1$ for $i > n$. Then $U := \{ u \in \mathbb{U} : |u| = k(u | 1) \}$ is a cover set with finite level. 
			
			Choose $V := \{ u \in U : u = (i, 1, \cdots, 1) \in (\mathbb{N}^*)^{n_i}, 1 \leq i \leq n \}$. For $i = 1, \cdots, n$, $j = 1, \cdots, n_i$, set $a_{u|j} = a^{(i)}_j$ if $u \in V$ satisfies $u | 1 = i$. By Condition~\ref{cond::conditions_on_independence}, conditioned on $\{N = n \}$, the random matrices $A_1, \cdots, A_n$ are i.i.d. with law $\mu$, which implies that $(n, a_1^{(1)}, \cdots, a_1^{(n)}, 0, \cdots) \in \mathrm{supp}(N, A_1, A_2, \cdots)$.  Then, by Part 1 of Theorem~\ref{thm::main_theorem}, with $g_u := a_{u|1} \cdots a_{u | |u|}$, 
			\begin{equation*}
				\sum_{i = 1}^n \mathbb{R}_+ v(g_i) = \sum_{u \in V} \mathbb{R}_+ v(g_u) \subset \mathrm{supp}(Z).
			\end{equation*}
			Since the choice of $(g_i)_{1 \le i \le n}$ is arbitrary, we derive that $D \subset \mathrm{supp}(Z)$.

			%
			
			(2) 
			Suppose that $d \leq \mathrm{esssup}(N) \leq \infty$. Since $H_d \subset D \subset H = H_d$ (the last equality is from the proof of Theorem~\ref{thm::main_theorem}), we derive from $D \subset \mathrm{supp}(Z) \subset H$ that $\mathrm{supp}(Z) = H = H_d = D$. 
			
		}
		
		
	\end{proof}
	
	\subsection{Proof of Theorem~\ref{thm::third_theorem}  }
	
	To prove Theorem~\ref{thm::third_theorem}, we will use the following proposition proved in~\cite{Ment16}. 
	\begin{proposition}[{{\cite[Theorem 6.1]{Ment16}}}]\label{prop::from_Ment16}
		Under the hypothesis of Theorem~\ref{thm::third_theorem} with $\alpha \in (0,1)$, there exists a probability measure $\nu^\alpha$ supported on
		$\Lambda$
		such that 
		\begin{equation}
			\lim_{r \to \infty} \frac{\mathbb{P}\big[\frac{Z}{|Z|} \in \cdot, |Z| > sr\big]}{\mathbb{P}[|Z| > r]} = s^{-\alpha}\nu^{\alpha}, \quad \forall s > 0.
		\end{equation}
	\end{proposition}
	
	
	\begin{proof}[Proof of Theorem~\ref{thm::third_theorem}]
		
		(1) We first prove that there exists a matrix $g$ such that $g > 0$, $r(g) < 1$ and $\{ gz : z \in \mathrm{supp}(Z) \} \subset \mathrm{supp}(Z)$. Since $m'(\alpha) < 0$ and 
		$m(\alpha) = 1$ , we know there exists $s \in (\alpha, 1)$ such that $m(s) < 1$. Let $\delta = (1 - m(s))/2$. Then, by the definitions of $\kappa$ and $m$, there exists $n' \geq 0$ such that for all integer $n \geq n'$, we have
		\begin{equation}
			(\mathbb{E}[N])^n \int \|a\|^{s} d\mu^n (da) < (1-\delta)^n,
		\end{equation}
		where $\mu^n$ is the law of the product of $n$ i.i.d. random matrices with law $\mu$. 
		Let $(A_{u1},A_{u2},\cdots), u \in \mathbb{U}$ be i.i.d. copies of $(A_1,A_2,\cdots)$. Using the subadditivity of $x \mapsto x^s$ and the definition of $\mu$, we know that for $n \geq n'$,
		\begin{equation}
			\mathbb{E}\big[ \big( \sum_{|u| = n} \|G_u\| \big)^s \big] \leq \mathbb{E}\big[ \sum_{|u| = n} \|G_u\|^{s} \big] = (\mathbb{E}[N])^n \int \|a\|^{s} d\mu^n (da) < (1-\delta)^n,
		\end{equation}
		where $G_u := A_{u|1} \cdots A_{ u | |u|}$. This implies that for each $\varepsilon > 0$, there exist $n_1 \geq 0$ and $(a_{u1}, a_{u2}, \cdots)\in \mathrm{supp}(A_1, A_2, \cdots)$ for $u \in \mathbb{U}$,  such that $b := \sum_{|u| = n_1} a_{u|1} \cdots a_{u |n_1}$ satisfies $\|b\| < \varepsilon$. 
		Using~\eqref{equ::support_breaking_easy_form}, we know that $\{ bz : z\in \mathrm{supp}(Z) \} \subset \mathrm{supp}(Z)$. By Condition~\ref{cond::condition_on_mu} and \eqref{equ::support_breaking_easy_form}, there exists a matrix $h > 0$ such that $\{ hz : z\in \mathrm{supp}(Z) \} \subset \mathrm{supp}(Z)$. Choose $\varepsilon < 1/\|h\|$ and let $g = hb$. If follows that $g > 0$, $r(g) \leq \|g\| < 1$ and $\{ gz : z\in \mathrm{supp}(Z) \} \subset \mathrm{supp}(Z)$. This proves the claim. 
		
		
		(2) We next prove that $\mathbb{R}_+ v(g) \subset \mathrm{supp}(Z)$. For any $\varepsilon > 0$, we know from Proposition~\ref{prop::from_Ment16} that
		\begin{equation}\label{lim-r-infty-frac-1-1+epsilon}
			\lim_{r \to \infty} \frac{\mathbb{P}\big[\frac{Z}{|Z|} \in \cdot, |Z| \in [r, (1+\varepsilon)r] \big]}{\mathbb{P}[|Z| > r]} = (1 - (1+\varepsilon)^{-\alpha})\nu^\alpha, \quad \forall s > 0.
		\end{equation}
		In particular, this implies that for any {\color{black}$v \in \mathrm{supp}(\nu^\alpha)$}, 
		$\varepsilon, \varepsilon_1 > 0$ and any large enough $r > 0$, we have
		\begin{equation}\label{p-fracZZ-v}
			\mathbb{P}\Big[ \big| \frac{Z}{|Z|} - v \big| \leq \varepsilon_1, |Z| \in [r, (1+\varepsilon)r] \Big] > 0.
		\end{equation}
		Let $v_0 \in \mathrm{supp}(\nu^\alpha)$ with $v_0 > 0$,
			and $s > 0$ be a positive real number. Using~\eqref{p-fracZZ-v} with 
		{ 
			$\varepsilon_1 =  (\min_{i = 1, \cdots, d}  v_0(i)) \varepsilon / (1 + \varepsilon)$, 
			$r=(1+\varepsilon)^n$, we know that there exist $n_0 \geq 1$, sequences $\{ w_n \}_{n \geq n_0}  \subset \mathbb{R}^d$ and $\{c_n\}_{n \geq n_0} \subset \mathbb{R}$,
			such that
			\begin{align}
				& z_n :=  c_n(v_0 + w_n) \in \mathrm{supp}(Z),
			\end{align}
			and $|w_n| \leq \varepsilon_1$, 
			$(1+\varepsilon)^n \leq c_n \leq (1+\varepsilon)^{n+1}$ for $n \geq n_0$. 
			By the choice of $\varepsilon_1$, we derive that $-\varepsilon v_0 / (1+\varepsilon) \leq w_n \leq \varepsilon v_0$. Thus, $(1 + \varepsilon)^{-1} v_0 \leq v_0 + w_n \leq (1 + \varepsilon) v_0$, hence
			\begin{equation}\label{1+epsilon-n-1-v-0}
				(1+\varepsilon)^{n-1} v_0 \leq z_n \leq (1+\varepsilon)^{n+2} v_0,\quad \forall n \geq n_0.
			\end{equation}
		}
		
		Recall that $v(g)$ is the positive Perron-Frobenius right eigenvector of $g$ with unit norm, and let $u(g)$ be the Perron-Frobenius left eigenvector of $g$ such that $\langle v(g), u(g) \rangle = 1$. We can argue as in~\eqref{C-1rlinvi} to derive that there exists an integer $k_0 > 0$ that depends on $\varepsilon$, such that
		\begin{equation}\label{1+epsilon-1-r-g-m-v-g}
			(1 + \varepsilon)^{-1} r(g)^k v(g) \otimes u(g) \leq g^k \leq (1 + \varepsilon) r(g)^k v(g) \otimes u(g), \quad \forall k \geq k_0.
		\end{equation}
		Combining~\eqref{1+epsilon-n-1-v-0} and~\eqref{1+epsilon-1-r-g-m-v-g}, we get
		\begin{equation}\label{1+epsilon-rg-k-vg}
			(1+\varepsilon)^{n-2} r(g)^k \langle u(g), v_0 \rangle v(g) \leq g^k z_n \leq (1 + \varepsilon)^{n+3} r(g)^k\langle u(g), v_0 \rangle v(g), \quad \forall n \geq n_0, \quad \forall k \geq k_0. 
		\end{equation}
		
		Since $r(g) < 1 < 1+\varepsilon$, we notice that for any large enough integer $k$, there exists an integer $n(s, k) \geq 1$ such that 
		{ 
			\begin{equation}\label{1+epsilon-n-1-rg-k}
				s(1+\varepsilon)^{-n(s,k)-1} \leq r(g)^k \langle u(g), v_0\rangle \leq s(1+\varepsilon)^{-n(s,k)}.
			\end{equation}
		}
		It follows from~\eqref{1+epsilon-rg-k-vg} and~\eqref{1+epsilon-n-1-rg-k} that for any $\varepsilon > 0$, there exists $k_1 \geq k_0$ such that 
		\begin{equation}
			(1+\varepsilon)^{-3}sv(g) \leq g^kz_{n(s,k)} \leq (1+\varepsilon)^{3}sv(g),\quad  
			\quad \forall k \geq k_1. 
		\end{equation}
		So 
		\begin{equation} 
			\big| g^k z_{n(s,k)} - sv(g) \big| \leq c(\varepsilon) s|v(g)|, \quad \forall k \geq k_1,
		\end{equation}
		where $c(\varepsilon) = ((1+\varepsilon)^3  - (1-\varepsilon)^{-3})$.
		Since $\lim_{\varepsilon \rightarrow 0}  c(\varepsilon) = 0$, it follows that as $\varepsilon \to \infty$,
		\begin{equation} \label{eq-svg-in-supp}
			g^{k_1} z_{n(s,k_1)} - sv(g)\rightarrow 0. 
		\end{equation}
		Since $\{ gz : z\in \mathrm{supp}(Z) \} \subset \mathrm{supp}(Z)$,
		we know that $g^{k_1} z_{n(s,k_1)} \in \mathrm{supp}(Z)$ for any $\varepsilon > 0$. So \eqref{eq-svg-in-supp}  implies that
		\begin{equation}
			sv(g)  \in \mathrm{supp}(Z). 
		\end{equation}
		Since $s > 0$ are arbitrary, we get $\mathbb{R}_+ v(g) \subset \mathrm{supp}(Z)$. 
		
		(3) To finish the proof of \eqref{equ::support-containus-many-rays}, we can argue as in steps (1) and (2) of  the proof of Theorem~\ref{thm::main_theorem} to derive that $\mathrm{supp}(Z)$ contains many rays. 
		
	\end{proof}
	
	
	\section{Absolute continuity}\label{section::absolute_continuity}
	In this section, we will prove Theorem~\ref{thm::absolute-continuity}. Assume Condition~\ref{cond::conditions_on_N}.
	For all $t \in \mathbb{R}^d \setminus \{ 0 \}$ and $\delta \geq 0$, we recall the notation 
	\begin{equation}
		N(t) := \# \{ i : A_i^T t \ne 0, 1 \leq i \leq N \} = \sum_{i = 1}^N \mathbbm{1}_{\{ A_i^T t \ne 0\}},
	\end{equation}
	and define
	\begin{equation}
		N_\delta(t) = \# \{ i : |A_i^T t| > \delta |t|, 1 \leq i \leq N \} = \sum_{i = 1}^N \mathbbm{1}_{\{ |A_i^T t| > \delta |t| \}}.
	\end{equation}
	Note that for each $t \in \mathbb{R}^d \setminus \{ 0 \}$, we have $N_0(t) = N(t)$ and $N_\delta(t)$ is decreasing in $\delta$.
	
	\begin{lemma}\label{lem::N-delta-t-has-larger-than-one-expectation}
		The random variable  $N_\delta$ has the homogeneous  property   that 
		\begin{equation}\label{N-delta-t-N-delta-at}
			N_\delta(t) = N_\delta(at), \quad \forall t \in \mathbb{R}^d\setminus\{ 0 \}, \quad \forall \delta \geq 0, \quad \forall a > 0.
		\end{equation}
		Assume Condition~\ref{cond::conditions_on_N}, \ref{cond::absolute-continuity}. Then, there exist $\delta_0 > 0$ and $\eta > 0$ such that
		\begin{equation}\label{E-N-delta-t-ge-1}
			\mathbb{E}[N_\delta(t)] > 1 + \eta,\quad \forall t \in \mathbb{R}^d \setminus \{ 0 \}, \quad \forall \delta \in [0, \delta_0].
		\end{equation}
	\end{lemma}
	\begin{proof} 
		The homogeneous property of $N_\delta$  is obvious  by the definition of $N_\delta$. 
		Thus, it suffies to prove that  \eqref{E-N-delta-t-ge-1} holds for any $t \in \mathbb{S}^{d-1}$.
		
		For each $t \in \mathbb{S}^{d-1}$, we know that $N_\delta(t)$ increases to $N(t)$ as $\delta \to 0+$. Since
		$\mathbb{E}[N] < +\infty$, by the dominated convergence theorem, we get
		\begin{equation}\label{lim-del-ta-0+-E-N-delta-t}
			\lim_{\delta \to 0+} \mathbb{E}[N_\delta(t)] = \mathbb{E}[N(t)], \quad \forall t \in \mathbb{S}^{d-1}.
		\end{equation}
		By the continuity of $s \mapsto A_i^T s$ for $s \in \mathbb{S}^{d-1}_+$, 
		\begin{equation}
			\liminf_{s \to t} \mathbbm{1}_{\{ |A_i^T s| > \delta \}} \geq \mathbbm{1}_{\{ |A_i^T t| > \delta \}}
		\end{equation}
		Using Fatou's lemma and the definition of $N_\delta(\cdot)$, we have
		\begin{equation}\label{lim-s-t-s-S-d-1}
			\liminf_{s \to t} \mathbb{E}[N_\delta(s)] \geq \mathbb{E}[N_\delta(t)], \quad \forall t \in \mathbb{S}^{d-1}, \quad \forall \delta > 0.
		\end{equation}
		We obtain from Condition~\ref{cond::absolute-continuity} that $\mathbb{E}[N(t)] > 1$ for all $t \in \mathbb{S}^{d-1}$.
		By~\eqref{lim-del-ta-0+-E-N-delta-t}, we know that for each $t \in \mathbb{S}^{d-1}$, there exist $\delta_t > 0$ and $\eta_t > 0$ such that
		\begin{equation}
			\mathbb{E}\big[ N_{\delta_t}(t) \big] \geq 1 + \eta_t. 
		\end{equation}
		This together with~\eqref{lim-s-t-s-S-d-1} implies that for each $t \in \mathbb{S}^{d-1}$, there exists $\varepsilon_t > 0$ such that for $s \in B(t, \varepsilon_t) := \{ u \in \mathbb{R}^d : |u - t| < \varepsilon_t \}$,
		\begin{equation}\label{E-N-delta-t-s-ge-1}
			\mathbb{E}[N_{\delta_t}(s)] \geq 1 + \frac{\eta_t}{2}.
		\end{equation}
		Since $\mathbb{S}^{d-1}$ is compact and $\{ B(t,\varepsilon_t) \}_{t \in \mathbb{S}^{d-1}}$ is an open cover of $\mathbb{S}^{d-1}$, there exists a finite subcover $\{ B(t_i, \varepsilon_i) \}_{1 \leq i \leq K}$ of $\mathbb{S}^{d-1}$. Let $\delta_0 = \min_{1 \leq i \leq K} \delta_{t_i} > 0$ and $\eta = \min_{1 \leq i \leq K} {\eta_{t_i}} / {2} > 0$.
		Then, using the fact that $N_\delta(t)$ is decreasing in $\delta$, we obtain
		\begin{equation}\label{E-N-delta-t-ge-1+eta}
			\mathbb{E}[N_{\delta}(t)] \geq 1 + \eta, \quad \forall t \in \mathbb{S}^{d-1}, \quad \forall \delta \in [0,\delta_0].
		\end{equation}
		It follows from~\eqref{N-delta-t-N-delta-at} that the inequality in~\eqref{E-N-delta-t-ge-1+eta} holds for $t \in \mathbb{R}^d \setminus \{ 0 \}$ and $\delta \in [0,\delta_0]$, hence~\eqref{E-N-delta-t-ge-1} holds. 
	\end{proof}

	Let $Z$ be the solution to~\eqref{equ::smoothing_transform} such that $\mathbb{P}[Z = 0] = 0$. Denote the characteristic function of $Z$ by 
	\begin{equation}
		\phi(t) = \mathbb{E}[e^{i \langle t, Z \rangle}], \quad t \in \mathbb{R}^d.
	\end{equation}
	From~\eqref{equ::smoothing_transform}, $\phi$ satisfies the following functional equation:
	\begin{equation}\label{equ::functional-equation}
		\phi(t) = \mathbb{E}\big[ \prod_{i=1}^N \phi(A_i^T t) \big], \quad \forall t \in \mathbb{R}^d.
	\end{equation}
	
	\begin{lemma}\label{lem::Fourier-transform-decays-to-0}
		Under the hypothesis of Theorem~\ref{thm::absolute-continuity}, we have
		$| \phi(t)  |< 1$ for all $t \neq 0$, and 
		\begin{equation}
			\lim_{|t| \to \infty} \phi(t) = 0.
		\end{equation}
	\end{lemma}
	We need the following three technical lemmas.
	
	\begin{lemma}\label{lem::N-delta-equals-zero-probability}
		Under Condition~\ref{cond::conditions_on_N}, \ref{cond::absolute-continuity}, 
		we have, 
		\begin{equation}\label{P-N-delta-t-0-C-delta-epsilon-0}
			\mathbb{P}[N_\delta(t) = 0] \leq \mathbb{P}[|A_{i(t)}^T t| \leq \delta ]  \leq C \delta^{\varepsilon_0}, \quad \forall t \in \mathbb{R}^d \setminus \{ 0 \}, \quad \forall \delta \geq 0,
		\end{equation}
		with  $\varepsilon_0 > 0$ and $C = \sup_{t \in \mathbb{S}^{d-1}}\mathbb{E}[|A_{i(t)}^T t|^{-\varepsilon_0}] < \infty $  introduced   in~\eqref{sup-t-S-d-1-E-A-it-T-t-epsilon-0}. 
	\end{lemma}
	\begin{proof} From~\eqref{N-delta-t-N-delta-at}, we only need to consider  $t \in \mathbb{S}^{d-1}$. 
		Note that for any $t \in \mathbb{S}^{d-1}$ and $\delta \geq 0$, 
		\begin{align} \label{eq-PNdeltat=0} 
			\mathbb{P}[N_\delta(t) = 0] & = \mathbb{P}[|A_i^T t| \leq \delta, 1 \leq i \leq N]  \leq \mathbb{P}[|A_{i(t)}^T t| \leq \delta ] = \mathbb{P}[|A_{i(t)}^T t|^{-1} \geq \delta^{-1}].
		\end{align}
		By Markov's inequality, we get
		\begin{equation} \label{PAitTt} 
			\mathbb{P}[|A_{i(t)}^T t|^{-1} \geq \delta^{-1}] \leq \delta^{\varepsilon_0} \mathbb{E}[|A_{i(t)}^T t|^{-\varepsilon_0}] \leq C\delta^{\varepsilon_0}, \quad \forall t \in \mathbb{S}^{d-1}, \quad \forall \delta \geq 0.
		\end{equation}
		Thus ~\eqref{P-N-delta-t-0-C-delta-epsilon-0}  follows from \eqref{eq-PNdeltat=0}  and \eqref{PAitTt}. 
	\end{proof}
	
	\begin{lemma}\label{lem::N-equals-one-with-less-than-one-probability}
		Assume Condition~\ref{cond::conditions_on_N}. For any $\eta > 0$,  there exists $\delta > 0$ such that, 
		for any $\mathbb{N}$-valued random variable $\tilde{N}$ satisfying
		\begin{equation}
			\tilde{N} \leq N  \mbox{ a.s. } \quad \mbox{ and }  \quad \mathbb{E}[\tilde{N}] \geq 1 + \eta,
		\end{equation}
		we have $\mathbb{P}[\tilde{N} \leq 1] \leq 1 - \delta$.
	\end{lemma}
	{ 
		\begin{proof}
			Let $\delta > 0$ be such that for any event $B$ with $\mathbb{P}[B] < \delta$, we have $\mathbb{E}[N \mathbbm{1}_{B}] < \frac{\eta}{2}$. We claim that $\mathbb{P}[\tilde{N} \leq 1] < 1-\delta$. Otherwise, we have $\mathbb{P}[\tilde{N} \geq 2] < \delta$, hence
			\begin{equation}
				\mathbb{E}[\tilde{N}] = \mathbb{E}[\tilde{N} \mathbbm{1}_{\{ \tilde{N} \leq 1 \}}] + \mathbb{E}[\tilde{N} \mathbbm{1}_{\{ \tilde{N} \geq 2 \}}] \leq 1 + \mathbb{E}[N \mathbbm{1}_{\{ \tilde{N} \geq 2 \}}] < 1 + \frac{\eta}{2}.
			\end{equation}
			This contradicts the condition that $\mathbb{E}[\tilde{N}] \geq 1 + \eta$.
		\end{proof}
	}
	
	Recall that for two real random variables $X_1, X_2$, we say  that $X_1$ is stochastically dominated by $X_2$ and denote $X_1 \overset{st.}{\leq} X_2$, if  $\mathbb{P}[X_1 \leq x] \geq \mathbb{P}[X_2 \leq x]$ for any $x \in \mathbb{R}$.
	This is equivalent to 	the existence of  a coupling $(X_1, X_2)$ such that $X_1 \leq X_2$ a.s. 
	
	\begin{lemma}\label{lem::stochastic-domination-and-weak-convergence}
		Assume Condition~\ref{cond::conditions_on_N}. Let $c > 0$ and $\{ N_i \}_{i \geq 1}$ be a sequence of  $\mathbb{N}$-valued  random variables  such that for all $i \geq 1$, 
		\begin{equation}
			N_i \overset{st.}{\leq}  N , \quad \mathbb{E}[N_i] \geq c.
		\end{equation}
		Then, the sequence $\{ N_i \}$ is tight and any subsequential weak limit $\tilde{N}$ satisfies
		\begin{equation}\label{tilde-N-ge-1}
			\tilde{N} \overset{st.}{\leq} N, \quad \mathbb{E}[\tilde{N}] \geq c.
		\end{equation}
	\end{lemma}
	\begin{proof}
		{ The tightness follows from the condition that  $N_i   \overset{st.}{\leq}   N$. 
			To prove that $\tilde{N} \overset{st.}{\leq} N$, by the right continuity of the distribution function,  it suffices to show that 
			\begin{equation}\label{P-tilde-N-le-x-le-P-N-le-x}
				\mathbb{P}[\tilde{N} \leq x] \geq \mathbb{P}[N \leq x], \quad \forall x \in \mathbb R \setminus \mathbb{N}. 
			\end{equation}
			The latter can be implied by $N_i \overset{st.}{\leq} N$ and the definition of weak convergence, using the fact that  $x \in \mathbb R \setminus \mathbb{N}$ is a continuous point of the distribution function of $N$.  
			
			We next show that $\mathbb{E}[\tilde{N}] \geq c$. Without loss of generality, suppose that $(N_i)_{i \geq 1}$ converges weakly to $\tilde{N}$. Since  $N_i \overset{st.}{\leq}  N$ and $\mathbb E N < \infty$, the family $\{N_i\}$ is uniformly integrable, so that the weak convergence implies 
			$\mathbb E \tilde N =  \lim_{n \to \infty}  \mathbb E N_i  \geq c$. } 
		%
		%
	\end{proof}
	
	\begin{proof}[Proof of Lemma~\ref{lem::Fourier-transform-decays-to-0}]
		(1) We first prove that $|\phi(t)| \ne 1$ for  any $t \ne 0$. Otherwise, there is $t \in \mathbb{R}^d \setminus \{ 0 \}$ such that $|\phi(t)| = |\mathbb{E}[e^{i\langle t, Z \rangle}]| = 1$, hence for some $\theta \in \mathbb{R}$, a.s. 
		\begin{equation}
			\langle t, Z \rangle \in 2\pi \mathbb{Z} + \theta.
		\end{equation}
		From~\eqref{equ::support-contains-many-vectors}, we know that there exist $d$ linearly independent vectors $v_i \in \mathbb{R}^d_+$ such that $\mathbb{R}v_i \subset \mathrm{supp}(Z)$, hence
		\begin{equation}
			a \langle t,  v_i \rangle \in 2\pi \mathbb{Z} + \theta, \quad \forall i = 1, \cdots, d, \quad \forall a \in \mathbb{R}_+.
		\end{equation}
		This implies that $\langle t, v_i\rangle = 0$ for all $i = 1, \cdots, d$. It follows from the linear independence of $\{ v_i \}_{i=1}^d$ that $t = 0$. This contradiction shows that $|\phi(t)| < 1$ for all $t \in \mathbb{R}^d \setminus \{ 0 \}$.
		
		(2) We next come to prove that 
		\begin{equation}
			\limsup_{|t| \to \infty} |\phi(t)| = 0 \mbox{ or } 1.
		\end{equation}
		Define
		\begin{equation}
			h(r) := \sup_{|t| \geq r}|\phi(t)|, \quad  \forall r \geq 0.
		\end{equation}
		We need to prove that 
		\begin{equation}
			l := \lim_{r \to \infty} h(r) = 0 \mbox{ or } 1.
		\end{equation}
		For any $r > 0$ and $\delta > 0$, using the functional equation~\eqref{equ::functional-equation}, we get
		\begin{align}
			h(r) & = \sup_{|t| \geq r} \mathbb{E}\big[ \prod_{i = 1}^N  |\phi(A_i^T t)|\big] \notag \\
			& \leq \sup_{|t| \geq r} \mathbb{E}\big[ \prod_{1 \leq i \leq N, |A_i^Tt| > \delta|t|}  |\phi(A_i^T t)|\big] \notag \\
			& \leq \sup_{|t| \geq r} \mathbb{E}\big[ \prod_{1 \leq i \leq N, |A_i^Tt| > \delta|t|}  |h(\delta|t|)|\big] \notag \\
			& \leq \sup_{|t| \geq r} \mathbb{E}\big[ h(\delta r)^{N_\delta(t)} \big].
		\end{align}
		From the homogeneous property  of $N_\delta$ (cf. \eqref{N-delta-t-N-delta-at}), this implies that
		\begin{equation}\label{h-r-le-sup-t-in-S-d-1}
			h(r) \leq \sup_{t \in \mathbb{S}^{d-1}} \mathbb{E} \big[ h(\delta r)^{N_\delta(t)} \big], \quad \forall r > 0, \quad \forall \delta > 0.
		\end{equation}

		From Lemmas ~\ref{lem::N-delta-t-has-larger-than-one-expectation} and \ref{lem::N-delta-equals-zero-probability}, there exist $\eta > 0$, $\delta_0 > 0$ and $C > 0$, such that 
		\begin{equation}
			N_\delta(t) \leq N, \quad \mathbb{E}[N_\delta(t)] \geq 1 + \eta, \quad \mathbb{P}[N_\delta(t) = 0] \leq C \delta^{\varepsilon_0}, \quad \forall t \in \mathbb{R}^d \setminus \{ 0 \}, \quad \forall \delta \in [0, \delta_0],
		\end{equation}
		where $\varepsilon_0 > 0$ is the constant in~\eqref{sup-t-S-d-1-E-A-it-T-t-epsilon-0} of Condition \ref{cond::absolute-continuity}. 
		For each $r > 0$ and each $\delta \in [0, \delta_0]$, suppose that $(t_i)_{i \geq 1} \subset \mathbb{S}^{d-1}$ is a sequence such that
		\begin{equation}\label{lim-t-infty-e-h-delta-r}
			\lim_{i \to \infty} \mathbb{E} \big[ h(\delta r)^{N_\delta(t_i)} \big] = \sup_{t \in \mathbb{S}^{d-1}} \mathbb{E} \big[ h(\delta r)^{N_\delta(t)} \big].
		\end{equation}
		Using Lemma~\ref{lem::stochastic-domination-and-weak-convergence},
		there exists a weakly converging subsequence of $(N_\delta(t_i))_{i \geq 1}$, such that its weak limit $\tilde{N}_{\delta, r}$ satisfies
		\begin{equation}
			\tilde{N}_{\delta, r} \overset{st.}{\leq} N, \quad \mathbb{E}[\tilde{N}_{\delta, r}] \geq 1 + \eta, \quad \mathbb{P}[\tilde{N}_{\delta, r} = 0] \leq C \delta^{\varepsilon_0}.
		\end{equation}
		Using Lemma~\ref{lem::stochastic-domination-and-weak-convergence} again, for each $\delta \in [0, \delta_0]$, we obtain a weakly converging sequence $(\tilde{N}_{\delta, r_i})_{i \geq 1}$ with $\lim_{i \to \infty} r_i = +\infty$, such that its weak limit $\tilde{N}_\delta$ satisfies
		\begin{equation}
			\tilde{N}_\delta \overset{st.}{\leq} N, \quad \mathbb{E}[\tilde{N}_\delta] \geq 1 + \eta, \quad \mathbb{P}[\tilde{N}_\delta = 0] \leq C \delta^{\varepsilon_0}.
		\end{equation} 
		Using Lemma~\ref{lem::stochastic-domination-and-weak-convergence} again, we obtain a weakly converging sequence $(\tilde{N}_{\delta_i})_{i \geq 1}$ with $\lim_{i \to \infty} \delta_i = 0$, such that the weak limit $\tilde{N}$ satisfies
		\begin{equation}\label{tilde-N-st-le-N}
			\tilde{N} \overset{st.}{\leq} N, \quad \mathbb{E}[\tilde{N}] \geq 1 + \eta, \quad \mathbb{P}[\tilde{N} = 0] = 0.
		\end{equation}
		
		For each $r > 0$ and $\delta \in [0 ,\delta_0]$, since $\tilde{N}_{\delta, r}$ is a subsequential weak limit of $(N_\delta(t_i))_{i \geq 1}$ and $h(\delta r )\in [0,1]$, the left hand side of~\eqref{lim-t-infty-e-h-delta-r} equals $\mathbb{E}[h(\delta r)^{\tilde{N}_{\delta, r}}]$, hence from~\eqref{h-r-le-sup-t-in-S-d-1},
		\begin{equation}\label{h-r-le-E-h-delta-r-N-delta-r}
			h(r) \leq \mathbb{E}[h(\delta r)^{\tilde{N}_{\delta, r}}], \quad \forall r > 0, \quad \forall \delta \in[0, \delta_0].
		\end{equation} 
		Recall that $l = \lim_{r \to \infty} h(r) = \inf_{r >0} h(r)$ and $\lim_{i \to \infty} r_i = +\infty$. Then, we deduce from~\eqref{h-r-le-E-h-delta-r-N-delta-r} and the weak convergence of $(\tilde{N}_{\delta,r_i})_{i \geq 1}$ that for all $\varepsilon > 0$ and $\delta \in [0, \delta_0]$,
		\begin{equation}
			l \leq \limsup_{i \to \infty} \mathbb{E}[h(\delta r_i)^{\tilde{N}_{\delta, r_i}}] \leq \lim_{i \to \infty} \mathbb{E}[\min(l+\varepsilon,1)^{\tilde{N}_{\delta, r_i}}] = \mathbb{E}[\min(l+\varepsilon,1)^{\tilde{N}_{\delta}}].
		\end{equation}
		Taking $\varepsilon \to 0+$, we get
		\begin{equation}
			l \leq \mathbb{E}[l^{\tilde{N}_\delta}], \quad \forall \delta \in [0, \delta_0].
		\end{equation}
		Since $(\tilde{N}_{\delta_i})_{i \geq 1}$ converges weakly to $\tilde{N}$, we get
		\begin{equation}
			l \leq \mathbb{E}[l^{\tilde{N}}].
		\end{equation}
		It follows from~\eqref{tilde-N-st-le-N} that $l = 0$ or $1$.
		
		(3) We then show that $l = 0$ by contradiction. Assume that $l = 1$. Let $(\varepsilon_k)_{k \geq 1}$ be a decreasing sequence such that $\varepsilon_1 < 1 - \sup_{|t| = 1}|\phi(t)|$ and $\lim_{k \to \infty} \varepsilon_k = 0$. Since $l = 1 = |\phi(0)|$, $|\phi(t)| < 1$ for all $t \in \mathbb{R}^d\setminus \{ 0 \}$ and $\phi$ is continuous, there exist two sequences $(s_k)_{k \geq 1}$, $(t_k)_{k \geq 1}$ in $\mathbb{R}^d \setminus \{ 0 \}$ satisfying the following properties:
		\begin{align}
			& \lim_{k \to \infty} |s_k| = 0, \quad \lim_{k \to \infty} |t_k| = +\infty,\\
			& 1 - \varepsilon_k = |\phi(t_k)| = |\phi(s_k)|, \quad \forall k \geq 1,  \label{1-epsilon-k-phi-t-k-s-k}
			\\
			& |\phi(t)| \leq 1 - \varepsilon_k, \quad \forall t \in \mathbb{R}^d \setminus \{ 0 \} \mbox{ with } |s_k| \leq |t| \leq |t_k|, \quad \forall k \geq 1. \label{phi-t-le-1-epsilon-k}
		\end{align}
		Denote $\tilde{t}_k := \frac{t_k}{|t_k|}$ for each $k \geq 1$. Since $(\tilde{t}_k)_{k \geq 1} \subset \mathbb{S}^{d-1}$ and $\mathbb{S}^{d-1}$ is compact, without loss of generality, we assume that there exists $\tilde{t} \in \mathbb{S}^{d-1}$ such that $\lim_{k \to \infty}\tilde{t}_k = \tilde{t}$. 
		
		Let $(A_{u1}, A_{u2},\cdots)_{u \in \mathbb{U}}$ be i.i.d. copies of $(A_1,A_2,\cdots)$, and $G_u = A_{u|1}\cdots A_{u| |u|}$ for each $u \in \mathbb{U}$. From Condition~\ref{cond::absolute-continuity}, we know that a.s. for any matrix $B$ with $B \tilde{t} \ne 0$ and any $n \geq 1$, there exists $u(B) \in \mathbb{U}$
		such that $|u(B)| = n$ and $G_{u(B)}^T B \tilde{t} \ne 0$. Define
		\begin{equation}
			\lambda_{n,k}(B) = \mathbb{P}[|s_k| \leq |G_{u(B)}^T B t_k| \leq |t_k|], \quad \forall k \geq 1, \quad \forall n \geq 1.
		\end{equation}
		Then, using the functional equation~\eqref{equ::functional-equation} together with~\eqref{phi-t-le-1-epsilon-k}, we have that for $k \geq 1$,
		\begin{equation}\label{1-epsilon-k-P-s-k}
			|\phi(Bt_k)| \leq \mathbb{E}[|\phi(G_{u(B)}^T B t_k)|] \leq (1 - \lambda_{n,k}(B)) + \lambda_{n,k}(B) (1-\varepsilon_k)= 1 - \varepsilon_k \lambda_{n,k}(B). 
		\end{equation}
		Since $|G_{u(B)}^T B \tilde{t}_k| \to |G_{u(B)}^T B \tilde{t}| > 0$ as $k \to \infty$, we have
		\begin{align}
			\lambda_{n,k}(B) & = \mathbb{P}[|G_{u(B)}^T Bt_k| \leq |t_k|] - \mathbb{P}[|G_{u(B)}^T Bt_k| < |s_k|] \notag \\
			& = \mathbb{P}[|G_{u(B)}^T B\tilde{t}_k| \leq 1] - \mathbb{P}\big[|G_{u(B)}^T B\tilde{t}_k| < \frac{|s_k|}{|t_k|} \big] \notag \\
			& \to \mathbb{P}[|G_{u(B)}^T B\tilde{t}| \leq 1], \quad k \to \infty. \label{P-G-u-T-B-t-le-1}
		\end{align}
		This implies that
		\begin{equation}\label{liminf-k-infty-lambda-k-B}
			\lim_{k \to \infty} \lambda_{n,k}(B) = \mathbb{P}[|G_{u(B)}^T B \tilde{t}| \leq 1] \geq  \mathbb{P}\big[\max_{|w| = n} \|G_w\| \leq \frac{1}{\|B\|} \big].
		\end{equation}
		From \cite[Lemma 4.1]{Ment16} and Condition~\ref{cond::conditions_on_alpha}, we know that $\lim_{n \to \infty}\max_{|w| = n} \|G_w\|  = 0$ a.s., hence the right hand side of~\eqref{liminf-k-infty-lambda-k-B} tends to $1$ as $n \to \infty$. This implies that
		\begin{equation}\label{lim-n-infty-liminf-k-lambda-n-k}
			\lim_{n \to \infty}\lim_{k \to \infty} \lambda_{n,k}(B) = 1, \quad \forall B \mbox{ s.t. } B\tilde{t} \ne 0.
		\end{equation}
		From~\eqref{1-epsilon-k-phi-t-k-s-k}, for each $k \geq 1$, we have $1 - \varepsilon_k = |\phi(t_k)|$, hence we deduce from the functional equation~\eqref{equ::functional-equation} and~\eqref{1-epsilon-k-P-s-k}  that
		\begin{align}\label{1-epsilon-k-phi-t-k}
			1 - \varepsilon_k = |\phi(t_k)| & = \mathbb{E}\big[ \prod_{i = 1}^N |\phi(A_i^T t_k)| \big] \leq \mathbb{E}\bigg[ \prod_{1\leq i\leq N, A_i^T\tilde{t} \ne 0} (1 - \varepsilon_k \lambda_{n,k}(A_i^T)) \bigg].
		\end{align}
		Since $\lambda_{n,k}(B) \in [0,1]$ for all $n \geq 1,k \geq 1$ and $B$, we know that the right hand side of~\eqref{1-epsilon-k-phi-t-k} is
		\begin{equation} \label{eq-sum-lambdank} 
			1 - \varepsilon_k \mathbb{E}\bigg [\sum_{1\leq i \leq N, A_i^T \tilde{t} \ne 0} \lambda_{n,k}(A_i^T) \bigg] + o(\varepsilon_k), \quad \mbox{ as } k \to \infty.
		\end{equation}
		This can be easily checked using the inequality $1-x \leq e^{-x} $ for $x \geq 0$ and the Taylor   development of  order 1 at 0 of the Laplace transform of
		$ \sum_{1\leq i \leq N, A_i \tilde{t} \ne 0} \lambda_{n,k}(A_i^T)$. 
		Combining \eqref{eq-sum-lambdank} 
		with ~\eqref{1-epsilon-k-phi-t-k}, we deduce that for all $n \geq 1$, 
		\begin{equation}
			\limsup_{k \to \infty} \mathbb{E}\bigg [\sum_{1\leq i \leq N, A_i^T \tilde{t} \ne 0} \lambda_{n,k}(A_i^T) \bigg] \leq 1.
		\end{equation}
		Using  ~\eqref{lim-n-infty-liminf-k-lambda-n-k} and Fatou's lemma twice,  first as $k\to \infty$ and then as   $n \to \infty$, we know that this implies
		\begin{equation}
			\mathbb{E}[N(\tilde{t})] = \mathbb{E}\big[ \sum_{i=1}^N \mathbbm{1}_{\{ A_i^T \tilde{t} \ne 0 \}} \big] 
			= \mathbb E  \lim_{n \to \infty} \lim_{k \to \infty}  \sum_{1\leq i \leq N, A_i^T \tilde{t} \ne 0} \lambda_{n,k}(A_i^T) 
			\leq 1.
		\end{equation}
		This contradicts Condition~\ref{cond::absolute-continuity}. Thus, we have $l = 0$. So we have proved that 
		\begin{equation}
			\lim_{r \to \infty} \sup_{|t| \geq r} |\phi(t)| = 0.
		\end{equation}
	\end{proof}

	The following Gronwall-type technical  lemma is a generalization of \cite[Lemma 3.2]{liu2001asymptotic} where the one dimensional case was considered. It gives a sharp  estimation of the decay rate of 
	a bounded function $\varphi: \mathbb{R}^d \to \mathbb{R}_+$ satisfying a functional inequality. For applications in the case $d=1$, see e.g. \cite{chauvin2014limit, grama2023asymptotic, grama2017berry}, etc.

	\begin{lemma}\label{lem::poly-decay-characteristic-func}
		Let $X : \mathbb{R}^d \to \mathbb{R}_+$ and $M : \mathbb{R}^d \to \mathbb{R}^d$ be two random mappings such that  $X(ct) = X(t)$, $M(ct) = cM(t)$ for any vector $t \in \mathbb{R}^{d} \setminus \{ 0 \}$ and scalar $c > 0$, and that there exists $a > 0$ satisfying 
		\begin{equation}
			\sup_{t \in \mathbb{R}^d \setminus \{0 \}} \mathbb{E}[X(t)] < 1 \quad \mbox{ and } \quad \sup_{t \in \mathbb{S}^{d-1}} \mathbb{E}[|M(t)|^{-a}X(t)] < 1.
		\end{equation}
		Suppose that $\varphi : \mathbb{R}^d \to \mathbb{R}_+$ is bounded and there exists $c_0 > 0$ such that
		\begin{equation}  \label{phi-t-E-A-it-T-t-epsilon-N-delta-t-1a}
			\varphi(t) \leq \mathbb{E}[\varphi(M(t)) X(t)], \quad \forall t \in \mathbb{R}^d  \mbox{ with } |t| > c_0.
		\end{equation}
		Then $\varphi(t) = O(|t|^{-a})$ as $|t| \to \infty$. 
	\end{lemma}
	\begin{proof}
		Denote $\eta = \sup_{t \in \mathbb{R}^d \setminus \{0 \}} \mathbb{E}[X(t)] < 1$ and $\theta = \sup_{t \in \mathbb{S}^{d-1}} \mathbb{E}[|M(t)|^{-a}X(t)] < 1$. Let $(M_i,X_i)_{i \geq 1}$ be i.i.d. copies of $(M,X)$, and
		\begin{align}
			& \tilde{M}_0 (t) := t, \quad \tilde{M}_{i} (t) := M_i \circ \cdots \circ M_1 (t), \\
			& \tilde{X}_0(t) := 1, \quad \tilde{X}_i(t) := X_i(\tilde{M}_{i-1}t) \cdots X_1(\tilde{M}_0 t), \quad \forall i \geq 1, \quad \forall t \in \mathbb{R}^d.
		\end{align}
		Let $\mathcal{F}_i := \sigma(M_1,\cdots M_i)$ for all $i$. By~\eqref{phi-t-E-A-it-T-t-epsilon-N-delta-t-1a}, we can choose $C>0$ large enough such that
		\begin{equation}\label{phi-t-le-E-phi-Mt-X-t-all-t}
			\varphi(t) \leq \mathbb{E}[\varphi(M(t))X(t)] + C|t|^{-a}, \quad \forall t \in \mathbb{R}^d \setminus \{  0 \}.
		\end{equation}
		Then by induction, 
		\begin{equation}\label{phi-t-le-E-phi-M-n-t-X-n-t-add}
			\varphi(t) \leq \mathbb{E}[\varphi(\tilde{M}_n (t)) \tilde{X}_n(t)] + C \sum_{i=0}^{n-1} \mathbb{E}[|\tilde{M}_i (t)|^{-a} \tilde{X}_i(t)], \quad \forall t \in \mathbb{R}^d \setminus \{  0 \}, \quad \forall n \geq 1.
		\end{equation}
		Using the definition of $\eta$ and the independence between $\tilde{X}_i$ and $\mathcal{F}_{i-1}$ for $i \geq 1$,
		\begin{equation}
			\mathbb{E}[\tilde{X}_i(t)] = \mathbb{E}[\mathbb{E}[X_i(\tilde{M}_{i-1}t) | \mathcal{F}_{i-1}] \tilde{X}_{i-1}(t)] \leq \eta \mathbb{E}[\tilde{X}_{i-1}(t)], \quad \forall i \geq 1, \quad \forall t \in \mathbb{R}^d \setminus \{ 0 \},
		\end{equation}
		which yields
		\begin{equation}\label{E-tildeX-i-t-le-eta-i}
			\mathbb{E}[\tilde{X}_i(t)] \leq \eta^i, \quad \forall i \geq 0, \quad \forall t \in \mathbb{R}^d \setminus \{ 0 \}.
		\end{equation}
		Besides, using the fact that $M(at) = aM(t)$ for $a > 0$ and $t \ne 0$, the definition of $\eta$, and the independence between $\tilde{M}_{i}$ and $\mathcal{F}_{i-1}$, we have
		\begin{align}
			\mathbb{E}[|\tilde{M}_i (t)|^{-a}\tilde{X}_i(t)] & = \mathbb{E}\big[ |\tilde{M}_{i-1}(t)|^{-a} \tilde{X}_{i-1}(t) \cdot  \mathbb{E}\big[  \big| M_i \big( \frac{\tilde{M}_{i-1}(t)}{|\tilde{M}_{i-1}(t)|} \big) \big|^{-a} X_i(\tilde{M}_{i-1}(t))  \big| \mathcal{F}_{i-1} \big]    \big] \notag \\
			& \leq \theta \mathbb{E}[|\tilde{M}_{i-1}(t)|^{-a} \tilde{X}_{i-1}(t)], \quad \forall i \geq 1, \quad \forall t \in \mathbb{R}^d \setminus \{ 0 \},
		\end{align}
		which yields
		\begin{equation}\label{E-tilde-M-i-t--a-i-t}
			\mathbb{E}[|\tilde{M}_i (t)|^{-a}\tilde{X}_i(t)] \leq \theta^i |t|^{-a}, \quad \forall i \geq 1, \quad \forall t \in \mathbb{R}^d \setminus \{ 0 \}.
		\end{equation}
		Using the bounds~\eqref{E-tildeX-i-t-le-eta-i} and~\eqref{E-tilde-M-i-t--a-i-t} and the fact that $\varphi(t) \leq c_1 := \sup_{ t \in \mathbb R^d}   \varphi(t)  < \infty$, we deduce from~\eqref{phi-t-le-E-phi-M-n-t-X-n-t-add} that
		\begin{equation}
			\varphi(t) \leq c_1 \eta^n + C\sum_{i=0}^{n-1} \theta^i |t|^{-a}, \quad \forall n \geq 1, \quad \forall t \in \mathbb{R}^d \setminus \{ 0 \}.
		\end{equation}
		Since $\eta < 1$ and $\theta < 1$, by taking $n \to \infty$, we obtain that
		\begin{equation}
			\varphi(t) \leq \frac{C|t|^{-a}}{1- \theta}, \quad \forall t \in \mathbb{R}^{d} \setminus \{ 0 \}. 
		\end{equation}
		So $\varphi(t) = O(|t|^{-a})$ as $|t| \to\infty$.
	\end{proof}
	
	The following result gives the decay rate of the characteristic function $\phi (t) = \mathbb E e^{i\langle t, Z\rangle} $
	as $ | t | \to \infty$.   
	\begin{lemma}\label{lem::Fourier-decays-a-order}
		Assume Conditions~\ref{cond::conditions_on_N}, \ref{cond::conditions_on_alpha}, \ref{cond::absolute-continuity}. Suppose that there exist $\delta_0 > 0$ and $a > 0$ such that
		\begin{equation}\label{sup-t-S-d-1-E-A-it-t--a}
			\sup_{t \in \mathbb{S}^{d-1}}  \mathbb{E}[|A_{i(t)}^T t|^{-a} \mathbbm{1}_{\{ N_{\delta_0}(t) \leq 1 \}}    ] < 1.
		\end{equation}
		Then, we have $|\phi(t)| = O(|t|^{-a})$ as $|t| \to \infty$.
	\end{lemma}	
	\begin{proof}
		Let $\varepsilon \in (0, 1)$ and $\delta > 0$. By Lemma~\ref{lem::Fourier-transform-decays-to-0}, there exists $t_\varepsilon > 0$ such that for all $t \in \mathbb{R}^d$ with $|t| > t_\varepsilon$, we have $|\phi(t)| < \varepsilon$. Using~\eqref{equ::functional-equation} and the facts that $|\phi(A_i^T t)| \leq \varepsilon$ if $|A_i^T t| > t_\varepsilon$ and $|\phi(A_i^T t)| \leq 1$ for all $t$, we see that if $|t| > \frac{t_\varepsilon}{\delta}$, then
		\begin{align}\label{phi-t-E-A-it-T-t-epsilon-N-delta-t-1}
			|\phi(t)| \leq \mathbb{E}\big[|\phi(A_{i(t)}^T t)| \big(\varepsilon^{N_\delta(t) - 1} \mathbbm{1}_{\{ N_\delta(t) \geq 1 \}} + \mathbbm{1}_{\{ N_\delta(t) = 0 \}}\big)\big] =: \mathbb{E}[|\phi(M(t))|X(t)],
		\end{align}
		where $i(t) = i\big( \frac{t}{|t|} \big)$ for $t \in \mathbb{R}^d \setminus \{ 0 \}$, $M$ is the random mapping $t \mapsto A_{i(t)}^T t$ and
		\begin{equation}
			X(t) := \varepsilon^{N_\delta(t) - 1} \mathbbm{1}_{\{ N_\delta(t) \geq 1 \}} + \mathbbm{1}_{\{ N_\delta(t) = 0 \}}, \quad \forall t \in \mathbb{R}^d \setminus \{ 0 \}.
		\end{equation}
		Note that $M(at) = aM(t)$ and $X(at) = X(t)$ for each   scalar $a > 0$ and each vector $t \in \mathbb{R}^d \setminus \{ 0 \}$, and $X(\cdot)$ depends on $\varepsilon$ and $\delta$. 
		
		We claim that there exist $\varepsilon \in (0,1)$ and $\delta < \delta_0$ such that 
		\begin{equation}\label{sup-t-S-d-1-E-X-t-le-1}
			\eta_{\varepsilon, \delta} := \sup_{t \in \mathbb{R}^d \setminus \{0 \}} \mathbb{E}[X(t)] < 1, \quad \mbox{ and } \quad \theta_{\varepsilon, \delta} := \sup_{t \in \mathbb{S}^{d-1}} \mathbb{E}[|M(t)|^{-a}X(t)] < 1.
		\end{equation}
		By Lemma~\ref{lem::N-delta-t-has-larger-than-one-expectation}, there exist $\delta_1 \in (0, \delta_0]$ and $\eta > 0$ such that
		\begin{equation}
			\mathbb{E}[N_\delta(t)] > 1 + \eta, \quad \forall t \ne 0, \quad \forall \delta \in (0,\delta_1).
		\end{equation}
		From Lemma~\ref{lem::N-equals-one-with-less-than-one-probability}, this implies that there exists $\varepsilon_1 > 0$ such that 
		\begin{equation}
			\mathbb{P}[N_\delta(t) \leq 1] \leq 1 - \varepsilon_1, \quad \forall t \ne 0, \quad \forall \delta \in (0,\delta_1).
		\end{equation}
		Then, we have
		\begin{equation}\label{p-epsilon-delta-t-le-P-N-delta-t-le-1-+-epsilon}
			\mathbb{E}[X(t)] \leq  \mathbb{P}[N_\delta(t) \leq 1] + \varepsilon \leq 1 - \varepsilon_1 + \varepsilon, \quad \forall t \ne 0, \quad \forall \delta \in (0, \delta_1), \quad \forall \varepsilon \in (0,1).
		\end{equation}
		So $\eta_{\varepsilon, \delta} < 1$ for all $\varepsilon \in (0,\varepsilon_1)$ and $\delta \in (0,\delta_1)$. 
		Note that $N_\delta(t)$ is decreasing in $\delta$ for each $t \in \mathbb{R}^d \setminus \{ 0 \}$. Using this fact, we have
		\begin{align}
			\theta_{\varepsilon, \delta} & \leq \varepsilon \cdot  \sup_{t \in \mathbb{S}^{d-1}}\mathbb{E}[|M(t)|^{-a}] + \sup_{t \in \mathbb{S}^{d-1}}\mathbb{E}[|M(t)|^{-a} \mathbbm{1}_{\{ N_{\delta_0}(t) \leq 1 \}}] =: c_1 \varepsilon + c_2, \quad \forall \delta \in (0,\delta_0),
		\end{align}
		where $c_1 := \sup_{t \in \mathbb{S}^{d-1}}\mathbb{E}[|M(t)|^{-a}]$ and $c_2 := \sup_{t \in \mathbb{S}^{d-1}}\mathbb{E}[|M(t)|^{-a} \mathbbm{1}_{\{ N_{\delta_0}(t) \leq 1 \}}]$. By~\eqref{sup-t-S-d-1-E-A-it-T-t-epsilon-0} and~\eqref{sup-t-S-d-1-E-A-it-t--a}, we know that $c_1 < \infty$ and $c_2 < 1$. Then, we have $\theta_{\varepsilon,\delta} < 1$ for all $\varepsilon \in (0,\frac{1- c_2}{c_1})$ and $\delta \in (0,\delta_0)$. This proves the claim. 
		
		We fix $\varepsilon \in (0,1)$ and $\delta < \delta_0$ such that~\eqref{sup-t-S-d-1-E-X-t-le-1} holds, { so $X$ and $M$ satisfy the conditions in Lemma~\ref{lem::poly-decay-characteristic-func}.} 
		It follows from Lemma~\ref{lem::poly-decay-characteristic-func} that $|\phi(t)| = O(|t|^{-a})$ as $|t| \to \infty$. 
	\end{proof}
	
	Recall that  $\{ (N_u, A_{u1},A_{u2},\cdots) \}_{u \in \mathbb{U}}$ are i.i.d., with the same distribution  as $(N,A_1,A_2,\cdots)$. Let 
	$G_u := A_{u|1}\cdots A_{u| |u|}$ for $u \in \mathbb{U}$, 
	with the convention that $G_\emptyset$ is the identity matrix. For $n \geq 0$, let 
	$T$ be the Galton-Watson tree generated by $(N_u)$: by definition, $\emptyset \in T$; when $u \in T$, then $ui \in T$ if and only if $1\leq i \leq N_u$. Let 
	$T_n = \{ u \in T: |u| =n\} $ be the set of individuals of generation $n$.  
	For each $u \in \mathbb{U}$, denote by $T^u$ the shifted tree at $u$, whose defining elements are $\{N_{uv}: v \in \mathbb{U}\}$: we have $\emptyset \in T^u$; when $v \in T^u$, then $vi \in T^u$ if and only if $1\leq i \leq N_{uv}$.  For $n \geq 0$, let 
	$$T_n^u = \{  v \in T^u : |v| =n\}$$  be the set of individuals of generation $n$ of the tree $T^u$. 
	Notice that $T=T^\emptyset$ and $T_n = T_n^\emptyset$. 
	Set for each $u \in \mathbb{U}$, 
	\begin{align}
		N^u(t, n) & = \# \{ v \in T_n^u : G_{u, v}^T t \ne 0 \}, \notag \\
		 N(t, n) & = \# \{ v \in T_n : G_v^T t \neq 0 \}, 
		\quad  \forall n \geq 1, \;  \forall t \in \mathbb{R}^d,
	\end{align}
	where $G_{u, v} = A_{v|(|u|+1)} \cdots A_{v||v|}$ for $v \in T^u$.
	Note that $N(t, n) = N^\emptyset (t, n)$ and $N(t) = N(t,1)$ for all $t$ and $n \geq 1$. Since $\mathbb{P}[N(t) \geq 1, \forall t \ne 0] = 1$, we know that almost surely, $N(t, n)$ is non-decreasing  in $n$, for all $t \ne 0$.
	The following result shows that, under suitable conditions,  as $n \to \infty$, $N(t, n) \to +\infty$ uniformly for  $t \in \mathbb{R}^d \setminus \{ 0 \}$. 
	
	\begin{lemma}\label{lem::inf-N-t-can-be-large}
		Under Conditions~\ref{cond::conditions_on_N}, \ref{cond::absolute-continuity}, we have  
		\begin{equation}\label{P-sharp-u-in-T-n-G-u-T-t-ne-0-ge-M}
			\lim_{n \to \infty} \inf_{t \in \mathbb{R}^d \setminus \{ 0 \}} N(t, n) = +\infty, \quad \mbox{a.s.}
		\end{equation}
	\end{lemma}
	We remark  that in the lemma the condition $\mathbb{P}[N(t) \geq 1, \forall t \ne 0] = 1$ in \ref{cond::absolute-continuity}  cannot be relaxed to  $\mathbb{P}[N(t) \geq 1] = 1$ for all $t \ne 0$. 
	\begin{proof}[Proof of Lemma \ref{lem::inf-N-t-can-be-large}]
		
		(1) 
		We first prove that for any fixed $t \in \mathbb{R}^d \setminus \{ 0 \}$, 
		\begin{equation}\label{lim-n-N-t-n-ge-2-all-t-as}
			\lim_{n \to \infty} N(t, n) \geq 2, \quad \mbox{a.s.}
		\end{equation}
		By Lemmas~\ref{lem::N-delta-t-has-larger-than-one-expectation} and \ref{lem::N-equals-one-with-less-than-one-probability}, 
		we know that there exists $\delta > 0$ such that
		\begin{equation}\label{P-N-t-ge-2-ge-delta}
			\mathbb{P}[N(t) \leq 1] < 1- \delta,
			\quad \forall t \in \mathbb{R}^d \setminus \{ 0 \}.
		\end{equation}
		Since a.s. $1\leq N(t, n-1) \leq N(t, n)$  for all $n \geq 1$ and  $t  \in \mathbb{R}^d \setminus \{ 0 \}$,  we know that when  $ N(t, n) =1$, then  $ N(t, n-1)= 1$, so that there is a unique $u \in T_{n-1}$ such that $G^T_u t \ne 0$, and that $N^u( G_u^Tt, 1) =1$. 
		Therefore,  
		for each $n \geq 1$, 
		\begin{equation}
			\mathbb{P}[N(t, n) = 1] = \mathbb{E}[\mathbbm{1}_{\{ N(t, n-1) = 1\}} \mathbb{E}[ 1_{\{ N^u (G^T_u t, 1) = 1 \} }| {\mathcal F}_{n-1}]] \leq (1-\delta) \mathbb{P}[N(t, n-1) = 1],
		\end{equation}
		where
		$${\mathcal F}_{0} =  \{ \emptyset, \Omega \}, \quad {\mathcal F}_{k} =  \sigma (N_u, A_{ui}:  |u| < k, i \geq 1 \} \mbox{  for } k \geq 1. $$ 
		Here
		we have used 
		~\eqref{P-N-t-ge-2-ge-delta}  
		together with  the fact that for each fixed $u$, 
		$N^u (t, 1)$ has the same distribution as  $N (t, 1) = N(t)$. 
		It follows by induction that
		\begin{equation}\label{P-N-t-n-=-1-le-1-delta-n}
			\mathbb{P}[N(t, n) = 1] \leq (1-\delta)^n, \quad \forall n \geq 1.
		\end{equation}
		Therefore, 
		\begin{equation}
			\sum_{n \geq 1} \mathbb{P}[N(t, n) = 1] \leq \sum_{n \geq 1} (1 - \delta)^n < \infty.
		\end{equation}
		By the Borel-Cantelli lemma and the fact that $N(t, n)$ is non-decreasing in $n$, we get ~\eqref{lim-n-N-t-n-ge-2-all-t-as}.
		
		(2) We next prove that
		
		\begin{equation}\label{lim-n-inf-t-R-d-setminus-0-N-t-n-ge-2}
			\lim_{n \to \infty} \inf_{t \in \mathbb{R}^d \setminus \{ 0 \}} N(t, n) \geq 2, \quad \mbox{a.s.}
		\end{equation}
		To this end we will prove that for any linear subspace $D \subset \mathbb{R}^d$,
		\begin{equation}\label{P-N-t-n-ge-2-t-in-D-setminus-0}
			\lim_{n \to \infty} \inf_{t \in D \setminus \{ 0 \}} N(t, n) \geq 2, \quad \mbox{a.s.,}
		\end{equation}
		with the convention that  $ \inf \emptyset   =\infty$.
		Then taking $D= \mathbb R^d$ gives the desired conclusion. 
		
		We will prove \eqref{P-N-t-n-ge-2-t-in-D-setminus-0} by induction on the dimension of $D$. When  $\dim D = 0$, there is nothing to prove. 
		By~\eqref{lim-n-N-t-n-ge-2-all-t-as} and the fact that $N(t, \cdot) = N(at, \cdot)$ for any vector $t \in \mathbb{R}^d \setminus \{ 0 \}$ and scalar $a > 0$, we know that  \eqref{P-N-t-n-ge-2-t-in-D-setminus-0} holds when  $\dim D = 1$.
		
		
		Suppose that~\eqref{P-N-t-n-ge-2-t-in-D-setminus-0} holds for all linear subspaces of  dimension at most $k$ with $1 \leq k < d$. Let $D$ be a linear subspace with dimension $k+1$. Define
		\begin{equation}
			n_1 = \inf\{ n \geq 1 : \exists t \in D \setminus \{ 0\}, N(t, n) \geq 2  \}, 
		\end{equation}  
		with the convention that $\inf \emptyset = +\infty$. We know that $n_1 < +\infty$ a.s. because if we fix an $t \in D$, then we can use~\eqref{lim-n-N-t-n-ge-2-all-t-as} to deduce that a.s. $N(t, n) \geq 2$ for $n$ large enough. 
		%
		%
		%
		For each $u \in \mathbb{U}$, define
		\begin{equation}
			K_{ui} = \bigcap_{j \ne i} \ker A_{uj}^T, \quad \forall i = 1, \cdots, N_u,\quad \mbox{if $N_u \geq 2$}.
		\end{equation}
		and define $K_{u1} = \mathbb{R}^d$ if $N_{u} = 1$. 
		By Condition~\ref{cond::absolute-continuity}, for each $u \in \mathbb{U}$, $\mathbb P [ \forall t \in \mathbb R^d \setminus \{0\}, N^u (t, 1) \geq 1] =1$.    
		Using this condition,  
		we see that  a.s., 
		for each $t \in \mathbb{R}^d \setminus \{ 0 \}$ and $u \in \mathbb{U}$,
		$K_{ui} \cap K_{uj} = \{ 0 \}$ for all $1 \leq i \ne j \leq N_u$, and 
		\begin{align}
			N^u(t, 1)  
			= 1 \quad & \Leftrightarrow \quad t \in \bigcup_{i=1}^{N_u} K_{ui}. 
			\label{sharp-i-A-ui-T-t-ne-0}
		\end{align}
		
		We now show that a.s. there exists a unique $u = u(n_1) \in T_{n_1 - 1}$ such that 
		\begin{equation}\label{dim-G-T-u-D-equals-dim-D-G-T-v-D-0-other-v}
			\dim G^T_u D = \dim D, \quad \mbox{ and } \quad G^T_v D = \{ 0 \} \quad \mbox{ if } v \in T_{n_1 - 1} \setminus\{ u\}.
		\end{equation}
		Since this holds for $n_1 = 1$ (because we can take $u = \emptyset \in \mathbb{U}$),  we  only consider the case $n_1 > 1$. In this case, by the definition of $n_1$,  
		for each $t \in D \setminus \{ 0 \}$, we have $N(t) = N(t, 1) = 1$,  so that $ t \in  \bigcup_{i = 1}^N K_i$ by  \eqref{sharp-i-A-ui-T-t-ne-0}. 
		Thus, 
		\begin{equation}
			D \subset \bigcup_{i = 1}^N K_i.
		\end{equation}
		Since { each $K_i$ is a linear subspace of $\mathbb{R}^d$} and  $K_i \cap K_j = \{ 0 \}$ for $i \ne j$, we know that there exists a unique $i_1 \in \{ 1, \cdots, N \}$ such that $D \subset K_{i_1}$ and $D \cap K_i = \{ 0 \}$ for $i \ne i_1$. This implies that  $D \cap \ker A^T_{i_1} = \{0 \}$, and $D \subset \ker A^T_i$ for $i \neq i_1$.
		Therefore $\dim A_{i_1}^T D = \dim D$ and $A_i^T D = \{ 0 \}$ for $i \ne i_1$. Then, if $n_1 > 2$, then we repeat this argument for $(N_{i_1}, A_{i_11}, A_{i_12},\cdots)$ to get that there exists $i_1i_2 \in T_2$ such that $\dim G^T_{i_1i_2} D = \dim D$ and $\dim G^T_v D = \{ 0 \}$ for $v \in T_2$ with $v \ne i_1i_2$. We can repeat this argument to get an $u \in T_{n_1 - 1}$ such that  \eqref{dim-G-T-u-D-equals-dim-D-G-T-v-D-0-other-v} holds. 
		%
		
		Notice that for each $ t \in \mathbb R^d \setminus \{0\}$, and for 
		$0\leq l \leq n$,
		$ N(t, n) = \sum_{ v \in T_l} N^v ( G_v^T t,  n-l) $, so that 
		$ N(t, n) \geq  N^u ( G_u^T t,  n-n_1+1)$, { where $u = u(n_1) \in T_{n_1-1}$ satisfies~\eqref{dim-G-T-u-D-equals-dim-D-G-T-v-D-0-other-v}}; consequently,
		$$ \inf_{t \in D \setminus \{ 0 \}} N (t, n) 
		\geq   \inf_{t \in G_u^T  (D \setminus \{ 0 \})} N^u(t, n-n_1+1).$$
		Now $  G_u^T  (D) = G_u^T  (D \setminus \{ 0 \}) \cup G_u^T (\{ 0 \})$, so 
		$$  G_u^T  (D) \setminus \{0\} =   G_u^T  (D \setminus \{ 0 \}) \setminus \{0\}
		=  G_u^T  (D \setminus \{ 0 \}) ,$$  where the last equality holds 
		since $D \cap \ker  G_u^T = \{0\}$  by \eqref{dim-G-T-u-D-equals-dim-D-G-T-v-D-0-other-v}.  
		Therefore, 
		to prove~\eqref{P-N-t-n-ge-2-t-in-D-setminus-0} for $ D$ (with dimension $k+1$),
		it suffices to prove that, with $n_1 \geq 1$ and $u  = u(n_1) \in T_{n_1 - 1}$ defined above, 
		\begin{equation}\label{lim-n-inf-t-in-G-u-T-D-setminus-0-N-t-n-ge-2}
			\lim_{n \to \infty} \inf_{t \in G_u^T D \setminus \{ 0 \}} N^u(t, n) \geq 2, \quad \mbox{a.s.}.
		\end{equation}
		By the definition of $n_1$ and $u = u(n_1)$, we know that there exists some $s \in G_u^T D$ such that $N^u(s, 1) \geq 2$, so by~\eqref{sharp-i-A-ui-T-t-ne-0},
		\begin{equation}\label{G-u-T-D-ne-0-cup-t-N-t-1-1}
			G_u^T D \ne 
			\{ 0 \}\cup\{ t : N^u(t, 1) = 1 \} = \bigcup_{i = 1}^{N_u} K_{ui}.
		\end{equation}
		Thus $G_u^T D$ contains an element which does not belong to $K_{ui}$ for any 
		$1\leq i \leq N_u$. 
		This implies that 
		$$\dim (G_u^T D \cap K_{ui}) < \dim G_u^T D = \dim D = k + 1.$$
		Thus by the induction hypothesis, we know that for each $i = 1\, \cdots, N_{u}$,
		\begin{equation}\label{lim-n-inf-t-in-G-u-T-D-cap-K-i}
			\lim_{n \to \infty} \inf_{t \in ((G_u^T D) \cap K_{ui}) \setminus \{0  \}} N^u(t, n) \geq 2, \quad \mbox{a.s.}
		\end{equation}
		But using the equality in~\eqref{G-u-T-D-ne-0-cup-t-N-t-1-1} and using the inverse of~\eqref{sharp-i-A-ui-T-t-ne-0}, we know that
		\begin{equation}\label{inf-t-in-G-u-T-D-setminus-cup-i-K-i-setminus-0}
			\inf_{t \in ((G_u^T D) \setminus \bigcup_{i} K_{ui}) \setminus \{0  \}} N^u(t, 1) \geq 2.
		\end{equation}
		Combining~\eqref{lim-n-inf-t-in-G-u-T-D-cap-K-i} and~\eqref{inf-t-in-G-u-T-D-setminus-cup-i-K-i-setminus-0}, we get~\eqref{lim-n-inf-t-in-G-u-T-D-setminus-0-N-t-n-ge-2}. So 
		~\eqref{P-N-t-n-ge-2-t-in-D-setminus-0} still holds when $\dim D = k+ 1$. 
		Thus, by induction, \eqref{P-N-t-n-ge-2-t-in-D-setminus-0}  holds for any linear subspace $D$ of $\mathbb R^d$. 
		
		
		(3) We then prove that 
		\begin{equation}\label{lim-n-inf-t-R-d-setminus-0-N-t-n-ge-4}
			\lim_{n \to \infty} \inf_{t \in \mathbb{R}^d \setminus \{ 0 \}} N(t, n) \geq 4, \quad \mbox{a.s.}
		\end{equation}
		By~\eqref{lim-n-inf-t-R-d-setminus-0-N-t-n-ge-2}, we let $k_1$ be a random integer such that a.s.
		\begin{equation}
			\inf_{t \in \mathbb{R}^d \setminus \{0 \}} N(t, k_1) \geq 2.
		\end{equation}
		Again by~\eqref{lim-n-inf-t-R-d-setminus-0-N-t-n-ge-2}, we know that a.s. for each $u \in T_{k_1}$ we have
		\begin{equation}\label{lim-n-infty-inf-t-tilde-N-u-t-n-ge-2}
			\lim_{n \to \infty} \inf_{t \in \mathbb{R}^d \setminus \{ 0\}} N^u(G_u^Tt, n) \geq 2.
		\end{equation}
		Let $k_2$ be a (random) integer such that $\inf_{t \ne 0} N^u(G_u^Tt, k_2) \geq 2$ for all $u \in T_{k_1}$. Then, for all $t \ne 0$, there exist two distinct $u_1, u_2$ in $T_{k_1}$ such that $G_{u_1}^T t \ne 0 \ne G_{u_2}^T t$, hence furthermore there exist four distinct $u_{i_1i_2}$, such that $u_{i_1} \leq u_{i_1i_2}  $ and
		\begin{equation}
			G_{u_{i_1i_2}}^T t \ne 0, \quad i_1, i_2 = 1, 2.
		\end{equation}
		This shows that a.s.
		\begin{equation}
			\inf_{t \ne 0} N(t, k_1 + k_2) \geq 4.
		\end{equation}
		So~\eqref{lim-n-inf-t-R-d-setminus-0-N-t-n-ge-4} holds.
		
		(4) We can repeatedly argue as in part (3) to derive that for any integer $k \geq 1$, we have
		\begin{equation}
			\lim_{n \to \infty} \inf_{t \in \mathbb{R}^d \setminus \{ 0 \}} N(t, n) \geq 2^k, \quad \mbox{a.s.}
		\end{equation}
		Thus  \eqref{P-sharp-u-in-T-n-G-u-T-t-ne-0-ge-M} follows. 

	\end{proof}
	
	\begin{proof}[Proof of Theorem~\ref{thm::absolute-continuity}]
		(1) We first  prove that $|\phi(t)| = O(|t|^{-a})$ as $|t| \to \infty$ for some $a > 0$. 
		By Lemma \ref{lem::Fourier-decays-a-order}, 
		it suffies to verify the inequality~\eqref{sup-t-S-d-1-E-A-it-t--a}  for some $\delta_0 > 0$ and $a > 0$, that is, 
		\begin{equation}\label{2-sup-t-S-d-1-E-A-it-T-t--a}
			\sup_{t \in \mathbb{S}^{d-1}}  \mathbb{E}[|A_{i(t)}^T t|^{-a} \mathbbm{1}_{\{ N_{\delta_0}(t) \leq 1 \}}    ] < 1.
		\end{equation}
		
		From~\eqref{sup-t-S-d-1-E-A-it-T-t-epsilon-0} in Condition~\ref{cond::absolute-continuity}, we know that the set of random variables $\{ |A_{i(t)}^T t|^{-\varepsilon } :t \in \mathbb{S}^{d-1}, \varepsilon \in [0,\varepsilon_0 / 2]  \}$ is uniformly integrable, since  if we denote $X := |A_{i(t)}^T t|^{-1}$ for any $t \in \mathbb{S}^{d-1}$, then for any $M > 1$ and $\varepsilon \in (0,\varepsilon_0 / 2]$,
		\begin{align}
			\mathbb{E}[X^\varepsilon \mathbbm{1}_{\{ X^\varepsilon \geq M \}}] & = \mathbb{E}[X^{\varepsilon_0} X^{\varepsilon - \varepsilon_0} \mathbbm{1}_{\{ X \geq M^{1/\varepsilon} \}}] \notag \\
			& \leq  \mathbb{E}[X^{\varepsilon_0} M^{(\varepsilon - \varepsilon_0)/\varepsilon} \mathbbm{1}_{\{ X \geq M^{1/\varepsilon} \}}] \notag \\
			& \leq M^{1 - \frac{\varepsilon_0}{\varepsilon}} \mathbb{E}[X^{\varepsilon_0}] \notag \\
			& \leq M^{-1} \mathbb{E}[X^{\varepsilon_0}]. 
		\end{align}
		
		By Lemmas~\ref{lem::N-delta-t-has-larger-than-one-expectation} and~\ref{lem::N-equals-one-with-less-than-one-probability}, we know that there exist $\delta_1 > 0$ and $\eta > 0$ such that
		\begin{equation}\label{t-S-d-1-P-N-delta-t-le-1-le-1-eta}
			\sup_{t \in \mathbb{S}^{d-1}, \delta \in [0,\delta_1]} \mathbb{P}[N_\delta(t) \leq 1] \leq 1 - \eta.
		\end{equation}
		Then, for $a > 0$, $\delta \in [0, \delta_1]$ and $t \in \mathbb{S}^{d-1}$,
		\begin{align}
			& \sup_{t \in \mathbb{S}^{d-1}}  \mathbb{E}[|A_{i(t)}^T t|^{-a} \mathbbm{1}_{\{ N_{\delta}(t) \leq 1 \}} ] \notag \\
			\leq 
			& \sup_{t \in \mathbb{S}^{d-1}}  \mathbb{E}[(|A_{i(t)}^T t|^{-a}-1) \mathbbm{1}_{\{ N_{\delta}(t) \leq 1 \}} ] 
			+ 
			\sup_{t \in \mathbb{S}^{d-1}, \delta \in [0,\delta_1]} \mathbb{P}[N_\delta(t) \leq 1] \notag \\
			\leq & \sup_{t \in \mathbb{S}^{d-1}}  \mathbb{E}[(|A_{i(t)}^T t|^{-a}-1) \mathbbm{1}_{\{ N_{\delta}(t) \leq 1 \}} ]  + 1 - \eta.
		\end{align}
		Thus, in order to prove~\eqref{2-sup-t-S-d-1-E-A-it-T-t--a}, it suffices to prove that there exist $a > 0$ and $\delta_0 \in [0,\delta_1]$ such that
		\begin{equation}\label{sup-t-S-d-1-eta-2}
			\sup_{t \in \mathbb{S}^{d-1}}  \mathbb{E}[(|A_{i(t)}^T t|^{-a}-1) \mathbbm{1}_{\{ N_{\delta_0}(t) \leq 1 \}} ] \leq \frac{\eta}{2}.
		\end{equation}
		We will prove this in the following. By  Lemma~\ref{lem::N-delta-equals-zero-probability},
		\begin{equation}
			\lim_{\delta \to 0+} \sup_{t \in \mathbb{S}^{d-1}} \mathbb{P}[|A_{i(t)}^T t| \leq \delta] = 0.
		\end{equation}
		Then, using the uniform integrability of $\{ |A_{i(t)}^T t|^{-\varepsilon } :t \in \mathbb{S}^{d-1}, \varepsilon \in [0,\varepsilon / 2]  \}$, we obtain
		\begin{equation}
			\lim_{\delta \to 0+}\sup_{t \in \mathbb{S}^{d-1}, \varepsilon \in [0,\varepsilon_0 / 2]}  \mathbb{E}[|A_{i(t)}^T t|^{-\varepsilon} \mathbbm{1}_{\{ |A_{i(t)}^T t| \leq \delta  \}}] = 0.
		\end{equation}
		In particular, there exists $\delta_0 \in [0,\delta_1]$ such that
		\begin{align}
			\sup_{t \in \mathbb{S}^{d-1}, \varepsilon \in [0,\varepsilon_0 / 2]}  \mathbb{E}[|A_{i(t)}^T t|^{-\varepsilon} \mathbbm{1}_{\{ |A_{i(t)}^T t| \leq \delta_0  \}}] \leq \frac{\eta}{4}.
		\end{align}
		Choose $a > 0$ small enough such that $a < \frac{\varepsilon_0}{2}$ and $\delta_0^{-a} - 1 < \frac{\eta}{4}$. Then, from the above display, we have,  for each $t \in \mathbb{S}^{d-1}$, 
		\begin{align}
			& \mathbb{E}[(|A_{i(t)}^T t|^{-a} - 1) \mathbbm{1}_{\{ N_{\delta_0}(t) \leq 1 \}}] \notag \\
			\leq & \mathbb{E}[(|A_{i(t)}^T t|^{-a} - 1)_+] \notag \\
			\leq & \mathbb{E}[|A_{i(t)}^T t|^{-a} \mathbbm{1}_{\{ |A_{i(t)}^T t| \leq \delta_0 \}}] + \mathbb{E}[(|A_{i(t)}^T t|^{-a} - 1)_+ \mathbbm{1}_{\{ |A_{i(t)}^T t| > \delta_0 \}}] \notag \\
			\leq & \frac{\eta}{4} + (\delta_0^{-a} - 1) \leq \frac{\eta}{2}.
		\end{align}
		This gives ~\eqref{sup-t-S-d-1-eta-2}, which implies ~\eqref{2-sup-t-S-d-1-E-A-it-T-t--a} and then  $|\phi(t)| = O(|t|^{-a})$ as $|t| \to \infty$. 
		
		(2) We next prove the absolute continuity of $Z$. 
		For $n \geq 1$ and $\delta > 0$, consider  the event 
		\begin{equation}\label{F-n-number-u-in-T-n-G-u-T-t-ne-0-ge-d-a}
			F_{n,\delta} := \big\{ \omega \in \Omega:   \#\{ u \in T_n : |G_u^T t| \geq \delta|t|  \} \geq \big\lfloor \frac{d}{a} \big\rfloor + 1, \forall t \ne 0 \big\}.
		\end{equation}
		Let $\varepsilon > 0$ be arbitrary. By Lemma~\ref{lem::inf-N-t-can-be-large} and the compactness of the unit sphere  (as in the proof of Lemma~\ref{lem::N-delta-t-has-larger-than-one-expectation}), we know that there exist  $n$ and $\delta$ such that
		\begin{equation}
			\mathbb{P}[F_{n,\delta}] \geq 1-\varepsilon. 
		\end{equation}
		Let $(Z_u)_{u \in \mathbb{U}}$ be a family of i.i.d. random variables with the same law as $Z$, independent to the family of the vectors $(N_u,A_{u1},A_{u2}, \cdots)_{u \in \mathbb{U}}$. From the distributional equation~\eqref{equ::smoothing_transform}, we know that
		\begin{equation}
			Z \overset{\mathcal{L}}{=} \sum_{u \in T_n} G_u Z_u. 
		\end{equation}
		Thus, for any measurable  $E \subset \mathbb{R}^d$ with  Lebesgue measure $0$,
		\begin{align} 
			\label{eq-abs-Z} 
			\mathbb{P}[Z \in E] & = \mathbb{P}\big[ \sum_{u\in T_n} G_u Z_u \in E \big] \leq \varepsilon + \mathbb{P}\big[ F_{n,\delta} \cap \big\{ \sum_{u\in T_n} G_u Z_u \in E \big\} \big].
		\end{align}
		Let $\mu$  be the conditional distribution of   $Y:= \sum_{u \in T_u} G_uZ_u$
		given $F_{n,\delta}$, that is,
		$$\mu (B) = \mathbb P ( Y \in B | F_{n, \delta} )$$
		for each measurable set $B$ in $\mathbb R^d$. 
		Its 
		characteristic function of $\mu$
		satisfies, for all $t\in \mathbb R^d $,
		\begin{align*}
			\hat \mu (t) :=  \int e^{i\langle x,  t \rangle} d \mu (x) 
			=   \frac{1}{ \mathbb P [F_{n,\delta}]}  \mathbb E [  \mathbb E [ \mathbbm{1}_{ F_{n,\delta}}  e^{i 	\langle Y,  t \rangle  }  |  {\mathcal F}_n ] ]
			= \frac{1}{ \mathbb P [F_{n,\delta}]}    
			\mathbb E  \big[ \mathbbm{1}_{ F_{n,\delta}}   \prod_{u \in T_n} \phi (G_u^T t ) \big] .
		\end{align*}
		{ By Part (1) above, there exists $C > 0$ such that $|\phi(t)| \leq C (1 + |t|)^{-a}$ for all $t$. Then, 
			\begin{align*}
				|\hat{\mu} (t)|
				&\leq  \frac{1}{ \mathbb P [F_{n,\delta}]}    \mathbb E  \big[   \mathbbm{1}_{ F_{n,\delta}}  
				\prod_{u \in T_n,   |G_u^T t| \geq \delta |t| } 
				|\phi(G_u^T t)|
				\big]   
				\leq  
				C'(1+  | \delta  t |)^{ -a  (\lfloor \frac{d}{a} \rfloor + 1)   }   ,
			\end{align*} 
			where $C' = C^{(\lfloor \frac{d}{a} \rfloor + 1)}$ and we use the definition of $F_{n, \delta}$ in the last inequality.}
		Therefore $\hat \mu$  is integrable on $\mathbb R^d$, so that $\mu$ is absolutely continuous with respect to the Lebesgue measure on $\mathbb R^d$. This implies that 
		the second term in the right hand side of \eqref{eq-abs-Z}  is $0$. 
		%
		%
		%
		%
		
		Thus, from  \eqref{eq-abs-Z} we have $\mathbb{P}[Z \in E] \leq \varepsilon$. Since $\varepsilon > 0$ is arbitrary, we get $\mathbb{P}[Z \in E] = 0$, for any measurable set $E$ with Lebesgue measure $0$. This proves the absolute continuity of $Z$. 
	\end{proof}
	
	\section{Existence of harmonic moments}\label{section::negative-moments}

	To find the critical value for the existence of harmonic moments of the solution, we need the spectral gap theory for negative moments developed in~\cite{xiao2022edgeworth}. 
	
	We recall the definitions of $\tilde{I}_-$ and $\tilde{\kappa}$. If $\mathbb{P}[N = 1] > 0$, then we denote by $\tilde{A}_1$ the random matrix with the law of $A_1$ conditioned on $\{  N = 1 \}$. Define
	\begin{equation}
		\tilde{I}_- := \{ s \leq 0 : \mathbb{E}[\| \tilde{A}_1 \|^s] < +\infty \},
	\end{equation}
	and let $(\tilde{A}_i)_{i \geq 2}$ be a sequence of i.i.d. copies of $\tilde{A}_1$. From~\cite[Proposition 2.2]{xiao2022edgeworth}, under Condition~\ref{cond::Furstenberg-Kesten}, the following limit exists for $s \in \tilde{I}_-$: 
	\begin{equation}
		\tilde{\kappa}(s) := \lim_{n \to \infty} \mathbb{E}[\| \tilde{A}_n \cdots \tilde{A}_1 \|^s ]^{\frac{1}{n}}  \in (0, \infty).
	\end{equation}
	
	We introduce the transfer operator $P_s$ and the conjugate transfer operator $P^*_s$ for each $s \in\tilde{I}_-$ as follows: for any bounded measurable function $f$ on $\mathbb{S}_+^{d-1}$ and $v \in \mathbb{S}_+^{d-1}$, 
	\begin{equation}
		P_s f(v) := \mathbb{E}[|\tilde{A}_1 v|^s f(\tilde{A}_1 \cdot v)], \quad P_s^* f(v) := \mathbb{E}[|\tilde{A}_1^T v|^s f(\tilde{A}_1^T \cdot v)].
	\end{equation}
	Under Condition~\ref{cond::Furstenberg-Kesten}, the operator $P_s$ has a unique probability eigenmeasure $\nu_s$ and a unique (up to a scaling constant) strictly positive and continuous eigenfunction $r_s$, both corresponding to the eigenvalue $\tilde{\kappa}(s)$:
	\begin{equation} \label{def-rs-nus}
		P_s r_s = \tilde{\kappa}(s) r_s, \quad P_s \nu_s = \tilde{\kappa}(s)\nu_s,
	\end{equation}
	Similarly, the conjugate operator $P^*_s$ has a unique probability eigenmeasure $\nu^*_s$ and a unique (up to a scaling constant) strictly positive and continuous eigenfunction $r^*_s$, both corresponding to the eigenvalue $\tilde{\kappa}(s)$:
	\begin{equation} \label{eq-P*s}
		P^*_s r^*_s = \tilde{\kappa}(s) r^*_s, \quad P^*_s \nu^*_s = \tilde{\kappa}(s)\nu^*_s. 
	\end{equation}
	They satisfy
	\begin{align}
		r_s(v) = \int_{\mathbb{S}^{d-1}_+} \langle u, v\rangle^s \nu^*_s(du), \quad 
		r^*_s(v) = \int_{\mathbb{S}^{d-1}_+} \langle u, v\rangle^s \nu_s(du), \quad v \in \mathbb{S}^{d-1}_+ . \label{r-s-v-int-langle-u-v-lrangle-s-nu-s}
	\end{align}
	
	We now come to the proof of  Theorem~\ref{thm::critical-exponent-negative-moment}.  
	To this end we will  need the following technical  lemma.
	
	\begin{lemma}\label{lem::modulus-matrix}
		Suppose that $a$ is a deterministic positive $d \times d$ matrix that satisfies the Furstenberg-Kesten condition: there exists $c > 1$ such that
		\begin{equation}
			0< \max_{1\leq i, j \leq d} a(i, j)\leq c \min_{1\leq i, j \leq d} a(i, j),
		\end{equation}
		where $a(i, j)$ is the $(i,j)$-th entry of $a$. Then, for any $x, y \in \mathbb{S}_+^{d-1}$, we have
		\begin{equation}
			\frac{\| a \|}{cd} \leq \langle ax, y \rangle \leq \| a \|.
		\end{equation}
	\end{lemma}	
	\begin{proof}
		On the one hand, $\langle ax, y\rangle \leq |ax| |y| \leq \|a\| |x| |y| = \|a\|$ for any $x, y \in \mathbb{S}^{d-1}_+$. On the other hand, for any $x, y \in \mathbb{S}^{d-1}_+$, 
		\begin{align}
			\langle ax, y\rangle \geq \min_{1 \leq i, j \leq d} a(i, j) \geq \frac{\max_{1 \leq i, j \leq d} a(i, j)}{c} \geq \frac{\|a\|}{cd},
		\end{align}
		where in the last inequality we use the fact $\|a\| = \max_{1 \leq j\leq d} \sum_{i = 1}^d |a(i, j)|$. 
	\end{proof}

	
	\begin{proof}[Proof of Theorem~\ref{thm::critical-exponent-negative-moment}]
		(1) We first prove that   if {$\mathbb{P}[N = 1] > 0$}, then (a) $\Rightarrow$ (b).  
		Assume that $\mathbb{E}[|Z|^{-a}] < +\infty$, and let $V^* \in \mathbb{S}^{d-1}_+$ be a random variable with law $\nu^*_{-a}$ independent to $Z$. By~\eqref{r-s-v-int-langle-u-v-lrangle-s-nu-s}, we have
		\begin{align}
			\mathbb{E}[\langle V^*, Z\rangle^{-a} | Z] = |Z|^{-a}   r_{-a}\big( \frac{Z}{|Z|}\big) \leq c |Z|^{-a},
		\end{align}
		where 
		$c =  \max_{x \in \mathbb S_+^{d-1} } r_{-a} (x) $.
		Therefore, we have $\mathbb{E}[\langle V^*, Z \rangle^{-a}] \leq c \mathbb{E}[| Z |^{-a}] < +\infty$. Suppose that $(N, A_1, A_2,\cdots)$ is independent of $(Z, V^*)$. Then, by~\eqref{equ::smoothing_transform},
		\begin{align}\label{E-theta-Z-a-ge-P-N-1-E-theta-A-1-Z-a}
			\mathbb{E}[\langle V^*, Z\rangle^{-a}  ] > \mathbb{E}[\mathbbm{1}_{\{ N = 1 \}} \langle V^*, A_1 Z \rangle^{-a}] = \mathbb{P}[N = 1] \mathbb{E}[\langle V^*, \tilde{A}_1 Z \rangle^{-a}].
		\end{align}		
		where the strict inequality follows from the fact that $\mathbb{P}[N \geq 2] > 0$. Using the equation $P^*_{-a}\nu_{-a}^* = \tilde{\kappa}(-a) \nu_{-a}^*$ (see  \eqref{eq-P*s}), we have
		\begin{equation}
			\mathbb{E}[\langle \tilde{A}_1^T V^*, Z \rangle^{-a} | Z] = \mathbb{E}[|\tilde{A}_1^T V^*|^{-a} \langle \tilde{A}_1^T \cdot V^*, Z \rangle^{-a} | Z] = \tilde{\kappa}(-a) \mathbb{E}[\langle V^*, Z\rangle^{-a} | Z].
		\end{equation}
		Taking expectation on both sides and plugging in to~\eqref{E-theta-Z-a-ge-P-N-1-E-theta-A-1-Z-a}, we get
		\begin{equation}
			\mathbb{E}[\langle V^*, Z \rangle^{-a}] >  \tilde{\kappa}(-a)\mathbb{P}[N = 1] \mathbb{E}[\langle V^*, Z \rangle^{-a}].
		\end{equation}
		Dividing both sides by $\mathbb{E}[\langle V^*, Z \rangle^{-a}]$, we derive that $ \tilde{\kappa}(-a)\mathbb{P}[N = 1] < 1$.
	
	\medskip 
		(2) we next prove that  if $\mathbb{P}[N = 1] > 0$, then 
		(b) $\Rightarrow$ (c). 
		Assume that $\mathbb{P}[N = 1]\tilde{\kappa}(-a) < 1$. Denote the Laplace transform of $Z$ by 
		\begin{equation}
			\psi(t) := \mathbb{E}[e^{-\langle t, Z \rangle}], \quad t \in \mathbb{R}^d_+.
		\end{equation}
		By~\eqref{equ::smoothing_transform}, the Laplace transform $\psi$ satisfies the functional equation
		\begin{equation}\label{psi-functional-equation}
			\psi(t) = \mathbb{E}\big[ \prod_{i=1}^N \psi(A_i^T t) \big], \quad t \in \mathbb{R}^d_+.
		\end{equation}
		To prove (c), it suffices to prove $\psi(t) = O(|t|^{-a})$ as $t \to \infty$ (see~\cite[Lemma 4.4]{liu1999asymptotic}).
		
		From~\eqref{equ::smoothing_transform} and Condition~\ref{cond::Furstenberg-Kesten}, we know that 
		$Z \in \mathrm{int}(\mathbb{R}_+^d)$ a.s. 
		Hence
		\begin{equation}
			\lim_{s \to \infty}  \sup_{|t| \geq s, t \in \mathbb{R}_+^d} \psi(t) = 0.
		\end{equation}
		For any $\varepsilon > 0$,  let $t_\varepsilon > 0$ be such that
		\begin{equation}
			\sup_{|t| \geq t_\varepsilon, t \in \mathbb{R}_+^d} \psi(t) \leq \varepsilon. 
		\end{equation}
		For any $\delta > 0$, define 
		\begin{equation}
			N_\delta := \# \{ 1\leq i \leq N: \iota(A_i^T) \geq \delta \},
		\end{equation}
		where $\iota(a) = \inf_{x \in \mathbb{S}_{+}^{d-1}} |ax|$. Then, we have a.s. $\lim_{\delta \to 0} N_\delta = N$. For $t \in \mathbb{R}^d_+$ with $|t| \geq t_\varepsilon / \delta$, we have
		\begin{equation}\label{psi-t-le-E-psi-A-1-T-t-epsilon-N-delta-1}
			\psi(t) \leq \mathbb{E}[\psi(A_1^T t) (\varepsilon^{N_\delta - 1} \mathbbm{1}_{\{ N_\delta \geq 1 \}} + \mathbbm{1}_{\{ N_\delta = 0 \}})] =: p_{\varepsilon, \delta}\mathbb{E}[\psi((A_1')^T t)],
		\end{equation}
		where $$p_{\varepsilon, \delta} = \mathbb{E}[\varepsilon^{N_\delta - 1} \mathbbm{1}_{\{ N_\delta \geq 1 \}} + \mathbbm{1}_{\{ N_\delta = 0 \}}],$$ and $A_1'$ is a random matrix whose law is determined by
		\begin{equation}\label{E-f-M-frac-1-p-epsilon-delta-E-f-A-1}
			\mathbb{E}[f(A_1')] = \frac{1}{p_{\varepsilon, \delta}} \mathbb{E}[f(A_1)(\varepsilon^{N_\delta - 1} \mathbbm{1}_{\{ N_\delta \geq 1 \}} + \mathbbm{1}_{\{ N_\delta = 0 \}})]
		\end{equation}
		for any bounded measurable function $f$. Let $V \in \mathbb{S}_+^{d-1}$ be a random variable with law $\nu_{-a}$,  and independent of $(N, A_1, A_2,\cdots)$ and $A_1'$. 
		
		We claim that there exist $\varepsilon, \delta, q \in (0,1)$ such that
		\begin{equation}\label{p-defined-p-epsilon-delta-le-1-p-E}
			p_{\varepsilon, \delta} < 1, \quad p_{\varepsilon, \delta}\mathbb{E}[\langle A_1' V, x\rangle^{-a} ] \leq q \mathbb{E}[\langle V, x\rangle^{-a}], \quad \forall x \in \mathbb{S}_+^{d-1}.
		\end{equation}
		By the dominated convergence theorem,
		\begin{equation}\label{lim-epsilon-0-delta-0-P-N=1-le-1}
			\lim_{\varepsilon \to 0+} \lim_{\delta \to 0+} p_{\varepsilon, \delta} = \mathbb{P}[N = 1] < 1.
		\end{equation}
		Note that by Lemma~\ref{lem::modulus-matrix}, for any $x \in \mathbb{S}_+^{d-1}$,
		$\langle A_1 V, x \rangle^{-a} \leq  (cd)^{a}  \| A_1\|^{-a}$, so that 
		\begin{align}
			p_{\varepsilon, \delta} \mathbb{E}[\langle A_1' V, x\rangle^{-a}] & = \mathbb{E}[\langle A_1 V, x \rangle^{-a} (\varepsilon^{N_\delta - 1} \mathbbm{1}_{\{ N_\delta \geq 1 \}} + \mathbbm{1}_{\{ N_\delta = 0 \}})] \notag \\
			& \leq (cd)^{a} \big(\varepsilon \mathbb{E}[\| A_1\|^{-a}] + \mathbb{E}[\| A_1\|^{-a} \mathbbm{1}_{\{ N_\delta = 0 \}}] + \mathbb{E}[\| A_1\|^{-a} |\mathbbm{1}_{\{ N_\delta = 1 \}} - \mathbbm{1}_{\{ N = 1 \}}|]  \big) \notag \\
			& \quad + \mathbb{E}[\langle A_1 V, x \rangle^{-a}\mathbbm{1}_{\{ N = 1 \}}]. \label{p-epsilon-delta-E-M-theta-x-a-le}
		\end{align}
		For any $\eta > 0$, since $\mathbb{E}[\| A_1\|^{-a}] < \infty$ and $\mathbb P [N=0]=0$, we can choose $\varepsilon>0$ and $\delta>0$ small enough, such that the first term on the right hand side of~\eqref{p-epsilon-delta-E-M-theta-x-a-le} is smaller than $\eta$. It follows that, by the definition of  the law $\nu_{-a}$  of $V$ (see \eqref{def-rs-nus}), for  $\varepsilon, \delta>0$ small enough and all $x \in \mathbb{S}_{+}^{d-1}$,
		\begin{align}
			p_{\varepsilon, \delta} \mathbb{E}[\langle A_1' V, x\rangle^{-a}] & \leq \eta + \mathbb{E}[\langle A_1 V, x \rangle^{-a}\mathbbm{1}_{\{ N = 1 \}}] \notag \\
			& = \eta + \mathbb{P}[N = 1] \tilde{\kappa}(-a) \mathbb{E}[\langle V, x\rangle^{-a}]\notag \\
			& \leq (\eta + \mathbb{P}[N = 1] \tilde{\kappa}(-a)) \mathbb{E}[\langle V, x\rangle^{-a}],\label{p-epsilon-delta-E-M-theta-x-a-le-eta}
		\end{align}
		where in the last inequality, we use the fact that $\langle V, x\rangle \leq 1$. Since $\mathbb{P}[N = 1]\tilde{\kappa}(-a) < 1$, we can choose $\eta \in (0,1)$ small enough such that 
		\begin{equation}
			q := \eta + \mathbb{P}[N = 1]\tilde{\kappa}(-a) < 1.
		\end{equation}
		This together with~\eqref{lim-epsilon-0-delta-0-P-N=1-le-1} implies that~\eqref{p-defined-p-epsilon-delta-le-1-p-E} holds for $\varepsilon, \delta \in (0,1)$ small enough.
		
		We  now fix $\varepsilon$ and $\delta$  in $  (0,1) $ such that~\eqref{p-defined-p-epsilon-delta-le-1-p-E} holds. Then, by~\eqref{psi-t-le-E-psi-A-1-T-t-epsilon-N-delta-1} and the fact that $\langle V, t \rangle^{-a} \geq |t|^{-a}$ for $t \in \mathbb{R}^d_+$, we know there exists $C > 0$ such that
		\begin{equation}
			\psi(t) \leq  p_{\varepsilon, \delta}  \mathbb{E}[\psi((A_1')^T t)] + C \mathbb{E}[\langle V, t\rangle^{-a}] , \quad \forall t \in \mathbb{R}^d_+.
		\end{equation}
		By iterating this inequality and using the condition that $ p_{\varepsilon, \delta}  < 1$, we get
		\begin{equation}
			\psi(t) \leq \sum_{n \geq 0} C  p_{\varepsilon, \delta}^n \mathbb{E}[\langle A_n'\cdots A_1' V, t\rangle^{-a}], 
		\end{equation} 
		where $(A_i')_{i \geq 2}$ are independent copies of  $A_1'$, which are also independent of $V$, with the convention that  $A_n'\cdots A_1' $ stands for the  identity matrix when $n=0$,.  Therefore, by iterating   \eqref{p-defined-p-epsilon-delta-le-1-p-E} and  using $r_{-a}^*(v) = \mathbb{E}[\langle V, v\rangle^{-a}]$ for $v\in \mathbb{S}_+^{d-1}$ (see \eqref{r-s-v-int-langle-u-v-lrangle-s-nu-s}), we have
		\begin{equation}
			\psi(t) \leq \sum_{n \geq 0} Cq^n \mathbb{E}[\langle V, t\rangle^{-a}] = \frac{Cr^*_{-a}(\frac{t}{|t|})|t|^{-a}}{1-q}, \quad t \in \mathbb{R}_+^d.
		\end{equation}
		Since $r_{-a}^*$ is a bounded function, we get that $\psi(t) = O(|t|^{-a})$ as $|t| \to \infty$. 
		
			
			(3) we then prove that  if  $\mathbb{P}[N = 1] = 0$, then (c) holds. 
			The proof is similar to Part (2) above. 
			In this case, corresponding 
			to~\eqref{p-epsilon-delta-E-M-theta-x-a-le-eta}, we have,  
			for any $\eta > 0$, there exists $\varepsilon$, $\delta$ and $p = p_{\varepsilon, \delta} < 1$ such that
			\begin{equation}
				p\mathbb{E}[\langle A_1' V, x\rangle^{-a}] \leq \eta \mathbb{E}[\langle V, x\rangle^{-a}], \quad x \in \mathbb{S}_+^{d-1},
			\end{equation}
			where $A_1'$ is defined as in~\eqref{E-f-M-frac-1-p-epsilon-delta-E-f-A-1}. The rest of the proof is the same as above. 
			
			(4) we finally prove that  if $\mathbb{P}[N = 1] > 0$ and  $\mathbb{E}[\|A_1\|^{-a-\varepsilon}] < +\infty$ for some $\varepsilon > 0$, then 
			(b) $\Rightarrow$ (a).  Since $-a-\varepsilon \in \tilde{I}_-$ and $\tilde{\kappa}$ is continuous on $\tilde{I}_-$, there exists $a' \in (a, a+\varepsilon)$ such that
			\begin{equation}
				\tilde{\kappa}(-a')\mathbb{P}[N = 1] < 1. 
			\end{equation}
			So from the implication   (b) $\Rightarrow$ (c) proved in Part (2), we know that $\mathbb{E}[|Z|^{-b}] < \infty$ for all $b \in (0, a')$. In particular, $\mathbb{E}[|Z|^{-a}] < +\infty$. 
		\end{proof}
		
		%
		%

		\section{Appendix}\label{sec::appendix}
		
		\begin{proof}[Proof of Lemma~\ref{lem::support_breaking_general}]
			We only prove the case where $n = 2$, since  the proof for the general case is similar. 
			For convenience denote $\Omega_0 = \Omega$. For $i = 0,1,2$, we denote by $d_i$ the metric on $\Omega_i$. For $x \in \Omega_i$ and $\varepsilon > 0$, define the open balls
			\begin{equation}
				B_i(x,\varepsilon) = \left\{ v \in \Omega_i : d_i(v,x) < \varepsilon \right\}, \quad i =0, 1, 2.
			\end{equation}
			
			(1)  We first  prove \eqref{equ::support_breaking_general_1} for $n=2$.
			Let $z \in \mathrm{supp}(f(X_1, X_2))$ be arbitrary. Since $f$ is continuous, we know that for each $\varepsilon > 0$, the inverse image $f^{-1}(B_0(z,\varepsilon))$ is also open. Since the Polish space is second-countable,  there exists a sequence of pairs of open sets $\{ U_i^{1,\varepsilon} \times U_i^{2,\varepsilon} \}_{ i\in\mathbb{N}}$ in $\Omega_1 \times \Omega_2$ such that 
			\begin{equation}
				f^{-1}(B_0(z,\varepsilon)) = \bigcup_{i \in \mathbb{N}} U_i^{1,\varepsilon} \times U_i^{2,\varepsilon}.
			\end{equation}
			Since $z \in \mathrm{supp}(f(X_1, X_2))$, 
			\begin{equation}
				0 < \mathbb{P}[f(X_1, X_2)\in B(z,\varepsilon)] 
				= \mathbb{P}\left[ (X_1, X_2) \in \bigcup_{i \in \mathbb{N}}  U_i^{1,\varepsilon} \times U_i^{2,\varepsilon}  \right].
			\end{equation}
			Therefore there exists some $i_0 \in \mathbb N$ such that, with 
			$(V_1, V_2) := (U_{i_0}^{1,\varepsilon},U_{i_0}^{2,\varepsilon})$, 
			\begin{equation}
				\mathbb{P}[(X_1, X_2) \in V_1 \times  V_2] > 0.
			\end{equation}
			In particular,  we have
			\begin{equation}
				\mathbb{P}[X_1 \in V_1] > 0.
			\end{equation}
			We claim that $\mathrm{supp}(X_1) \cap V_1 \ne \emptyset$. Otherwise,  for all $x \in V_1$, there exists a neighborhood $N_x$ of $x$ such that $N_x \subset V_1$ and $\mathbb{P}[X_1 \in N_x] = 0$. Since $\Omega_1$ is second-countable, 
			the open set $ V_1 =  \bigcup_{x \in V_1} N_{x} $ can be written as a countable union 
			\begin{equation}\label{equ::ad_hoc_5}
				V_1 = \bigcup_{i \in \mathbb{N}} N_{x_i},  \quad x_i \in V_1. 
			\end{equation}
			Since   $\mathbb{P}[X_1 \in  N_{x_i} ] = 0$ for all  $ i \in \mathbb{N}$, 
			we get that $\mathbb{P}[X_1 \in V_1] = 0$. This contradiction proves our claim. 
			
			Similarly, we have 
			$\mathrm{supp}(X_2) \cap V_2 \ne \emptyset$. 
			Therefore we can take 
			$x_\varepsilon \in \mathrm{supp}(X_1) \cap V_1$ and $y_\varepsilon \in \mathrm{supp}(X_2) \cap V_2$. Since $f^{-1}(B_0(z,\varepsilon)) \supset V_1 \times V_2$, we have $f(x_\varepsilon,y_\varepsilon) \in B_0(z,\varepsilon)$.  Since  $\varepsilon >0$ is arbitrary and  $f(x_\varepsilon,y_\varepsilon) \in  
			f(\mathrm{supp}(X_1), \mathrm{supp}(X_2))$, it follows that  
			\begin{equation}
				z \in \overline{f(\mathrm{supp}(X_1), \mathrm{supp}(X_2))}.
			\end{equation}
			This proves \eqref{equ::support_breaking_general_1}. 
			
			Now suppose that \eqref{equ::ad_hoc_6} holds. Then 
			\begin{equation}
				O \cap \mathrm{supp}(f(X_1, X_2)) \ne \emptyset.
			\end{equation}
			By \eqref{equ::support_breaking_general_1}, this implies that 
			\begin{equation}
				O \cap \overline{f(\mathrm{supp}(X_1), \mathrm{supp}(X_2))} \ne \emptyset. 
			\end{equation}
			Since $O$ is open, it follows that 
			\begin{equation}
				O \cap {f(\mathrm{supp}(X_1), \mathrm{supp}(X_2))} \ne \emptyset. 
			\end{equation}
			This means that there exists $(x_1, x_2) \in \prod_{i = 1}^2 \mathrm{supp}(X_i)$ such that $f(x_1, x_2) \in O$.
			
			(2) We next prove \eqref{equ::support_breaking_general_2} for $n=2$. 
			For $z \in f(\mathrm{supp}(X_1), \mathrm{supp}(X_2))$, there exist $x \in \mathrm{supp}(X_1)$ and $y \in \mathrm{supp}(X_2)$ such that $z = f(x,y)$. Since $f$ is continuous, for all $\varepsilon > 0$, there exist $\varepsilon_1 > 0$ and $\varepsilon_2 > 0$ such that 
			\begin{equation}
				B_0(z, \varepsilon) \supset f(B_1(x, \varepsilon_1), B_2(y, \varepsilon_2)).
			\end{equation} 
			Then, 
			\begin{align}
				\mathbb{P}[f(X_1,X_2) \in B_0(z,\varepsilon)] & \geq \mathbb{P}[X_1 \in B_1(x,\varepsilon_1), X_2 \in B_2(y,\varepsilon_2)] \notag \\
				& = \mathbb{P}[X_1 \in B_1(x,\varepsilon_1)] \mathbb{P}[X_2\in B_2(y,\varepsilon_2)] \notag \\
				& > 0.
			\end{align}
			Since $\varepsilon > 0$ is arbitrary, we know that $z \in \mathrm{supp}(f(X_1,X_2))$. Thus, we have
			\begin{equation}\label{equ::ad_hoc_4}
				\mathrm{supp}(f(X_1, X_2)) \supset f(\mathrm{supp}(X_1), \mathrm{supp}(X_2)).
			\end{equation}
			By taking closure on both sides of~\eqref{equ::ad_hoc_4},  and combining with the conclusion proved in Part (1) above, 
			we get   \eqref{equ::support_breaking_general_2}. 
			
			(3)  We finally prove  \eqref{supp-Z1Z2}. { Without loss of generality, assume that $X_i$ is $\Omega$-valued for each $i \geq 1$.}
			We equip 
			the product space $ \Omega^{ \mathbb N^*} $ with the metric 
			\begin{align}
				\tilde{d}(z, w) := \sum_{i\geq 1} \frac{d(z_i, w_i)}{2^i(1 + d(z_i,w_i))}, \quad \mbox{ for }  z = (z_i)_{i\geq1}, w = (w_i)_{i\geq1} \in  \Omega^{ \mathbb N^*},
			\end{align}
			where $d$ is the metric on $\Omega$. 
			
			We claim that any product of closed sets in $\Omega$ are also closed in $ \Omega^{ \mathbb N^*} $. Indeed, Suppose that $(O_i)_{i\geq1}$ is a sequence of open sets in $\Omega$ and let $C_i = \Omega \setminus O_i$. Then, we have
			\begin{align}
				\prod_{i\geq1} C_i 
				=   \bigcap_{i\geq1}  \bigg\{ x \in \Omega^{\mathbb N^*} : x_i \in C_i \bigg\}  
				= \left( \bigcup_{i\geq1} \left( \prod_{j < i} \Omega_j \times O_i \times \prod_{j > i} \Omega_j \right) \right)^c.
			\end{align}
			Since the countable union of open sets is open, this proves the claim. 
			
			Using this claim and the metric $\tilde{d}$, we can argue as in the proof of Parts (1) and (2) to derive that
			\begin{align}
				\mathrm{supp}((X_i)_{i\geq1}) = \overline{\prod_{i \geq 1}\mathrm{supp}(X_i)} = \prod_{i \geq 1}\mathrm{supp}(X_i). 
			\end{align}
		\end{proof}
		




\bigskip 
\noindent
{\textbf {Acknowledgments.}}
The work has been supported 
by  
the ANR project ‘‘Rawabranch’’ number ANR-23-CE40-0008,  the France 2030 framework program, Centre Henri Lebesgue ANR-11-LABX-0020-01, and the 
Tsinghua Scholarship for Overseas Graduate Studies. 


  \bibliographystyle{elsarticle-num-names-alphsort} 
  \bibliography{ref}


%
%
%
\end{document}